\documentclass[10pt,onefignum,onetabnum]{siamonline190516}

\usepackage[tracking=false,kerning=true,spacing=true]{microtype}

\usepackage{xcolor}

\usepackage{amssymb,mathtools,bbm,dsfont,stmaryrd,aligned-overset,physics}
\allowdisplaybreaks

\usepackage{comment}

\usepackage{tikz,pgfplots}
\pgfplotsset{compat=1.16}
  
\usepackage[acronym]{glossaries}
\glsdisablehyper

\usepackage{booktabs,multirow,array}

\usepackage{etoolbox}
\newbool{booltrue} 
\booltrue{booltrue}
\newbool{boolfalse} 
\boolfalse{boolfalse}

\newcommand{\qedhere}{}

\newcommand{\naturals}                       {\mathbb{N}}
\newcommand{\reals}                          {\mathbb{R}}
\newcommand{\realsBar}                       {\bar{\mathbb{R}}}
\newcommand{\nonnegativeReals}               {\mathbb{R}_{\geq 0}}

\newcommand{\cardinality}[1]                 {|#1|}

\DeclareMathOperator{\Id}                    {Id}

\DeclareMathOperator{\proj}                  {proj}

\newcommand{\effectiveDomain}[1]             {D(#1)}
\newcommand{\onotation}[2][]                 {\ifstrequal{#1}{adapt}{o\left(#2\right)}{o(#2)}}

\newcommand{\conv}                           {\ast}

\DeclareMathOperator*{\argmax}               {argmax}

\newcommand{\negativePart}[1]                {#1^-}

\newcommand{\innerProduct}[2]                {\left\langle #1,#2\right\rangle}

\newcommand{\eye}[1]                         {I_{#1}}
\DeclareMathOperator{\spec}                  {spec}
\newcommand{\transpose}[1]                   {{#1}^\top}
\newcommand{\inv}[1]                         {#1^{-1}}

\let\var\relax
\DeclareMathOperator{\var}{Var}
\DeclareMathOperator{\std}{std}
\DeclareMathOperator{\kl}{KL}
\newcommand{\expectedValue}[2]               {\mathbb{E}^{#1}\left[#2\right]}
\newcommand{\variance}[2]                    {\var^{#1}\left[#2\right]}
\newcommand{\standardDeviation}[2]           {\std^{#1}\left[#2\right]}
\newcommand{\kullbackLeibler}[2]             {\kl(#1|#2)}
\newcommand{\productMeasure}                 {\times}
\newcommand{\diracMeasure}[1]                {\delta_{#1}}
\newcommand{\radonNikodyn}[2]                {\frac{\d #1}{\d #2}}
\newcommand{\support}[1]                     {\mathrm{supp}(#1)}

\newcommand{\spaceProbabilityMeasures}[1]    {\mathcal{P}(#1)}
\newcommand{\Pp}[2]                          {\mathcal{P}_{#1}(#2)}
\newcommand{\Ppabs}[2]                       {\mathcal{P}_{#1,\mathrm{abs}}(#2)}

\newcommand{\wassersteinDistance}[3]         {W_{#1}(#2,#3)}
\newcommand{\optimalTransportDiscrepancy}[3] {W_{c}(#2,#3)}
\newcommand{\wassersteinBall}[3]             {B_{#2}(#3)}
\newcommand{\closedWassersteinBall}[3]       {\bar B_{#2}(#3)}
\newcommand{\optMap}[2]                      {T_{#1}^{#2}}
\makeatletter
\newcommand{\setPlans}[2]                    {\Gamma(#1,#2\checknextargPlans} % variable number of inputs
\newcommand{\setOptimalPlans}[2]             {\Gamma_o(#1,#2\checknextargPlans} % variable number of inputs
\newcommand{\checknextargPlans}              {\@ifnextchar\bgroup{\gobblenextargPlans}{)}}
\newcommand{\gobblenextargPlans}[1]          {,#1\@ifnextchar\bgroup{\gobblenextargPlans}{)}}

\newcommand{\pushforward}[1]                 {#1_{\#}}
\newcommand{\lebesgueMeasure}[1]             {\mathcal{L}^{#1}}

\renewcommand{\d}[0]                         {\mathrm{d}}

\newcommand{\diff}[2]                        {\frac{\d #1}{\d #2}}

\renewcommand{\gradient}[1]                  {\nabla_{#1}}

\newcommand{\subdifferential}[1]             {\partial^{-}#1}
\newcommand{\superdifferential}[1]           {\partial^{+}#1}
\newcommand{\tangentSpace}[2]                {\mathrm{Tan}_{#1}#2}

\newcommand{\weakconvergence}[0]             {\rightharpoonup}

\newcommand{\Lp}[2]                          {L^{#1}\ifstrempty{#2}{}{(#2)}}

\newcommand{\LpNorm}[3]                      {\left\Vert #1\right\Vert_{L^{#2}(#3)}}

\newcommand{\Ccinf}[1]                       {C_c^\infty(#1)}
\newcommand{\Cb}[1]                          {C_b\ifstrempty{#1}{}{(#1)}}
\newcommand{\C}[2]                           {C^{#1}\ifstrempty{#2}{}{(#2)}}

\newcommand{\Wrloc}[2]                       {W^{1,#1}_\mathrm{loc}\ifstrempty{#2}{}{(#2)}}

\newcommand{\compSymbol}                     {\fatsemi}
\makeatletter
\newcommand{\checknextargComp}               {\@ifnextchar\bgroup{\gobblenextargComp}{}}
\newcommand{\gobblenextargComp}[1]           {\compSymbol #1\@ifnextchar\bgroup{\gobblenextargComp}{}}
\newacronym{acr:amod}{AMoD}{autonomous mobility-on-demand}
\newacronym{acr:av}{AV}{autonomous vehicle}
\newacronym{acr:dro}{DRO}{distributionally robust optimization}
\newacronym{acr:kl}{KL}{Kullback-Leibler}
\newacronym{acr:mmd}{MMD}{Maximum Mean Discrepancy}
\newbool{showproofinverseoptimaltransportmap}

\newcommand{\proofinverseoptimaltransportmap}[1]{
\ifbool{#1}{\begin{proof}}{\begin{proof}[Proof of~\cref{prop:inverse optimal transport map}]}
By symmetry of the Wasserstein distance, it suffices to show
\begin{equation*}
    \wassersteinDistance{2}{\mu}{\nu}^2=\int_{\reals^d}\norm{(\optMap{\mu}{\nu})^{-1}(y)-y}^2\d\nu(y).
\end{equation*}
Indeed, 
\begin{equation*}
\begin{aligned}[b]
    \wassersteinDistance{2}{\mu}{\nu}^2
    &=
    \int_{\reals^d}\norm{x-\optMap{\mu}{\nu}(x)}^2\d\mu(x) \\
    &=
    \int_{\reals^d}\norm{(\optMap{\mu}{\nu})^{-1}(\optMap{\mu}{\nu}(x))-\optMap{\mu}{\nu}(x)}^2\d\mu(x) \\
    &=
    \int_{\reals^d}\norm{(\optMap{\mu}{\nu})^{-1}(y)-y}^2\d(\pushforward{(\optMap{\mu}{\nu})}\mu)(y) \\
    &=
    \int_{\reals^d}\norm{(\optMap{\mu}{\nu})^{-1}(y)-y}^2\d\nu(y).
\end{aligned}
\end{equation*}
This concludes the proof. 
\end{proof}}

\newbool{showproofperturbationtransportmap}

\newcommand{\proofperturbationtransportmap}[1]{
\ifbool{#1}{\begin{proof}}{\begin{proof}[Proof of \cref{lemma:perturbation transport map}]}
Our proof follows the proof of~\cite[Corollary 1.1]{bonnet2019optimal}.
By~\cref{prop:brenier2}, it suffices to establish that $\Id+s\gradient{}\psi$ is the gradient of a convex function. Indeed, $\Id+s\gradient{}\psi$ is smooth, and so differentiable $\mu$-almost everywhere, and 
\begin{equation}\label{eq:proof perturbation transport map}
\begin{aligned}
    \int_{\reals^d}\norm{x+s\gradient{}\psi(x)}^2\d\mu(x)
    &\leq 
    \int_{\reals^d}(\norm{x}+|s|\norm{\gradient{}\psi(x)})^2\d\mu(x)
    \\
    &\leq 2\int_{\reals^d}\norm{x}^2+|s|^2\norm{\gradient{}\psi(x)}^2\d\mu(x)
    \\
    &\leq 2\int_{\reals^d}\norm{x}^2\d\mu(x)+|s|^2\max_{x\in\reals^d}\norm{\gradient{}\psi(x)}^2
    \\
    &< +\infty,
\end{aligned}
\end{equation}
where the maximum of $\norm{\gradient{}\psi(x)}^2$ exists and is bounded by (i) continuity of the gradient and the norm and (ii) due to compactness of the support of $\psi$. So, by \cref{prop:brenier2}, if $\Id+s\gradient{}\psi$ is the gradient of a convex function, then it is an optimal transport map between $\mu$ and $\pushforward{(\Id+s\gradient{}\psi)}\mu$.
To prove convexity, by smoothness, it suffices to show that $I+s\nabla^2\psi(x)$ is positive definite for all $x\in\reals^d$.
Without loss of generality, assume $\nabla^2\psi(x)$ is non-zero at least for some $x\in\reals^d$; else, the statement is trivial.
Let $\bar s\coloneqq 1/\max_{x\in\reals^d}\norm{\nabla^2\psi(x)}>0$, where the maximum exists by compactness of the support of $\psi$ and continuity  of the norm. Then, $I+s\nabla^2\psi(x)$ is positive definite for all $s\in(-\bar s,\bar s)$, which implies $\Id+s\gradient{}\psi$ is the gradient of a convex function, and thus the statement follows. Finally, \eqref{eq:proof perturbation transport map} shows that $\pushforward{(\Id+s\gradient{}\psi)}\mu\in\Pp{2}{\reals^d}$ and, by definition of Wasserstein distance and optimal transport map, 
\begin{equation*}
    \wassersteinDistance{2}{\mu}{\pushforward{(\Id+s\gradient{}\psi)}\mu}
    =
    \left(\int_{\reals^d}\norm{x-(x-s\gradient{}\psi)}^2\d\mu(x)\right)^{\frac{1}{2}}
    =
    |s|\norm{\gradient{}\psi}_{\Lp{2}{\reals^d,\reals^d;\mu}}.
\end{equation*}
This concludes the proof. 
\end{proof}}

\newbool{showproofgradientsactontangentvectors}

\newcommand{\proofgradientsactontangentvectors}[1]{
\ifbool{#1}{\begin{proof}}{\begin{proof}[Proof of \cref{prop:gradients act on tangent vectors}]}
The proof of \eqref{eq:gradients act on tangent vectors} follows from~\cite[Proposition 8.5.4, Eq. 8.4.4]{Ambrosio2008a}. Note that when $\gamma$ is induced by a transport map, \eqref{eq:gradients act on tangent vectors} also follows from $\optMap{\mu}{\nu}-\Id\in\tangentSpace{\mu}\Pp{2}{\reals^d}$ (cf. \cref{rem:tangent space}) and the definition of $\Lp{2}{}$ inner product.
Indeed, since $\xi\in\tangentSpace{\mu}{\Pp{2}{\reals^d}}^\perp$ and $\optMap{\mu}{\nu}-\Id\in\tangentSpace{\mu}{\Pp{2}{\reals^d}}$, by definition of the orthogonal complement we have
\begin{equation*}
\begin{aligned}
    \int_{\reals^d\times\reals^d}\innerProduct{\xi(x)}{y-x}\d(\pushforward{(\Id,\optMap{\mu}{\nu})}\mu)(x,y)
    &=
    \int_{\reals^d}\innerProduct{\xi(x)}{\optMap{\mu}{\nu}(x)-x}\d\mu(x)
    \\
    &=
    \innerProduct{\xi}{\optMap{\mu}{\nu}-\Id}_{\Lp{2}{\reals^d,\reals^d;\mu}}
    \\
    &=0.
\end{aligned}
\end{equation*}
To show uniqueness of Wasserstein gradients it suffices to show that $\subdifferential{\cost}(\mu)\cap\superdifferential{\cost}(\mu)$ contains at most one element in $\tangentSpace{\mu}{\Pp{2}{\reals^d}}$.
Without loss of generality, assume $\subdifferential{\cost}(\mu)\cap\superdifferential{\cost}(\mu)$ is non-empty.
Let $\xi,\xi'\in\tangentSpace{\mu}{\Pp{2}{\reals^d}}\cap\subdifferential{\cost}(\mu)\cap\superdifferential{\cost}(\mu)$, and let $\varepsilon>0$.  Then, there exist $\varphi_\varepsilon,\varphi_\varepsilon'\in\Ccinf{\reals^d}$ such that
\begin{equation*}
    \LpNorm{\xi-\gradient{}\varphi_\varepsilon}{2}{\reals^d,\reals^d,\mu}<\frac{\varepsilon}{2}, 
    \qquad 
    \LpNorm{\xi'-\gradient{}\varphi_\varepsilon'}{2}{\reals^d,\reals^d,\mu}<\frac{\varepsilon}{2}.
\end{equation*}
Then, let $\sigma\coloneqq\pushforward{(\Id+s\gradient{}\varphi_\varepsilon)}\mu$ and $\sigma'\coloneqq\pushforward{(\Id+s\gradient{}\varphi_\varepsilon')}\mu$. By \cref{lemma:perturbation transport map}, there is $s>0$ sufficiently small such that $\Id+s\gradient{}\varphi_\varepsilon$ is an optimal transport map between $\mu$ and $\sigma$, and $\Id+s\gradient{}\varphi_\varepsilon'$ is an optimal transport map between $\mu$ and $\sigma'$. 
Thus, by definition of sub- and super-differentials
\begin{align*}
    \cost(\sigma)-\cost(\mu)&\geq\int_{\reals^d}\innerProduct{\xi}{s\gradient{}\varphi_\varepsilon}\d\mu(x)+\onotation{s}
    \\
    -\cost(\sigma')+\cost(\mu)&\geq\int_{\reals^d}\innerProduct{-\xi}{s\gradient{}\varphi_\varepsilon'}\d\mu(x)+\onotation{s}
    \\
    \cost(\sigma')-\cost(\mu)&\geq\int_{\reals^d}\innerProduct{\xi'}{s\gradient{}\varphi_\varepsilon'}\d\mu(x)+\onotation{s}
    \\
    -\cost(\sigma)+\cost(\mu)&\geq\int_{\reals^d}\innerProduct{-\xi'}{s\gradient{}\varphi_\varepsilon}\d\mu(x)+\onotation{s}.
\end{align*}
We can divide all inequalities by $s$, sum them up, and let $s\to 0$ to get  
\begin{equation}\label{eq:proposition only act on tangent vectors:ineqluality}
\begin{aligned}
    0
    &\geq
    \int_{\reals^d}\innerProduct{\xi}{\gradient{}\varphi_\varepsilon} + \innerProduct{\xi'}{\gradient{}\varphi_\varepsilon'} - \innerProduct{\xi}{\gradient{}\varphi_\varepsilon'} -  \innerProduct{\xi'}{\gradient{}\varphi_\varepsilon}\d\mu(x)
    \\
    &=
    \int_{\reals^d}\innerProduct{\xi-\xi'}{\gradient{}\varphi_\varepsilon-\gradient{}\varphi_\varepsilon'}\d\mu(x).
\end{aligned}
\end{equation}
Thus, 
\begin{align*}
    \LpNorm{\xi-\xi'}{2}{\reals^d,\reals^d;\mu}
    \overset{\eqref{eq:proposition only act on tangent vectors:ineqluality}}&{\leq}
    \begin{aligned}[t]
    &\bigg(\LpNorm{\xi-\xi'}{2}{\reals^d,\reals^d;\mu}^2 \\
    &-2\int_{\reals^d}\innerProduct{\xi-\xi'}{\gradient{}\varphi_\varepsilon-\gradient{}\varphi_\varepsilon'}\d\mu(x)
    +
    \LpNorm{\gradient{}\varphi_\varepsilon-\gradient{}\varphi_\varepsilon'}{2}{\reals^d,\reals^d;\mu}^2\bigg)^{\frac{1}{2}}
    \end{aligned}
    \\
    &=\LpNorm{\xi-\xi'-(\gradient{}\varphi_\varepsilon-\gradient{}\varphi_\varepsilon')}{2}{\reals^d,\reals^d;\mu}
    \\
    &\leq
    \LpNorm{\xi-\gradient{}\varphi_\varepsilon}{2}{\reals^d,\reals^d;\mu}
    +\LpNorm{\xi'-\gradient{}\varphi_\varepsilon}{2}{\reals^d,\reals^d;\mu}
    \\
    &<\varepsilon.
\end{align*}
Let $\varepsilon\to 0$ to conclude $\xi=\xi'$.
\end{proof}}

\newbool{showproofdifferentiablecontinuous}

\newcommand{\proofdifferentiablecontinuous}[1]{
\ifbool{#1}{\begin{proof}}{\begin{proof}[Proof of \cref{prop:differentiable continuous}]}
Let $(\mu_n)_{n\in\naturals}\subset\Pp{2}{\reals^d}$ such that $\mu_n\weakconvergence\mu$ and $\gamma_n\in\setOptimalPlans{\mu}{\mu_n}$. By differentiability, 
\begin{align*}
    &\limsup_{n\to\infty}|\cost(\mu_n)-\cost(\mu)|
    \\
    &\leq  
    \limsup_{n\to\infty}\left|\int_{\reals^d\times\reals^d}\innerProduct{\gradient{\mu}\cost(\mu)(x)}{y-x}\d\gamma_n(x,y)\right| + |\onotation{\wassersteinDistance{2}{\mu_n}{\mu}}|
    \\
    &\leq  
    \limsup_{n\to\infty}\left|\int_{\reals^d\times\reals^d}\innerProduct{\gradient{\mu}\cost(\mu)(x)}{y-x}\d\gamma_n(x,y)\right| + \limsup_{n\to\infty}|\onotation{\wassersteinDistance{2}{\mu_n}{\mu}}|
    \\
    &\leq
    \limsup_{n\to\infty}\left(\int_{\reals^d\times\reals^d}\norm{\gradient{\mu}\cost(\mu)(x)}^2\d\gamma_n(x,y)\right)^{\frac{1}{2}}\left(\int_{\reals^d\times\reals^d}\norm{y-x}^2\d\gamma_n(x,y)\right)^{\frac{1}{2}}
    \\
    &\leq
    \limsup_{n\to\infty}\left(\int_{\reals^d}\norm{\gradient{\mu}\cost(\mu)(x)}^2\d\mu(x)\right)^{\frac{1}{2}}\wassersteinDistance{2}{\mu}{\mu_n}
    \\
    &=
    \norm{\gradient{\mu}\cost(\mu)}_{\Lp{2}{\reals^d,\reals^d;\mu}}\limsup_{n\to\infty}\wassersteinDistance{2}{\mu_n}{\mu},
\end{align*}
where the second inequality uses subadditivity of the limit superior and the third inequality follows from Cauchy-Schwarz inequality in the Hilbert space $\Lp{2}{\reals^d,\reals^d;\gamma_n}$~\cite[Theorem 4.2]{Rudin1987} 
Since the Wasserstein distance metrizes weak convergence in $\Pp{2}{\reals^d}$, $\wassersteinDistance{2}{\mu_n}{\mu}\to 0$. Thus, we conclude $\lim_{n\to\infty}\cost(\mu_n)=\cost(\mu)$, and thus continuity w.r.t. weak convergence in $\Pp{2}{\reals^d}$.
\end{proof}}

\newbool{showproofgradientsarestrong}

\newcommand{\proofgradientsarestrong}[1]{
\ifbool{#1}{\begin{proof}}{\begin{proof}[Proof of \cref{prop:strong differentiability}]}
We prove subdifferentiability; the proofs of superdifferentiability and thus differentiability are analogous. Let $\xi\in\subdifferential\cost(\mu)$ be a subgradient of $\cost$ at $\mu$.
We conduct the proof by contradiction. Assume that there is $\delta>0$ and a sequence $(\nu_n)_{n\in\naturals}\subset\Pp{2}{\reals^d}$ and $\gamma_n\in\setPlans{\mu}{\nu_n}$ so that 
\begin{equation*}
    \varepsilon_n
    \coloneqq 
    \sqrt{\int_{\reals^d\times\reals^d}\norm{x-y}^2\d\gamma_n(x,y)}\to 0
\end{equation*}
and
\begin{equation}\label{prop:strong subdifferential 1}
    \cost(\nu_n)-\cost(\mu)
    -\int_{\reals^d\times\reals^d}\innerProduct{\xi(x)}{y-x}\d\gamma_n(x,y)\leq -\delta \varepsilon_n.
\end{equation}
Then, let $\beta_n\in\setOptimalPlans{\mu}{\nu_n}$ be an optimal transport plan between $\mu$ and $\nu_n$.
By definition, $\wassersteinDistance{2}{\mu}{\nu_n}\leq\varepsilon_n\to 0$.
By definition of subdifferential, there exists $N$ so that for all $n>N$ we have 
\begin{equation}\label{prop:strong subdifferential 2}
    \cost(\nu_n)-\cost(\mu)
    \geq 
    \int_{\reals^d\times\reals^d}\innerProduct{\xi(x)}{y-x}\d\beta_n(x,y)-\frac{\delta}{2}\varepsilon_n.
\end{equation}
We can combine~\eqref{prop:strong subdifferential 1} and \eqref{prop:strong subdifferential 2} to get
\begin{equation}\label{prop:strong subdifferential 3}
    \int_{\reals^d\times\reals^d}\innerProduct{\xi(x)}{y-x}\d\beta_n(x,y)-
    \int_{\reals^d\times\reals^d}\innerProduct{\xi(x)}{y-x}\d\gamma_n(x,y)
    \leq 
    -\frac{\delta}{2}\varepsilon_n
\end{equation}
for all $n>N$.
Let us now introduce the ``rescaled'' plans
\begin{equation}
\begin{aligned}
    \hat\gamma_n\coloneqq\pushforward{\left(\proj_1,\frac{\proj_2-\proj_1}{\varepsilon_n}\right)}\gamma_n,
    \\
    \hat\beta_n\coloneqq\pushforward{\left(\proj_1,\frac{\proj_2-\proj_1}{\varepsilon_n}\right)}\beta_n.
\end{aligned}
\end{equation}
Then,~\eqref{prop:strong subdifferential 3} can be rewritten to
\begin{equation}\label{prop:strong subdifferential 3 rewritten}
    \int_{\reals^d\times\reals^d}\innerProduct{\xi(x)}{y}\d\hat\beta_n(x,y)-
    \int_{\reals^d\times\reals^d}\innerProduct{\xi(x)}{y}\d\hat\gamma_n(x,y)
    \leq 
    -\frac{\delta}{2}
\end{equation}
for all $n>N$.
We now claim that $\hat\mu_n$ and $\hat\beta_n$ converge narrowly (up to subsequences) to some $\hat\mu$ and $\hat\beta$ with finite second moment:
\begin{itemize}
    \item Narrow convergence: We equivalently show that the sets $\{\hat\gamma_n\}_{n\in\naturals}$ and $\{\hat\beta_n\}_{n\in\naturals}$ are tight, which automatically establishes convergence to some $\hat\gamma$ and $\hat\beta$ by Prokhorov's Theorem (\cref{thm:prokhorov} in~\cref{app:preliminaries}).
    To show tightness it suffices to prove that 
    \begin{equation*}
        \sup_{n}\int_{\reals^d\times\reals^d}\norm{x}^2+\norm{y}^2\d\hat\gamma_n(x,y)<+\infty,
    \end{equation*}
    by~\cref{app:criterion tightness} in~\cref{app:preliminaries}. In particular, 
    \begin{equation}\label{eq:background strong differential gamma tight}
    \begin{aligned}
        \sup_{n}\int_{\reals^d\times\reals^d}\norm{x}^2+\norm{y}^2\d\hat\gamma_n(x,y)
        &=
        \sup_n\int_{\reals^d}\norm{x}^2\d\mu(x) + \int_{\reals^d}\norm{\frac{y-x}{\varepsilon_n}}^2\d\gamma_n(x,y)
        \\
        &=
        \sup_n\int_{\reals^d}\norm{x}^2\d\mu(x)+\frac{\varepsilon_n^2}{\varepsilon_n^2}
        \\
        &= 
        \int_{\reals^d}\norm{x}^2\d\mu(x)+1
        \\
        &<+\infty,
    \end{aligned}
    \end{equation}
    where we used that the first marginal of $\gamma_n$ is $\mu$, and $\mu$ has finite second moment. 
    Similarly, the set $\{\beta_n\}_{n\in\naturals}$ is tight since 
    \begin{equation}\label{eq:background strong differential beta tight}
    \begin{aligned}
        \sup_{n}\int_{\reals^d\times\reals^d}\norm{x}^2+\norm{y}^2\d\hat\beta_n(x,y)
        &=
        \sup_n\int_{\reals^d}\norm{x}^2\d\mu(x) + \int_{\reals^d}\norm{\frac{y-x}{\varepsilon_n}}^2\d\beta_n(x,y)
        \\
        &=
        \sup_n\int_{\reals^d}\norm{x}^2\d\mu(x)+\frac{\wassersteinDistance{2}{\mu}{\nu_n}^2}{\varepsilon_n^2}
        \\
        &\leq 
        \int_{\reals^d}\norm{x}^2\d\mu(x)+1
        \\
        &<+\infty.
    \end{aligned}
    \end{equation}
    Thus, $\gamma_n$ and $\beta_n$ converge narrowly (up to subsequences) to $\hat\mu$ and $\hat\beta$.
    
    \item Finite second moment: Since $(x,y)\mapsto \norm{x}^2+\norm{y}^2$ is non-negative and continuous (and thus lower semi-continuous), $\gamma\mapsto\int_{\reals^d\times\reals^d}\norm{x}^2+\norm{y}^2\d\gamma(x,y)$ is lower semi-continuous w.r.t. narrow convergence (see~\cref{prop:expectation} below or \cite[Lemma 5.1.7]{Ambrosio2008a}). Thus,
    \begin{equation*}
        \int_{\reals^d\times\reals^d}\norm{x}^2+\norm{y}^2\d\hat\gamma(x,y)
        \leq
        \liminf_{n\to\infty}\int_{\reals^d\times\reals^d}\norm{x}^2+\norm{y}^2\d\hat\gamma_n(x,y)
        \leq 
        \int_{\reals^d}\norm{x}^2\d\mu(x)+1,
    \end{equation*}
    which implies that $\hat\gamma$ has finite second moment. Analogously, $\hat\beta$ has finite second moment. 
\end{itemize}
We now seek to prove that 
\begin{equation}\label{eq:background strong differential beta inner product}
    \int_{\reals^d\times\reals^d}\innerProduct{\xi(x)}{y}\d\hat\beta_n(x,y)\to
    \int_{\reals^d\times\reals^d}\innerProduct{\xi(x)}{y}\d\hat\beta(x,y).
\end{equation}
Note that this does not follow from narrow convergence, since the inner product is not a bounded function. Let $\eta>0$, and define $\zeta_{\eta}\in\Ccinf{\reals^d}$ so that $\norm{\xi-\gradient{}\zeta_\eta}_{\Lp{2}{\reals^d,\reals^d;\mu}}<\eta$, which exists by density of gradients of compactly supported functions in $\tangentSpace{\mu}{\Pp{2}{\reals^d}}$. Then, 
\begin{equation*}
\begin{aligned}
    \int_{\reals^d\times\reals^d}\innerProduct{\zeta_{\eta}(x)}{y}\d\hat\beta_n(x,y)
    &=
    \int_{\reals^d\times\reals^d}\innerProduct{x}{y}\d(\pushforward{\zeta_{\eta}\times\Id)}\hat\beta_n(x,y).
\end{aligned}
\end{equation*}
By continuity of $\zeta_{\eta}$, $(\pushforward{\zeta_{\eta}\times\Id)}\hat\beta_n$ converges narrowly to $(\pushforward{\zeta_{\eta}\times\Id)}\hat\beta$. Moreover, we have that (i)
\begin{equation*}
\begin{aligned}
    \sup_n\int_{\reals^d\times\reals^d}\norm{x}^2+\norm{y}^2\d(\pushforward{\zeta_{\eta}\times\Id)}\hat\beta_n(x,y)
    \ifbool{compact}{&}{}=
    \sup_n\int_{\reals^d}\norm{\zeta_{\eta}(x)}^2\d\mu(x) + 
    \int_{\reals^d\times\reals^d}\norm{y}^2\d\hat\beta_n(x,y)
    \ifbool{compact}{\\ &}{}<
    +\infty
\end{aligned}
\end{equation*}
by the compact support of $\zeta_{\eta}$ and~\eqref{eq:background strong differential beta tight}, and (ii) $\pushforward{(\proj_1)}\pushforward{(\zeta_{\eta}\times\Id)}\hat\beta_n=\pushforward{(\zeta_{\eta})}\mu$ has trivially uniformly integrable second moment (since it does not depend on $n$). Thus, \cref{prop:narrow convergence transport plan} gives
\begin{equation*}
    \int_{\reals^d\times\reals^d}\innerProduct{\zeta_{\eta}(x)}{y}\d\hat\beta_n(x,y)\to
    \int_{\reals^d\times\reals^d}\innerProduct{\zeta_{\eta}(x)}{y}\d\hat\beta(x,y).
\end{equation*}
Thus, 
\begin{equation*}
\begin{aligned}
    \limsup_{n\to\infty}&\int_{\reals^d\times\reals^d}\innerProduct{\xi(x)}{y}\d\hat\beta_n(x,y)
    \\
    &=
    \limsup_{n\to\infty}\int_{\reals^d\times\reals^d}\innerProduct{\zeta_{\eta}(x)}{y}\d\hat\beta_n(x,y)+\int_{\reals^d\times\reals^d}\innerProduct{\xi(x)-\zeta_{\eta}(x)}{y}\d\hat\beta_n(x,y)
    \\
    &\leq 
    \lim_{n\to\infty}\int_{\reals^d\times\reals^d}\innerProduct{\zeta_{\eta}(x)}{y}\d\hat\beta_n(x,y) + \eta \sup_n\norm{y}_{\Lp{2}{\reals^d,\reals^d;\hat\beta_n}}
    \\
    \overset{\eqref{eq:background strong differential beta tight}}&{\leq}
    \lim_{n\to\infty}\int_{\reals^d\times\reals^d}\innerProduct{\zeta_{\eta}(x)}{y}\d\hat\beta_n(x,y) + \eta
    \sqrt{\int_{\reals^d}\norm{x}^2\d\mu(x)+1}
    \\
    &= 
    \int_{\reals^d\times\reals^d}\innerProduct{\zeta_{\eta}(x)}{y}\d\hat\beta(x,y) + \eta
    \sqrt{\int_{\reals^d}\norm{x}^2\d\mu(x)+1}
    \\
    &\leq 
    \int_{\reals^d\times\reals^d}\innerProduct{\xi(x)}{y}\d\hat\beta(x,y) + 2\eta
    \sqrt{\int_{\reals^d}\norm{x}^2\d\mu(x)+1},
\end{aligned}
\end{equation*}
where the first inequality results from Cauchy-Schwarz inequality~\cite[Theorem 4.2]{Rudin1987}. 
Similarly, 
\begin{equation*}
\begin{aligned}
    \liminf_{n\to\infty}&\int_{\reals^d\times\reals^d}\innerProduct{\xi(x)}{y}\d\hat\beta_n(x,y)
    \geq 
    \int_{\reals^d\times\reals^d}\innerProduct{\xi(x)}{y}\d\hat\beta(x,y) - 2\eta 
    \sqrt{\int_{\reals^d}\norm{x}^2\d\mu(x)+1}.
\end{aligned}
\end{equation*}
Let $\eta\to 0$ to establish~\eqref{eq:background strong differential beta inner product}. An analogous argument gives 
\begin{equation*}
    \int_{\reals^d\times\reals^d}\innerProduct{\xi(x)}{y}\d\hat\gamma_n(x,y)\to
    \int_{\reals^d\times\reals^d}\innerProduct{\xi(x)}{y}\d\hat\gamma(x,y),
\end{equation*}
and so, by \eqref{prop:strong subdifferential 3 rewritten},
\begin{equation}\label{eq:strong subdifferential contradiction negative}
    \int_{\reals^d\times\reals^d}\innerProduct{\xi(x)}{y}\d\hat\beta(x,y)-
    \int_{\reals^d\times\reals^d}\innerProduct{\xi(x)}{y}\d\hat\gamma(x,y)
    \leq 
    -\frac{\delta}{2}<0.
\end{equation}
We will show that the left-hand side of~\eqref{eq:strong subdifferential contradiction negative} is at the same time non-negative, which yields a contradiction. 
Consider $\zeta\in\Ccinf{\reals^d}$, and $x,y\in\reals^d$. Then, by Taylor's expansion, there exists $z\in \reals^d$ so that
\begin{equation}\label{eq:strong subdifferential taylor expansion 1}
    \zeta(y)=\zeta(x)+\innerProduct{\zeta(x)}{y-x}+\frac{1}{2}\innerProduct{(y-x)}{\gradient{}^2\zeta(z)(y-x)}.
\end{equation}
Thus, with $M=\sup_{x\in\reals^d}\norm{\gradient{}^2\zeta}$, we have
\begin{equation}\label{eq:strong subdifferential taylor expansion 2}
\begin{aligned}
    \zeta(y)-\zeta(x)
    \leq 
    \innerProduct{\gradient{}\zeta(x)}{y-x}+\frac{M}{2}\norm{x-y}^2.
\end{aligned}
\end{equation}
The same bound applied to $-\zeta$ yields
\begin{equation*}
\begin{aligned}
    -\zeta(y)+\zeta(x)
    \leq 
    \innerProduct{\gradient{}\zeta(x)}{x-y}+\frac{M}{2}\norm{x-y}^2.
\end{aligned}
\end{equation*}
Thus, for all $x,y_1,y_2\in\reals^d$ we have
\begin{equation*}
\begin{aligned}
    \zeta(y_1)-\zeta(y_2)
    \leq
    \innerProduct{\gradient{}\zeta(x)}{y_1-y_2}+\frac{M}{2}\norm{x-y_1}^2+\frac{M}{2}\norm{x-y_2}^2.
\end{aligned}
\end{equation*}
Since $\pushforward{(\proj_{1})}\beta_n=\pushforward{(\proj_{1})}\gamma_n=\mu$, we can deploy Gluing Lemma (\cref{prop:gluing} in~\cref{app:preliminaries}) to construct $\tilde\gamma_n\in\Pp{2}{\reals^d\times\reals^d\times\reals^d}$ so that
\begin{equation*}
    \pushforward{(\proj_{12})}\tilde\gamma_n=\beta_n
    \qquad
    \text{and}
    \qquad 
    \pushforward{(\proj_{13})}\tilde\gamma_n=\gamma_n.
\end{equation*}
Then, we have
\begin{align*}
    0
    &=
    \int_{\reals^d}\zeta(y_1)\d\nu_n(y_1)-\int_{\reals^d}\zeta(y_2)\d\nu_n(y_2)
    \\
    &=
    \int_{\reals^d\times\reals^d\times\reals^d}\zeta(y_1)-\zeta(y_2)\d\tilde\gamma_n(x,y_1,y_2)
    \\
    &\leq  
    \begin{aligned}[t]
    \ifbool{compact}{&}{}
    \int_{\reals^d\times\reals^d\times\reals^d}
    \innerProduct{\gradient{}\zeta(x)}{y_1-y_2}\d\tilde\gamma_n(x,y_1,y_2) \ifbool{compact}{\\ &}{}+ \frac{M}{2}\int_{\reals^d\times\reals^d\times\reals^d}\norm{y_1-x}^2+\norm{y_2-x}^2\d\tilde\gamma_n(x,y_1,y_2)
    \end{aligned}
    \\
    &=  
    \begin{aligned}[t]
    &\int_{\reals^d\times\reals^d\times\reals^d}
    \innerProduct{\gradient{}\zeta(x)}{y_1-x-(y_2-x)}\d\tilde\gamma_n(x,y_1,y_2) \\
    &+ \frac{M}{2}\int_{\reals^d\times\reals^d}\norm{y-x}^2\d\beta_n(x,y)+\frac{M}{2}\int_{\reals^d\times\reals^d}\norm{y-x}^2\d\gamma_n(x,y)
    \end{aligned}
    \\
    &\leq 
    \int_{\reals^d\times\reals^d}
    \innerProduct{\gradient{}\zeta(x)}{y-x}\d\beta_n(x,y)-\int_{\reals^d\times\reals^d}\innerProduct{\gradient{}\zeta(x)}{y-x}\d\gamma_n(x,y) + M\varepsilon_n^2
    \\
    &=
    \varepsilon_n\int_{\reals^d\times\reals^d}
    \innerProduct{\gradient{}\zeta(x)}{y}\d\hat\beta_n(x,y)-
    \varepsilon_n\int_{\reals^d\times\reals^d}\innerProduct{\gradient{}\zeta(x)}{y}\d\hat\gamma_n(x,y) + M\varepsilon_n^2.
\end{align*}
Overall, this can be re-expressed as  
\begin{equation*}
    \int_{\reals^d\times\reals^d}
    \innerProduct{\gradient{}\zeta(x)}{y}\d\hat\beta_n(x,y)
    -\int_{\reals^d\times\reals^d}\innerProduct{\gradient{}\zeta(x)}{y}\d\hat\gamma_n(x,y) + M\varepsilon_n
    \geq 0.
\end{equation*}
Again by~\cref{prop:narrow convergence transport plan} (see the proof of~\eqref{eq:background strong differential beta inner product} above), we let $n\to\infty$ and conclude
\begin{equation}\label{prop:strong subdifferential contradiction}
    \int_{\reals^d\times\reals^d}
    \innerProduct{\gradient{}\zeta(x)}{y}\d\hat\beta(x,y)-\int_{\reals^d\times\reals^d}\innerProduct{\gradient{}\zeta(x)}{y}\d\hat\gamma(x,y)
    \geq 0.
\end{equation}
We now use the density of gradients of smooth compactly supported functions in $\tangentSpace{\mu}\Pp{2}{\reals^d}$ to obtain the desired contradiction. 
Let $\zeta\in\Ccinf{\reals}$ so that $\norm{\gradient{}\zeta-\xi}\leq\frac{\delta}{6}(\int_{\reals^d}\norm{x}^2\d\mu(x)+1)^{-1}$. The existence of $\zeta$ is ensured by the density of gradients of smooth compactly supported functions in $\tangentSpace{\mu}\Pp{2}{\reals^d}$.
Then,
\begin{align*}
    0>
    -\frac{\delta}{2}
    \overset{\eqref{eq:strong subdifferential contradiction negative}}&{\geq}
    \int_{\reals^d\times\reals^d}\innerProduct{\xi(x)}{y}\d\hat\beta(x,y)-
    \int_{\reals^d\times\reals^d}\innerProduct{\xi(x)}{y}\d\hat\gamma(x,y)
    \\
    &\geq 
    \begin{aligned}[t]
    &\int_{\reals^d\times\reals^d}\innerProduct{\gradient{}\zeta(x)}{y}\d\hat\beta(x,y)+
    \int_{\reals^d\times\reals^d}\innerProduct{\xi(x)-\gradient{}\zeta(x)}{y}\d\hat\beta(x,y)
    \\
    &-\int_{\reals^d\times\reals^d}\innerProduct{\gradient{}\zeta(x)}{y}\d\hat\gamma(x,y) -
    \int_{\reals^d\times\reals^d}\innerProduct{\xi(x)-\gradient{}\zeta(x)}{y}\d\hat\gamma(x,y)
    \end{aligned}
    \\
    &\geq 
    \begin{aligned}[t]
    &\int_{\reals^d\times\reals^d}\innerProduct{\gradient{}\zeta(x)}{y}\d\hat\beta(x,y) -
    \norm{\xi-\gradient{}\zeta}_{\Lp{2}{\reals^d,\reals^d;\hat\beta}}\norm{y}_{\Lp{2}{\reals^d,\reals^d;\hat\beta}}
    \\
    &-\int_{\reals^d\times\reals^d}\innerProduct{\gradient{}\zeta(x)}{y}\d\hat\gamma(x,y) - 
    \norm{\xi-\gradient{}\zeta}_{\Lp{2}{\reals^d,\reals^d;\hat\gamma}}\norm{y}_{\Lp{2}{\reals^d,\reals^d;\hat\gamma}}
    \end{aligned}
    \\
    \overset{\eqref{eq:background strong differential beta tight},\eqref{eq:background strong differential gamma tight}}&{\geq}
    \begin{aligned}[t]
    &\int_{\reals^d\times\reals^d}\innerProduct{\gradient{}\zeta(x)}{y}\d\hat\beta(x,y) -
    \int_{\reals^d\times\reals^d}\innerProduct{\gradient{}\zeta(x)}{y}\d\hat\gamma(x,y)
    \\
    &-2\norm{\xi-\gradient{}\zeta}_{\Lp{2}{\reals^d,\reals^d;\mu}}\left(\int_{\reals^d}\norm{x}^2\d\mu(x)+1\right)
    \end{aligned}
    \\
    &\geq 
    \int_{\reals^d\times\reals^d}\innerProduct{\gradient{}\zeta(x)}{y}\d\hat\beta(x,y) -\int_{\reals^d\times\reals^d}\innerProduct{\gradient{}\zeta(x)}{y}\d\hat\gamma(x,y) - 
    \frac{\delta}{3}.
\end{align*}
which gives 
\begin{equation*}
    \int_{\reals^d\times\reals^d}\innerProduct{\gradient{}\zeta(x)}{y}\d\hat\beta(x,y) -\int_{\reals^d\times\reals^d}\innerProduct{\gradient{}\zeta(x)}{y}\d\hat\gamma(x,y)
    \leq 
    -\frac{\delta}{6}<0.
\end{equation*}
This contradicts~\eqref{prop:strong subdifferential contradiction}, and concludes the proof. 
\end{proof}}

\newbool{showproofgradientsgeodesicallyconvex}

\newcommand{\proofgradientsgeodesicallyconvex}[1]{
\ifbool{#1}{\begin{proof}}{\begin{proof}[Proof of \cref{prop:gradients geodesically convex}]}
Let $\gamma\in\setOptimalPlans{\mu}{\nu}$ and set $\mu_t\coloneqq\pushforward{((1-t)\proj_1+t\proj_2)}\gamma$. By definition of geodesic convexity
\begin{equation*}
    \cost(\mu_t)\leq(1-t)\cost(\mu)+t\cost(\nu)-\frac{\alpha}{2}t(1-t)\wassersteinDistance{2}{\mu}{\nu}^2
\end{equation*}
for all $t\in(0,1)$ and so
\begin{equation}\label{eq:proof gradients geodesically convex inequality 1}
    \frac{\cost(\mu_t)-\cost(\mu)}{t}\leq\cost(\nu)-\cost(\mu)-\frac{\alpha}{2}(1-t)\wassersteinDistance{2}{\mu}{\nu}^2.    
\end{equation}
Also, by definition of geodesic, $\wassersteinDistance{2}{\mu}{\mu_t}=t\wassersteinDistance{2}{\mu}{\nu}$ for all $t\in(0,1)$.
Let $\gamma_t\coloneqq\pushforward{(\proj_1,(1-t)\proj_1+t\proj_2)}\gamma$. We claim $\gamma_t\in\setOptimalPlans{\mu}{\mu_t}$. Indeed, 
\begin{equation*}
\begin{aligned}
    \int_{\reals^d\times\reals^d}\norm{x-y}^2\d\gamma_t
    =
    \int_{\reals^d\times\reals^d}\norm{x-((1-t)x+ty)}^2\d\gamma(x,y)
    =
    t^2\wassersteinDistance{2}{\mu}{\nu}^2
    =
    \wassersteinDistance{2}{\mu}{\mu_t}^2.
\end{aligned}
\end{equation*}
Thus, by Wasserstein subdifferentiability,
\begin{equation*}
\begin{aligned}
    \cost(\nu)-\cost(\mu)-\frac{\alpha}{2}\wassersteinDistance{2}{\mu}{\nu}^2
    &=
    \liminf_{t\downarrow 0} \cost(\nu)-\cost(\mu)-\frac{\alpha}{2}(1-t)\wassersteinDistance{2}{\mu}{\nu}^2
    \\
    \overset{\eqref{eq:proof gradients geodesically convex inequality 1}}&{\geq}
    \liminf_{t\downarrow 0} \frac{\cost(\mu_t)-\cost(\mu)}{t}
    \\
    &\geq
    \liminf_{t\downarrow 0}
    \frac{
    \int_{\reals^{d}\times\reals^d}\innerProduct{\xi(x)}{y-x}\d\gamma_t(x,y)
    +
    \onotation{t}}{t}
    \\
    &=
    \liminf_{t\downarrow 0}
    \frac{
    \int_{\reals^{d}\times\reals^d}\innerProduct{\xi(x)}{(1-t)x+ty-x}\d\gamma(x,y)
    +
    \onotation{t}}{t}
    \\
    &=
    \int_{\reals^{d}\times\reals^d}\innerProduct{\xi(x)}{y-x}\d\gamma(x,y).
\end{aligned}
\end{equation*}
In particular, the inequality holds for the supremum over $\gamma\in\setOptimalPlans{\mu}{\nu}$. For the second inequality, observe that $\pushforward{(\proj_2,\proj_1)}\gamma\in\setOptimalPlans{\nu}{\mu}$. So, 
\begin{equation}\label{eq:prop gradient monotone proof}
\begin{aligned}
    \cost(\mu)-\cost(\nu)
    &\geq
    \int_{\reals^d\times\reals^d}\innerProduct{\zeta(x)}{y-x}\d(\pushforward{(\proj_2,\proj_1)}\gamma)(x,y)+\frac{\alpha}{2}\wassersteinDistance{2}{\mu}{\nu}^2
    \\
    &=
    \int_{\reals^d\times\reals^d}\innerProduct{\zeta(y)}{x-y}\d\gamma(x,y)+\frac{\alpha}{2}\wassersteinDistance{2}{\mu}{\nu}^2
    \\
    &=
    \int_{\reals^d\times\reals^d}\innerProduct{-\zeta(y)}{y-x}\d\gamma(x,y)+\frac{\alpha}{2}\wassersteinDistance{2}{\mu}{\nu}^2.
\end{aligned}
\end{equation}
Finally, the ``sum'' of~\eqref{eq:prop gradient convex} and \eqref{eq:prop gradient monotone proof} establishes \eqref{eq:prop gradient monotone}.
\end{proof}}

\newbool{showproofsummultiplicationrule}

\newcommand{\proofsummultiplicationrule}[1]{
\ifbool{#1}{\begin{proof}}{\begin{proof}[Proof of~\cref{prop:sum and multiplication rule}]}
Let $\xi_1\in\subdifferential{\cost_1}(\mu)$ and $\xi_2\in\subdifferential{\cost_2}(\mu)$. Then, for any $\nu\in\Pp{2}{\reals^d}$, we have 
\begin{align*}
    (&\cost_1(\nu)+\cost_2(\nu))-(\cost_1(\mu)+\cost_2(\mu))
    \\
    &=
    \cost_1(\nu)-\cost_1(\mu)+\cost_2(\nu)-\cost_2(\mu)
    \\
    &\geq 
    \sup_{\gamma_1\in\setOptimalPlans{\mu}{\nu}}\int_{\reals^d}\innerProduct{\xi_1(x)}{y-x}\d\gamma_1(x,y)+
    \sup_{\gamma_2\in\setOptimalPlans{\mu}{\nu}}\int_{\reals^d}\innerProduct{\xi_2(x)}{y-x}\d\gamma_2(x,y)
    +\onotation{\wassersteinDistance{2}{\mu}{\nu}}
    \\
    &\geq 
    \sup_{\gamma\in\setOptimalPlans{\mu}{\nu}}\int_{\reals^d}\innerProduct{\xi_1(x)}{y-x}\d\gamma(x,y)+\int_{\reals^d}\innerProduct{\xi_2(x)}{y-x}\d\gamma(x,y)
    +\onotation{\wassersteinDistance{2}{\mu}{\nu}}
    \\
    &= 
    \sup_{\gamma\in\setOptimalPlans{\mu}{\nu}}\int_{\reals^d}\innerProduct{\xi_1(x)+\xi_2(x)}{y-x}\d\gamma(x,y)
    +\onotation{\wassersteinDistance{2}{\mu}{\nu}}.
\end{align*}
Thus, $\xi_1+\xi_2\in\subdifferential{(\cost_1+\cost_2)}(\mu)$, which proves (i). For (ii), it suffices to observe that  
\begin{equation*}
    \alpha\cost_1(\nu)-\alpha\cost_1(\mu)
    \geq 
    \sup_{\gamma\in\setOptimalPlans{\mu}{\nu}}\int_{\reals^d}\innerProduct{\xi(x)}{y-x}\d\gamma(x,y)
    +\onotation{\wassersteinDistance{2}{\mu}{\nu}}
\end{equation*}
holds if and only if $\xi\in\alpha\subdifferential{\cost_1}(\mu)$. 
We now focus on the differentiable case. One inclusion follows readily from (i) above:
\begin{equation*}
    \subdifferential{(\cost_1+\cost_2)}(\mu)
    \supset \{\gradient{\mu}\cost_1(\mu)\}
    +\subdifferential{\cost_2}(\mu).
\end{equation*}
For the other inclusion, we use (i) again: 
\begin{align*}
    \subdifferential{\cost_2}(\mu)
    &=\subdifferential{((\cost_1+\cost_2)-\cost_1)}(\mu)
    \\
    &\supset \subdifferential{(\cost_1+\cost_2)}(\mu) + \subdifferential{(-\cost_1)}(\mu)
    =
    \subdifferential{(\cost_1+\cost_2)}(\mu) + \{-\gradient{\mu}\cost_1(\mu)\},
\end{align*}
where we used super-differentiability of $\cost_1$. This proves the other inclusion and thus (i) for the differentiable case. Finally, (ii) follows directly as (ii) in the subdifferentiable case, with the definition of Wasserstein differentiability.
\end{proof}}

\newbool{showproofchainrule}

\newcommand{\proofchainrule}[1]{
\ifbool{#1}{\begin{proof}}{\begin{proof}[Proof of \cref{prop:chain rule}]}
Let $\nu\in\Pp{2}{\reals^d}$ and $\gamma\in\setOptimalPlans{\mu}{\nu}$. Then, 
\begin{equation*}
\begin{aligned}
    g(\cost(\nu))-g(\cost(\mu))
    &=
    g'(\cost(\mu))(\cost(\nu)-\cost(\mu)) + \onotation{|\cost(\nu)-\cost(\mu)|}
    \\
    &=
    \begin{aligned}[t]
    &g'(\cost(\mu))\left(
    \int_{\reals^d\times\reals^d}\innerProduct{\gradient{\mu}\cost(\mu)(x)}{y-x}\d\gamma(x,y) + \onotation{\wassersteinDistance{2}{\nu}{\mu}}
    \right)
    \\&+
    \onotation{|\cost(\nu)-\cost(\mu)|}
    \end{aligned}
    \\
    &=
    \begin{aligned}[t]
    &\int_{\reals^d\times\reals^d}\innerProduct{ g'(\cost(\mu))\gradient{\mu}\cost(\mu)(x)}{y-x}\d\gamma(x,y) \\&+g'(\cost(\mu))\onotation{\wassersteinDistance{2}{\nu}{\mu}} + 
    \onotation{|\cost(\nu)-\cost(\mu)|}.
    \end{aligned}
\end{aligned}
\end{equation*}
By differentiability (and thus continuity, by~\cref{prop:differentiable continuous}) of $\cost$ we have 
\begin{equation*}
    \lim_{\wassersteinDistance{2}{\nu}{\mu}\to 0}\frac{\onotation{|\cost(\nu)-\cost(\mu)|}}{\wassersteinDistance{2}{\nu}{\mu}}
    =
    \lim_{\wassersteinDistance{2}{\nu}{\mu}\to 0}\frac{\onotation{|\cost(\nu)-\cost(\mu)|}}{|\cost(\nu)-\cost(\mu)|}\frac{|\cost(\nu)-\cost(\mu)|}{\wassersteinDistance{2}{\nu}{\mu}}
    =
    0.
\end{equation*}
This concludes the proof. 
\end{proof}}

\newbool{showproofoptimaltransportdiscrepancy}

\newcommand{\proofoptimaltransportdiscrepancy}[1]{
\ifbool{#1}{\begin{proof}}{\begin{proof}[Proof of \cref{prop:optimal transport discrepancy}]}
First, by continuity and non-negativity of $c$, $\cost$ is well defined (e.g., see \cite[Chapter 4]{Villani2009a}).
We prove the statements separately. 
\begin{enumerate}
    \item Non-negativity follows directly from non-negativity of $c$. Moreover, since $\pushforward{(\Id,\Id)}\bar\mu\in\setPlans{\bar\mu}{\bar\mu}$, $\cost(\bar\mu)\leq \int_{\reals^d}c(y,y)\d\bar\mu(y)\leq C$. Thus, $\cost$ is proper. 
    
    \item Since weak convergence in $\Pp{2}{\reals^d}$ implies narrow convergence, the statement follows from (iii). 
    \item The statement is given in~\cite[Remark 6.12]{Villani2009a} and in~\cite[Proposition 7.4]{Santambrogio2015}.
    
    \item Let $\mu_0,\mu_1\in\Pp{2}{\reals^d}$, $\gamma\in\setPlans{\mu_0}{\mu_1}$, and consider the interpolation
    \begin{equation*}
        \mu_t=\pushforward{((1-t)\proj_1+t\proj_2)}\gamma.
    \end{equation*}
    The proof simply leverages the $\alpha$-convexity of $-c(\cdot,y)$; i.e., 
    \begin{equation}\label{eq:optimal transport discrepancy alpha convexity convexity c}
        -c((1-t)x_0+tx_1,y)
        \leq 
        -(1-t)c(x_0,y) - tc(x_1,y) -\frac{\alpha}{2}t(1-t)\norm{x_0-x_1}^2.
    \end{equation}
    Moreover, from~\cite[Lemma 2]{Aolaritei2023DistributionalTransport}, for Polish spaces $X,Y,Z$, probability measures $\nu_0\in\Pp{}{X}$, $\nu_1\in\Pp{}{Y}$, and a Borel map $f:X\to Z$
    \begin{equation}\label{eq:optimal transport discrepancy alpha convexity property set plans}
        \pushforward{(f\times \Id_Y)}\setPlans{\nu_0}{\nu_1}
        =
        \setPlans{\pushforward{f}\nu_0}{\nu_1}
        \subset
        \Pp{}{Z\times Y},
    \end{equation}  
    where $\Id_Y$ is the identity map on $Y$.
    Then, 
    \begin{align*}
        -\cost(\mu_t)
        &=
        -\min_{\tilde\gamma\in\setPlans{\mu_t}{\bar\mu}}
        \int_{\reals^d\times\reals^d}c(x,y)\d\tilde\gamma(x,y)
        \\
        &=
        -\min_{\tilde\gamma\in\setPlans{\pushforward{((1-t)\proj_1+t\proj_2)}\gamma}{\bar\mu}}
        \int_{\reals^d\times\reals^d}c(x,y)\d\tilde\gamma(x,y)
        \\
        \overset{\eqref{eq:optimal transport discrepancy alpha convexity property set plans}}&{=}
        -\min_{\tilde\gamma\in\pushforward{(((1-t)\proj_1+t\proj_2)\times\Id)}\setPlans{\gamma}{\bar\mu}}
        \int_{\reals^d\times\reals^d}c(x,y)\d\tilde\gamma(x,y)
        \\
        &=
        -\min_{\bar\gamma\in\setPlans{\gamma}{\bar\mu}}
        \int_{\reals^d\times\reals^d}c(x,y)\d(\pushforward{(((1-t)\proj_1+t\proj_2)\times\Id)}\bar\gamma)(x,y)
        \\
        &=
        -\min_{\bar\gamma\in\setPlans{\gamma}{\bar\mu}}
        \int_{\reals^d\times\reals^d\times\reals^d}c((1-t)x_0+x_1,y)\d\bar\gamma(x_0,x_1,y)
        \\
        &=
        \max_{\bar\gamma\in\setPlans{\gamma}{\bar\mu}}
        \int_{\reals^d\times\reals^d\times\reals^d}-c((1-t)x_0+x_1,y)\d\bar\gamma(x_0,x_1,y)
        \\
        \overset{\eqref{eq:optimal transport discrepancy alpha convexity convexity c}}&{\leq}
        \begin{aligned}[t]
        \max_{\bar\gamma\in\setPlans{\gamma}{\bar\mu}}
        &-(1-t)\int_{\reals^d\times\reals^d\times\reals^d}c(x_0,y)\d\bar\gamma(x_0,x_1,y)
        \ifbool{compact}{\\&}{}
        -t\int_{\reals^d\times\reals^d\times\reals^d}c(x_1,y)\d\bar\gamma(x_0,x_1,y)
        \\
        &-\frac{\alpha}{2}t(1-t)\int_{\reals^d\times\reals^d\times\reals^d}\norm{x_0-x_1}^2\d\bar\gamma(x_0,x_1,y)
        \end{aligned}
        \\
        &\leq 
        \begin{aligned}[t]
        &\max_{\bar\gamma\in\setPlans{\gamma}{\bar\mu}}
        -(1-t)\int_{\reals^d\times\reals^d\times\reals^d}c(x_0,y)\d\bar\gamma(x_0,x_1,y)
        \\
        &+\max_{\bar\gamma\in\setPlans{\gamma}{\bar\mu}}-t\int_{\reals^d\times\reals^d\times\reals^d}c(x_1,y)\d\bar\gamma(x_0,x_1,y)
        \\
        &-\frac{\alpha}{2}t(1-t)\int_{\reals^d\times\reals^d}\norm{x_0-x_1}^2\d\gamma(x_0,x_1)
        \end{aligned}
        \\
        &\leq 
        \begin{aligned}[t]
        &-(1-t)\min_{\bar\gamma\in\setPlans{\gamma}{\bar\mu}}\int_{\reals^d\times\reals^d\times\reals^d}c(x_0,y)\d\bar\gamma(x_0,x_1,y)
        \\
        &-t\min_{\bar\gamma\in\setPlans{\gamma}{\bar\mu}}\int_{\reals^d\times\reals^d\times\reals^d}c(x_1,y)\d\bar\gamma(x_0,x_1,y)
        -\frac{\alpha}{2}t(1-t)\wassersteinDistance{2}{\mu_0}{\mu_1}^2
        \end{aligned}
        \\
        &=
        \begin{aligned}[t]
        &-(1-t)\min_{\bar\gamma\in\setPlans{\gamma}{\bar\mu}}\int_{\reals^d\times\reals^d\times\reals^d}c(x_0,y)\d(\pushforward{(\proj_1\times\Id)}\bar\gamma)(x_0,x_1,y)
        \\
        &-t\min_{\bar\gamma\in\setPlans{\gamma}{\bar\mu}}\int_{\reals^d\times\reals^d\times\reals^d}c(x_1,y)\d(\pushforward{(\proj_2\times\Id)}\bar\gamma)(x_0,x_1,y)
        \\
        &-\frac{\alpha}{2}t(1-t)\wassersteinDistance{2}{\mu_0}{\mu_1}^2
        \end{aligned}
        \\
        &=
        \begin{aligned}[t]
        &-(1-t)\min_{\tilde\gamma_0\in\pushforward{(\proj_1\times\Id)}\setPlans{\gamma}{\bar\mu}}\int_{\reals^d\times\reals^d}c(x_0,y)\d\tilde\gamma_0(x_0,y)
        \\
        &-t\min_{\tilde\gamma_1\in\pushforward{(\proj_1\times\Id)}\setPlans{\gamma}{\bar\mu}}\int_{\reals^d\times\reals^d}c(x_1,y)\d\tilde\gamma_1(x_1,y)
        -\frac{\alpha}{2}t(1-t)\wassersteinDistance{2}{\mu_0}{\mu_1}^2
        \end{aligned}
        \\
        \overset{\eqref{eq:optimal transport discrepancy alpha convexity property set plans}}&{=}
        \begin{aligned}[t]
        &-(1-t)\min_{\tilde\gamma_0\in\setPlans{\pushforward{(\proj_1)}\gamma}{\bar\mu}}\int_{\reals^d\times\reals^d}c(x_0,y)\d\tilde\gamma_0(x_0,y)
        \\
        &-t\min_{\tilde\gamma_1\in\setPlans{\pushforward{(\proj_2)}\gamma}{\bar\mu}}\int_{\reals^d\times\reals^d}c(x_1,y)\d\tilde\gamma_1(x_1,y)
        -\frac{\alpha}{2}t(1-t)\wassersteinDistance{2}{\mu_0}{\mu_1}^2
        \end{aligned}
        \\
        &=
        \begin{aligned}[t]
        &-(1-t)\min_{\tilde\gamma_0\in\setPlans{\mu_0}{\bar\mu}}\int_{\reals^d\times\reals^d}c(x_0,y)\d\tilde\gamma_0(x_0,y)
        \\
        &-t\min_{\tilde\gamma_1\in\setPlans{\mu_1}{\bar\mu}}\int_{\reals^d\times\reals^d}c(x_1,y)\d\tilde\gamma_1(x_1,y)
        -\frac{\alpha}{2}t(1-t)\wassersteinDistance{2}{\mu_0}{\mu_1}^2
        \end{aligned}
        \\
        &=
        -(1-t)\cost(\mu_0) - t\cost(\mu_1) -\frac{\alpha}{2}t(1-t)\wassersteinDistance{2}{\mu_0}{\mu_1}^2.
    \end{align*}
    This proves $\alpha$-convexity of $-\cost$ along any interpolating curve.

    \item First, observe that $\pushforward{((1-t)\proj_2+t\proj_3,\proj_1)}\gamma\in\setPlans{\mu_t}{\bar\mu}.$
    Thus,
    \begin{align*}
        \cost(\mu_t)
        &\leq
        \int_{\reals^d\times\reals^d} c(x,y)\d(\pushforward{((1-t)\proj_2+t\proj_3,\proj_1)}\gamma)(x,y)
        \\
        &=
        \int_{\reals^d\times\reals^d\times\reals^d} c((1-t)x_0+tx_1),x)\d\gamma(x,x_0,x_1)
        \\
        &\leq 
        \begin{aligned}[t]
        &(1-t)\int_{\reals^d\times\reals^d\times\reals^d}c(x_0,x)\d\gamma(x,x_0,x_1)
        +
        t\int_{\reals^d\times\reals^d\times\reals^d}c(x_1,x)\d\gamma(x,x_0,x_1)
        \\
        &-
        \frac{\alpha}{2}t(1-t)\int_{\reals^d\times\reals^d\times\reals^d} \norm{x_0-x_1}^2\d\gamma(x,x_0,x_1)
        \end{aligned}
        \\
        &=
        (1-t)\cost(\mu) + t\cost(\nu)
        -
        \frac{\alpha}{2}t(1-t)\int_{\reals^d\times\reals^d\times\reals^d} \norm{x_0-x_1}^2\d\gamma(x,x_0,x_1)
        \\
        &\leq 
        (1-t)\cost(\mu) + t\cost(\nu)
        -
        \frac{\alpha}{2}t(1-t)\wassersteinDistance{2}{\mu}{\nu}^2.
    \end{align*}
    
    \item To show that the Wasserstein gradient is well-defined observe that
    \begin{align*}
        \int_{\reals^d}\norm{\gradient{x}c(x,\optMap{\mu}{\bar\mu}(x))}^2\d\mu(x)
        &\leq 
        \int_{\reals^d}M\left(1+\norm{x}^2+\norm{\optMap{\mu}{\bar\mu}(x)}^2\right)\d\mu(x)
        \\
        &=
        M\left(1 + \int_{\reals^d}\norm{x}^2\d\mu(x)+\int_{\reals^d}\norm{\optMap{\mu}{\bar\mu}(x)}^2\d\mu(x)\right)
        \\
        &<+\infty,
    \end{align*}
    since $\mu,\bar\mu\in\Pp{2}{\reals^d}$.
    Then, we prove subdifferentiability and superdifferentiability separately. 
    \begin{itemize}
        \item Superdifferentiability: Let $\nu\in\Pp{2}{\reals^d}$ and $\gamma\in\setOptimalPlans{\mu}{\nu}$. Via a Taylor's expansion (recall~\eqref{eq:strong subdifferential taylor expansion 1} and \eqref{eq:strong subdifferential taylor expansion 2} in the proof of~\cref{prop:strong differentiability}), we have that
        \begin{equation*}
            c(x_1,y)-c(x_0,y)
            \leq 
            \innerProduct{\gradient{x}c(x_0,y)}{x_1-x_0}
            +
            \frac{M}{2} \norm{x_1-x_0}^2,
        \end{equation*}
        where we used that the norm of Hessian $\gradient{x}^2c(x_0,y)$ is uniformly bounded. Then, 
        \begin{equation*}
        \begin{aligned}
            -\cost(\nu)+\cost(\mu)
            &\geq 
            -\int_{\reals^d\times\reals^d}c(x_1,\optMap{\mu}{\bar\mu}(x_0))\d\gamma(x_0,x_1)+\int_{\reals^d}c(x_0,\optMap{\mu}{\bar\mu}(x_0))\d\mu(x_0)
            \\
            &=
            -\int_{\reals^d\times\reals^d}c(x_1,\optMap{\mu}{\bar\mu}(x_0))-c(x_0,\optMap{\mu}{\bar\mu}(x_0))\d\gamma(x_0,x_1)
            \\
            &\geq 
            -\int_{\reals^d\times\reals^d}\innerProduct{\gradient{x}c(x_0,\optMap{\mu}{\bar\mu}(x_0))}{x_1-x_0}
            +\frac{M}{2}\norm{x_0-x_1}^2
            \d\gamma(x_0,x_1)
            \\
            &=
            \int_{\reals^d\times\reals^d}\innerProduct{-\gradient{x}c(x_0,\optMap{\mu}{\bar\mu}(x_0))}{x_1-x_0}\d\gamma(x_0,x_1)
            -\frac{M}{2}\wassersteinDistance{2}{\mu}{\nu}^2,
        \end{aligned}
        \end{equation*}
        which proves superdifferentiability and, in particular, $-\gradient{x}c(\cdot,\optMap{\mu}{\bar\mu}(\cdot))\in\subdifferential{(-\cost)}$.
        
        \item Subdifferentiability: Let $(\mu_n)_{n\in\naturals}\subset\Pp{2}{\reals^d}$ so that $\mu_n\weakconvergence\mu$. Define $\gamma_n\in\Pp{2}{\reals^d\times\reals^d\times\reals^d}$ so that
        \begin{equation*}
            \pushforward{(\proj_{13})}\gamma_n\in\setOptimalPlans{\mu}{\mu_n}
            \qquad
            \pushforward{(\proj_{12})}\gamma_n=\pushforward{(\Id,\optMap{\mu}{\bar\mu})}\mu,
        \end{equation*}
        which exists by Gluing Lemma (\cref{prop:gluing} in~\cref{app:preliminaries}).
        Informally, $\gamma_n$ is optimal from $\mu$ to $\mu_n$, and optimal from $\mu$ to $\bar\mu$.
        Moreover, again by Gluing Lemma (\cref{prop:gluing} in~\cref{app:preliminaries}), define $\beta_n\in\Pp{2}{\reals^d\times\reals^d\times\reals^d}$ so that 
        \begin{equation*}
            \pushforward{(\proj_{13})}\beta_n=\pushforward{(\proj_{13})}\gamma_n,
            \qquad
            \pushforward{(\proj_{23})}\beta_n\in\setOptimalPlans{\bar\mu}{\mu_n}.
        \end{equation*}
        Informally, $\beta_n$ is optimal from $\mu$ to $\mu_n$ (analogously to $\gamma_n$), and optimal from $\bar\mu$ to $\mu_n$, and thus does not coincide with $\gamma_n$. Accordingly, $\beta_n$ is generally not optimal from $\mu$ to $\bar\mu$; i.e., 
        \begin{equation*}
            \cost(\mu)=\int_{\reals^d}c(x,\optMap{\mu}{\bar\mu}(x))\d\mu(x)\leq \int_{\reals^d\times\reals^d\times\reals^d}c(x_1,x_2)\d\beta_n(x_1,x_2,x_3).
        \end{equation*}
        Define $\varepsilon_n\coloneqq\wassersteinDistance{2}{\mu_n}{\mu}\to 0$, and define the ``rescaled plans''
        \begin{equation}
        \begin{aligned}
            \hat\gamma_n
            \coloneqq\pushforward{\left(\proj_1,\proj_2,\frac{\proj_3-\proj_1}{\varepsilon_n}\right)}\gamma_n,
            \\
            \hat\beta_n
            \coloneqq\pushforward{\left(\proj_1,\proj_2,\frac{\proj_3-\proj_1}{\varepsilon_n}\right)}\beta_n.
        \end{aligned}
        \end{equation}
        Then, 
        \begin{equation}\label{eq:optimal transport discrepancy subdifferential inequality 1}
        \begin{aligned}
            \cost(\mu_n)-\cost(\mu)
            &\geq 
            \int_{\reals^d\times\reals^d\times\reals^d}c(x_3,x_2)-c(x_1,x_2)\d\beta_n(x_1,x_2,x_3)
            \\
            &=
            \int_{\reals^d\times\reals^d\times\reals^d}c(x_1+\varepsilon_n x_3,x_2)-c(x_1,x_2)\d\hat\beta_n(x_1,x_2,x_3)
            \\
            &\geq
            \int_{\reals^d\times\reals^d\times\reals^d}\innerProduct{\gradient{x}c(x_1,x_2)}{\varepsilon_n x_3} - \frac{M}{2}\varepsilon_n^2\norm{x_3}^2\d\hat\beta_n(x_1,x_2,x_3)
            \\
            &=
            \begin{aligned}[t]
            &\varepsilon_n \int_{\reals^d\times\reals^d\times\reals^d}\innerProduct{\gradient{x}c(x_1,x_2)}{x_3}\d\hat\beta_n(x_1,x_2,x_3)
            \\
            &- \frac{M}{2}\varepsilon_n^2 \int_{\reals^d\times\reals^d\times\reals^d}\norm{x_3}^2\d\hat\beta_n(x_1,x_2,x_3),
            \end{aligned}
        \end{aligned}
        \end{equation}
        and for any $\tilde\gamma_n\in\setOptimalPlans{\mu}{\mu_n}$
        \begin{equation}\label{eq:optimal transport discrepancy subdifferential inequality 2}
        \begin{aligned}
            \int_{\reals^d\times\reals^d}&\innerProduct{\gradient{x}c(x_1,\optMap{\mu}{\bar\mu}(x_1)}{x_3-x_1}\d\tilde\gamma_n(x_1,x_3)
            \\
            &\hspace{2cm}=
            \int_{\reals^d\times\reals^d\times\reals^d}\innerProduct{\gradient{x}c(x_1,x_2)}{x_3-x_1}\d\gamma_n(x_1,x_2,x_3)
            \\
            &\hspace{2cm}=
            \int_{\reals^d\times\reals^d\times\reals^d}\innerProduct{\gradient{x}c(x_1,x_2)}{\varepsilon_n x_3}\d\hat\gamma_n(x_1,x_2,x_3)
            \\
            &\hspace{2cm}=
            \varepsilon_n\int_{\reals^d\times\reals^d\times\reals^d}\innerProduct{\gradient{x}c(x_1,x_2)}{x_3}\d\hat\gamma_n(x_1,x_2,x_3).
        \end{aligned}
        \end{equation}
        We now claim that $\hat\gamma_n$ and $\hat\beta_n$ converge narrowly (up to subsequences) to some $\hat\gamma$ and $\hat\beta$ with finite second moment:
        \begin{itemize}
        \item Narrow convergence: We equivalently show that the sets $\{\hat\gamma_n\}_{n\in\naturals}$ and $\{\hat\beta_n\}_{n\in\naturals}$ are tight, which automatically establishes convergence to some $\hat\gamma$ and $\hat\beta$ by Prokhorov's theorem (\cref{thm:prokhorov} in~\cref{app:preliminaries}).
        To show tightness it suffices to prove that 
        \begin{equation*}
            \sup_{n}\int_{\reals^d\times\reals^d\times\reals^d}\norm{x_1}^2+\norm{x_2}^2+\norm{x_3}^2\d\hat\gamma_n(x_1,x_2,x_3)<+\infty,
        \end{equation*}
        by~\cref{app:criterion tightness} in~\cref{app:preliminaries}. In particular, 
        \begin{equation}\label{eq:optimal transport discrepancy subdifferential tight gamma}
        \begin{aligned}
            \sup_{n}&\int_{\reals^d\times\reals^d\times\reals^d}\norm{x_1}^2+\norm{x_2}^2+\norm{x_3}^2\d\hat\gamma_n(x_1,x_2,x_3)
            \\
            &=
            \sup_n\int_{\reals^d}\norm{x}^2\d\mu(x) + \int_{\reals^d}\norm{x}^2\d\bar\mu(x) 
            + \int_{\reals^d\times\reals^d\times\reals^d}\norm{\frac{x_3-x_1}{\varepsilon_n}}^2\d\gamma_n(x_1,x_2,x_3)
            \\
            &=
            \sup_n\int_{\reals^d}\norm{x}^2\d\mu(x)+ \int_{\reals^d}\norm{x}^2\d\bar\mu(x) + \frac{\varepsilon_n^2}{\varepsilon_n^2}
            \\
            &= 
            \int_{\reals^d}\norm{x}^2\d\mu(x)+\int_{\reals^d}\norm{x}^2\d\bar\mu(x)+1
            \\
            &<+\infty,
        \end{aligned}
        \end{equation}
        where we used that the first marginal of $\gamma_n$ is $\mu$, the second marginal is $\bar\mu$, and both $\mu$ and $\bar\mu$ have finite second moment. 
        Analogously,
        \begin{equation}\label{eq:optimal transport discrepancy subdifferential tight beta}
        \begin{aligned}
            \sup_{n}\int_{\reals^d\times\reals^d\times\reals^d}&\norm{x_1}^2+\norm{x_2}^2+\norm{x_3}^2\d\hat\beta_n(x_1,x_2,x_3)
            \\
            &\leq 
            \int_{\reals^d}\norm{x}^2\d\mu(x)+\int_{\reals^d}\norm{x}^2\d\bar\mu(x)+1 
            \\
            &<+\infty,
        \end{aligned}
        \end{equation}
        and the set $\{\beta_n\}_{n\in\naturals}$ is tight. 
        Thus, $\hat\gamma_n$ and $\hat\beta_n$ converge narrowly (up to subsequences) to $\hat\gamma$ and $\hat\beta$.
        
        \item Finite second moment: Since $(x_1,x_2,x_3)\mapsto \norm{x_1}^2+\norm{x_2}^2+\norm{x_3}^2$ is non-negative and lower semi-continuous, $\gamma\mapsto\int_{\reals^d\times\reals^d\times\reals^d}\norm{x_1}^2+\norm{x_2}^2+\norm{x_3}^2\d\gamma(x_1,x_2,x_3)$ is lower semi-continuous w.r.t. narrow convergence (see~\cref{prop:expectation} below or \cite[Lemma 5.1.7]{Ambrosio2008a}). Thus,
        \begin{equation*}
        \begin{aligned}
            \int_{\reals^d\times\reals^d\times\reals^d}&\norm{x_1}^2+\norm{x_2}^2+\norm{x_3}^2\d\hat\gamma(x_1,x_2,x_3)
            \\
            &\leq
            \liminf_{n\to\infty}\int_{\reals^d\times\reals^d\times\reals^d}\norm{x_1}^2+\norm{x_2}^2+\norm{x_3}^2\d\hat\gamma_n(x_1,x_2,x_3)
            \\
            \overset{\eqref{eq:optimal transport discrepancy subdifferential tight gamma}}&{\leq}
            \int_{\reals^d}\norm{x}^2\d\mu(x)+\int_{\reals^d}\norm{x}^2\d\bar\mu(x)+1,
        \end{aligned}
        \end{equation*}
        which implies $\hat\gamma$ has finite second moment. Analogously, $\hat\beta$ has finite second moment. 
    \end{itemize}
        Moreover, by construction of $\hat\beta_n$,
        \begin{equation}\label{eq:optimal transport discrepancy subdifferential equality proj 13}
            \pushforward{(\proj_{13})}\hat\gamma=\pushforward{(\proj_{13})}\hat\beta.
        \end{equation}
        Indeed, for any $\phi\in\Cb{\reals^d\times\reals^d}$, which by approximation suffices to prove~\eqref{eq:optimal transport discrepancy subdifferential equality proj 13}, we have
        \begin{align*}
            \int_{\reals^d\times\reals^d}
            \phi(x_1,x_3)\d(\pushforward{(\proj_{13})}\hat\gamma)(x_1,x_3)
            &=
            \int_{\reals^d\times\reals^d}
            \phi(x_1,x_3)\d\hat\gamma(x_1,x_2x_3)
            \\
            &=
            \lim_{n\to\infty}\int_{\reals^d\times\reals^d\times\reals^d}\phi(x_1,x_3)\d\hat\gamma_n(x_1,x_2,x_3)
            \\
            &=
            \lim_{n\to\infty}\int_{\reals^d\times\reals^d\times\reals^d}\phi(x_1,x_3)\d\hat\beta_n(x_1,x_2,x_3)
            \\
            &=
            \int_{\reals^d\times\reals^d\times\reals^d}\phi(x_1,x_3)\d\hat\beta(x_1,x_2,x_3)
            \\
            &=
            \int_{\reals^d\times\reals^d}
            \phi(x_1,x_3)\d(\pushforward{(\proj_{13})}\hat\beta)(x_1,x_3).
        \end{align*}
        Since $\pushforward{(\proj_{12})}\hat\gamma$ is induced by a transport, by \cite[Lemma 10.2.8]{Ambrosio2008a}, we have
        \begin{equation*}
            \hat\beta=\pushforward{(\Id,\optMap{\mu}{\bar\mu})}\mu.
        \end{equation*}
        By~Gluing Lemma (\cref{prop:gluing} in~\cref{app:preliminaries}), since $\pushforward{(\proj_{12})}\hat\gamma=\pushforward{(\proj_{12})}\hat\beta$ is induced by a transport, we conclude $\hat\gamma=\hat\beta$. Next, we seek to prove that 
        \begin{equation}\label{eq:optimal transport discrepancy subdifferential convergence gamma hat}
        \begin{aligned}
            \lim_{n\to\infty}
            \int_{\reals^d\times\reals^d\times\reals^d}&\innerProduct{\gradient{x}c(x_1,x_2)}{x_3}\d\hat\gamma_n(x_1,x_2,x_3)
            \\
            &=
            \int_{\reals^d\times\reals^d\times\reals^d}\innerProduct{\gradient{x}c(x_1,x_2)}{x_3}\d\hat\gamma(x_1,x_2,x_3).
        \end{aligned}
        \end{equation}
        Note that the inner product is continuous, but not bounded, so the result does not follow directly from the definition of narrow convergence. 
        Nonetheless, by~\cref{prop:narrow convergence transport plan}, we have 
        \begin{equation*}
        \begin{aligned}
            \lim_{n\to\infty}
            \int_{\reals^d\times\reals^d\times\reals^d}&\innerProduct{\gradient{x}c(x_1,x_2)}{x_3}\d\hat\gamma_n(x_1,x_2,x_3)
            \\
            &=
            \lim_{n\to\infty}
            \int_{\reals^d\times\reals^d}\innerProduct{y}{x_3}
            \d(\pushforward{(\gradient{x}c(\cdot,\cdot)\times\Id)}\hat\gamma_n)(y,x_3)
            \\
            &=
            \int_{\reals^d\times\reals^d}\innerProduct{y}{x_3}
            \d(\pushforward{(\gradient{x}c(\cdot,\cdot)\times\Id)}\hat\gamma)(y,x_3)
            \\
            &=
            \int_{\reals^d\times\reals^d\times\reals^d}\innerProduct{\gradient{x}c(x_1,x_2)}{x_3}\d\hat\gamma(x_1,x_2,x_3),
        \end{aligned}
        \end{equation*}
        where we used that
        \begin{itemize}
            \item 
            if $\hat\gamma_n$ converges narrowly to $\hat\gamma$, then $\pushforward{(\gradient{x}c(\cdot,\cdot)\times\Id)}\hat\gamma_n$ converges narrowly to $\pushforward{(\gradient{x}c(\cdot,\cdot)\times\Id)}\hat\gamma$ (by continuity of $\gradient{x}c$);
            
            \item 
            the fact that
            \begin{equation*}
            \begin{aligned}
                \sup_n\int_{\reals^d\times\reals^d\times\reals^d}
                &\norm{x}^2+\norm{y}^2\d(\pushforward{(\gradient{x}c(\cdot,\cdot)\times\Id)}\hat\gamma_n)(y,x)
                \\
                &\leq 
                \sup_n\int_{\reals^d\times\reals^d\times\reals^d}
                \norm{x_3}^2+\norm{\gradient{x}c(x_1,x_2)}^2\d\hat\gamma_n(x_1,x_2,x_3)
                \\
                &=\sup_n\int_{\reals^d}\norm{x_3}^2 + M\left(1+\norm{x_1}^2+\norm{x_2}^2\right)\d\hat\gamma_n(x_1,x_2,x_3)
                \\
                \overset{\eqref{eq:optimal transport discrepancy subdifferential tight gamma}}&{<}+\infty;
            \end{aligned}
            \end{equation*}
            
            \item 
            the fact that $\pushforward{(\proj_2)}\pushforward{(\gradient{x}c(\cdot,\cdot)\times\Id)}\hat\gamma_n=\mu_n$ has uniformly integrable 2-moments since $\mu_n\weakconvergence\mu$ in $\Pp{2}{\reals^d}$ (see~\cite[Proposition 7.1.5]{Ambrosio2008a}).
        \end{itemize}
        A similar argument gives 
        \begin{equation}\label{eq:optimal transport discrepancy subdifferential convergence beta hat}
        \begin{aligned}
            \lim_{n\to\infty}
            \int_{\reals^d\times\reals^d\times\reals^d}\ifbool{compact}{&}{}\innerProduct{\gradient{x}c(x_1,x_2)}{x_3}\d\hat\beta_n(x_1,x_2,x_3)
            \ifbool{compact}{ \\ &=}{=}
            \int_{\reals^d\times\reals^d\times\reals^d}\innerProduct{\gradient{x}c(x_1,x_2)}{x_3}\d\hat\beta(x_1,x_2,x_3).
        \end{aligned}
        \end{equation}
        Thus, 
        \begin{align*}
        \liminf_{n\to\infty}&
        \frac{\cost(\mu_n)-\cost(\mu)
            -\int_{\reals^d\times\reals^d}\innerProduct{\gradient{x}c(x_1,\optMap{\mu}{\bar\mu}(x_1))}{x_2-x_1}\d\tilde\gamma(x_1,x_2)}{\varepsilon_n}
            \\
            \overset{\eqref{eq:optimal transport discrepancy subdifferential inequality 1},\eqref{eq:optimal transport discrepancy subdifferential inequality 2}}&\geq 
            \begin{aligned}[t]
            &\liminf_{n\to\infty}
            \int_{\reals^d\times\reals^d\times\reals^d}\innerProduct{\gradient{x}c(x_1,x_2)}{x_3}\d\hat\beta_n(x_1,x_2,x_3)
            \\
            &-\frac{M}{2}\varepsilon_n \int_{\reals^d\times\reals^d\times\reals^d}\norm{x_3}^2\d\hat\beta_n(x_1,x_2,x_3) \\
            &-\int_{\reals^d\times\reals^d\times\reals^d}\innerProduct{\gradient{x}c(x_1,x_2)}{x_3}\d\hat\gamma_n(x_1,x_2,x_3)
            \end{aligned}
            \\
            \overset{\eqref{eq:optimal transport discrepancy subdifferential convergence beta hat},\eqref{eq:optimal transport discrepancy subdifferential convergence gamma hat}}&\geq 
            \begin{aligned}[t]
            &\int_{\reals^d\times\reals^d\times\reals^d}\innerProduct{\gradient{x}c(x_1,x_2)}{x_3}\d\hat\beta(x_1,x_2,x_3)
            \\
            &+\liminf_{n\to\infty}-\frac{M}{2}\varepsilon_n \int_{\reals^d\times\reals^d\times\reals^d}\norm{x_3}^2\d\hat\beta_n(x_1,x_2,x_3)
            \\
            &-\int_{\reals^d\times\reals^d\times\reals^d}\innerProduct{\gradient{x}c(x_1,x_2)}{x_3}\d\hat\gamma(x_1,x_2,x_3)
            \end{aligned}
            \\
            &\geq 
            \liminf_{n\to\infty}-
            \frac{M}{2}\varepsilon_n \int_{\reals^d\times\reals^d\times\reals^d}\norm{x_3}^2\d\hat\beta_n(x_1,x_2,x_3)
            \\
            \overset{\eqref{eq:optimal transport discrepancy subdifferential tight beta}}&{\geq}
            \liminf_{n\to\infty}-
            \frac{M}{2}\varepsilon_n \left(\int_{\reals^d}\norm{x}^2\d\mu(x)+\int_{\reals^d}\norm{x}^2\d\bar\mu(x)+1\right)
            \\
            &=0,
        \end{align*}
        where we used $\hat\gamma=\hat\beta$ and superadditivity of the liminf. This proves $\gradient{x}c(\cdot,\optMap{\mu}{\bar\mu}(\cdot))\in\subdifferential{\cost}(\mu)$, and concludes the proof. 
        \qedhere 
    \end{itemize}

    \item
    If $\mu$ is absolutely continuous and $c$ is continuous, has bounded Hessian, and satisfies the twist condition, then  $\setOptimalPlans{\mu}{\bar\mu}$ contains a unique optimal transport plan induced by an optimal transport map~\cite[Theorem  10.28 and Example 10.35]{Villani2009a}. Then, the statement follows directly from (vi). 
\end{enumerate}
\end{proof}}

\newbool{showproofwassersteindistance}

\newcommand{\proofwassersteindistance}[1]{
\ifbool{#1}{\begin{proof}}{\begin{proof}[Proof of \cref{cor:wasserstein distance}]}
We prove the statements separately: 
\begin{enumerate}
    \item Since $\wassersteinDistance{2}{\cdot}{\cdot}$ is distance on $\Pp{2}{\reals^d}$, it is non-negative and it does not evaluate to $+\infty$. Thus, $\effectiveDomain{\cost}=\Pp{2}{\reals^d}$, $\cost$ is proper, and $\cost(\mu)\geq 0$ for all $\mu\in\Pp{2}{\reals^d}$.
    \item Continuity w.r.t. weak convergence in $\Pp{2}{\reals^d}$ follows directly from triangle inequality and continuity of $x^2/2$.
    \item Lower semi-continuity w.r.t. narrow convergence follows from~\cref{prop:optimal transport discrepancy} and monotonicity of $x^2/2$.
    \item It suffices to observe that $-\norm{x-y}^2$ is $-2$-convex for all $y\in\reals^d$, and deploy~\cref{prop:optimal transport discrepancy}.
    \item It suffices to observe that $\norm{x-y}^2$ is $2$-convex for all $y\in\reals^d$, and deploy~\cref{prop:optimal transport discrepancy}.
    \item It suffices to observe that $\gradient{x}\norm{x-y}^2=2(x-y)$, and deploy~\cref{thm:brenier,prop:optimal transport discrepancy}.\qedhere 
\end{enumerate}
\end{proof}}

\newbool{showproofexpectation}

\newcommand{\proofexpectation}[1]{
\ifbool{#1}{\begin{proof}}{\begin{proof}[Proof of \cref{prop:expectation}]}
We prove the statements separately: 
\begin{enumerate}
    \item It suffices to prove that $V$ has at most quadratic growth, i.e., that there exists $C>0$ so that
    \begin{equation*}
        |V(x)|\leq C(1+\norm{x}^2). 
    \end{equation*}
    Then, the statement follows from Jensen's inequality: 
    \begin{equation*}
        |\expectedValue{\mu}{V(x)}|
        \leq 
        \int_{\reals^d}|V(x)|\d\mu(x)
        \leq 
        \int_{\reals^d}C(1+\norm{x}^2)\d\mu(x)
        < 
        +\infty,
    \end{equation*}
    which also directly establishes properness. 
    To see that $V$ has at most quadratic growth, observe that by Taylor's expansion there exists $z\in\reals^d$ so that  
    \begin{align*}
        V(x)
        &=
        V(0) + \innerProduct{\gradient{x}V(0)}{x} + \frac{1}{2}\innerProduct{x}{\gradient{x}^2V(z)x}
        \\
        &\leq 
        V(0) + \norm{\gradient{x}V(0)}\norm{x} + \frac{M}{2}\norm{x}^2
        \\
        &\leq
        V(0)+\frac{M}{2}\left(\frac{\norm{\gradient{x}V(0)}}{M}+\norm{x}\right)^2
        \\
        &\leq 
        V(0)+\frac{\norm{\gradient{x}V(0)}^2}{M}+M\norm{x}^2
        \\
        &\leq 
        C(1+\norm{x}^2)
    \end{align*}
    for $C\coloneqq\max\{V(0)+\norm{\gradient{}V(0)}^2/M,M\}$.
    \item Continuity follows directly from the definition of weak convergence in $\Pp{2}{\reals^d}$ since $V$ has at most quadratic growth. 
    \item \cite[Theorem 3]{Yue2021} proves an analogous statement for upper semi-continuity; the proof of lower semi-continuity follows \emph{muta mutandis}. 
    \item See~\cite[Proposition 9.3.2]{Ambrosio2008a}. 
    \item A more sophisticated proof involving subdifferentials can be found in~\cite[Proposition 10.4.2]{Ambrosio2008a}; we provide a simplified proof for the differentiable case which follows closely~\cite[§5]{bonnet2021necessary}.
Let $\nu\in\Pp{2}{\reals}$ and $\gamma\in\setOptimalPlans{\mu}{\nu}$. Then,
\begin{align*}
    \cost(\nu)-\cost(\mu)
    &=
    \int_{\reals^d}V(x)\d\nu(x)-\int_{\reals^d}V(x)\d\mu(x)
    \\
    &=
    \int_{\reals^d\times\reals^d}V(y)\d\gamma(x,y)-\int_{\reals^d\times\reals^d}V(x)\d\gamma(x,y)
    \\
    &=
    \int_{\reals^d\times\reals^d}V(y)-V(x)\d\gamma(x,y)
    \\
    &=
    \int_{\reals^d\times\reals^d}\innerProduct{\gradient{x}V(x)}{y-x}\d\gamma(x,y) + \int_{\reals^d\times\reals^d}R_x(y-x)\,\d\gamma(x,y),
\end{align*}
where $R_x(\cdot)$ is the remainder term of Taylor's series expansion at $x$ of $V$.
Since the Hessian of $V$ is uniformly bounded, the Taylor's (multivariate) remainder formula establishes the existence of $z\in\reals^d$ so that $R_x(x-y)=\frac{1}{2}\innerProduct{x-y}{\gradient{}^2V(z)(x-y)}$.
Thus, 
\begin{equation*}
    |R_x(y-x)|\leq \frac{M}{2}\norm{y-x}^2,
\end{equation*}
and so 
\begin{align*}
    \left|\int_{\reals^d\times\reals^d}R_x(y-x)\,\d\gamma(x,y)\right|
    &\leq 
    \int_{\reals^d\times\reals^d}|R_x(y-x)|\,\d\gamma(x,y)
    \\
    &\leq 
    \int_{\reals^d\times\reals^d}\frac{M}{2}\norm{y-x}^2\,\d\gamma(x,y)
    \\
    &=\frac{M}{2}\wassersteinDistance{2}{\mu}{\nu}^2,
\end{align*}
which clearly approaches 0 as $\wassersteinDistance{2}{\mu}{\nu}\to 0$. Hence, 
\begin{equation}\label{eq:expected value proof differential}
\begin{aligned}
    \cost(\nu)-\cost(\mu)
    =
    \int_{\reals^d\times\reals^d}\innerProduct{\gradient{x}V(x)}{y-x}\d\gamma(x,y) + \onotation{\wassersteinDistance{2}{\mu}{\nu}}.
\end{aligned}
\end{equation}
By \cref{def:wasserstein differentiable},~\eqref{eq:expected value proof differential} directly gives subdifferential and superdifferential. These coincide and thus $\gradient{\mu}\cost(\mu)=\gradient{x}V$.
Finally, since $V$ has bounded Hessian, $\gradient{x}V$ is Lipschitz continuous (with Lipschitz constant, say, $L\geq 0$), and
\begin{align*}
    \int_{\reals^d}\norm{\gradient{x}V(x)}^2\d\mu(x)
    &\leq 
    \int_{\reals^d}(\norm{\gradient{x}V(x)-\gradient{x}V(0)}+\norm{\gradient{x}V(0)})^2\d\mu(x)
    \\
    &\leq 
    2\norm{\gradient{x}V(0)}^2+
    2\int_{\reals^d}\norm{\gradient{x}V(x)-\gradient{x}V(0)}^2\d\mu(x)
    \\
    &\leq
    2\norm{\gradient{x}V(0)}^2+
    2L^2\int_{\reals^d}\norm{x}^2\d\mu(x)
    \\
    &<
    +\infty. 
\end{align*} 
Hence, $\gradient{x}V\in\Lp{2}{\reals^d,\reals^d;\mu}$.
\qedhere 
\end{enumerate}
\end{proof}}

\newbool{showproofinteraction}

\newcommand{\proofinteraction}[1]{
\ifbool{#1}{\begin{proof}}{\begin{proof}[Proof of \cref{prop:interaction}]}
Since $U$ has bounded Hessian, we have $U(z)\leq C(1+\norm{z}^2)$ for some $C>0$ (see (i) in the proof of~\cref{prop:expectation}). Thus, 
\begin{equation}\label{eq:interaction well defined}
\begin{aligned}
    \int_{\reals^d}\left(\int_{\reals^d}|U(x-y)|\d\mu(y)\right)\d\mu(x)
    &\leq 
    \int_{\reals^d}\left(\int_{\reals^d}C(1+\norm{x-y}^2)\d\mu(y)\right)\d\mu(x)
    \\
    &=
    C + 2C\int_{\reals^d}\left(\int_{\reals^d}\norm{x}^2+\norm{y}^2\d\mu(y)\right)\d\mu(x)
    \\
    &=
    C + 2C\left(\int_{\reals^d}\norm{x}^2\d\mu(x)\right)^2
    \\
    &<+\infty. 
\end{aligned}
\end{equation}
Thus, by Fubini's Theorem~\cite[Theorem 8.8]{Rudin1987}, the order of integration does not matter, and $\cost$ is well defined. Also, without loss of generality, we can assume that $U$ is even. If not, $U$ can be written as the sum of an even function $U_e$ and an odd function $U_o$.
On one hand, since $U_o$ is odd,
\begin{equation*}
    \int_{\reals^d\times\reals^d}U_o(x-y)\d(\mu\productMeasure\mu)(x,y)
    =
    -\int_{\reals^d\times\reals^d}U_o(y-x)\d(\mu\productMeasure\mu)(x,y),
\end{equation*}
but on the other hand, by ``relabeling'' $x$ with $y$ and Fubini's Theorem~\cite[Theorem 8.8]{Rudin1987},
\begin{equation*}
\begin{aligned}
    \int_{\reals^d}\int_{\reals^d}U_o(x-y)\d\mu(x)\d\mu(y)
    =
    \int_{\reals^d}\int_{\reals^d}U_o(y-x)\d\mu(y)\d\mu(x)
    =
    \int_{\reals^d}\int_{\reals^d}U_o(y-x)\d\mu(x)\d\mu(y)
\end{aligned}
\end{equation*}
which implies $\int_{\reals^d\times\reals^d}U_o(x-y)\d(\mu\productMeasure\mu)(x,y)=0$. Thus, 
\begin{equation*}
\begin{aligned}
    \int_{\reals^d\times\reals^d}U(x-y)\ifbool{compact}{&}{}\d(\mu\productMeasure\mu)(x,y)
    \ifbool{compact}{\\}{}&=
    \int_{\reals^d\times\reals^d}U_e(x-y)\d(\mu\productMeasure\mu)(x,y)
    +
    \int_{\reals^d\times\reals^d}U_o(x-y)\d(\mu\productMeasure\mu)(x,y)
    \\
    &=
    \int_{\reals^d\times\reals^d}U_e(x-y)\d(\mu\productMeasure\mu)(x,y).
\end{aligned}
\end{equation*}
Then, we prove the statements separately: 
\begin{enumerate}
    \item Since $|\cost(\mu)|\leq\frac{1}{2}\int_{\reals^d\times\reals^d}|U(x-y)|\d(\mu\productMeasure\mu)(x,y)$, the statement follows from~\eqref{eq:interaction well defined}. Properness follows from $\effectiveDomain{\cost}\neq 0$.  
    \item Since $\mu_n\weakconvergence\mu_n$ implies $\mu_n\productMeasure\mu_n\weakconvergence\mu\productMeasure\mu$~\cite[Eq. 9.3.7]{Ambrosio2008a}, continuity follows from~\cref{prop:expectation}.
    \item Since narrow convergence of $\mu_n$ to $\mu$ implies narrow convergence of $\mu_n\productMeasure\mu_n$ to $\mu\productMeasure\mu$~\cite[Lemma 7.3]{Santambrogio2015}, lower semi-continuity follows from~\cref{prop:expectation}.
    \item A proof can be found in~\cite[Proposition 9.3.5]{Ambrosio2008a}; we provide a more direct proof for completeness.  Let $\mu_0,\mu_1\in\Pp{2}{\reals^d}$, and $\gamma\in\setPlans{\mu_0}{\mu_1}$. Let $\mu_t\coloneqq\pushforward{((1-t)\proj_1+t\proj_2)}\gamma$. Then, 
    \begin{align*}
        \cost(\mu_t)
        &=
        \frac{1}{2}\int_{\reals^d\times\reals^d}U(x-y)\d(\mu_t\productMeasure\d\mu_t)(x,y)
        \\
        &=
        \frac{1}{2}\int_{\reals^d\times\reals^d}\int_{\reals^d\times\reals^d}U((1-t)x+tx'-((1-t)y+y'))\d\gamma(x,x')\d\gamma(y,y')
        \\
        &\leq 
        \frac{1}{2}\int_{\reals^d\times\reals^d}\int_{\reals^d\times\reals^d}(1-t)U(x-y)+tU(x'-y')\d\gamma(x,x')\d\gamma(y,y')
        \\
        &=   
        \frac{1-t}{2}\int_{\reals^d}\int_{\reals^d}U(x-y)\d\mu_0(x)\d\mu_1(y)+\frac{t}{2}\int_{\reals^d}\int_{\reals^d}U(x'-y')\d\mu_0(x')\d\mu_1(y')
        \\
        &=
        (1-t)\cost(\mu_0) + t\cost(\mu_1).
    \end{align*}
    \item A more sophisticated proof involving subdifferentials can be found in~\cite[Theorem 10.4.11]{Ambrosio2008a} under different assumptions (in particular, convexity of $U$ and the doubling assumption~\cite[Eq. 10.4.42]{Ambrosio2008a}).
    We provide instead a proof of the smooth case under boundedness of the Hessian.
    Let $\nu\in\Pp{2}{\reals}$ and $\gamma\in\setOptimalPlans{\mu}{\nu}$. Then,
    \begin{align*}
        \cost(\nu)-\cost(\mu)
        &=
        \frac{1}{2}\int_{\reals^d\times\reals^d}U(x'-y')\d(\nu\productMeasure\nu)(x',y')-\frac{1}{2}\int_{\reals^d\times\reals^d}U(x-y)\d(\mu\productMeasure\mu)(x,y)
        \\
        &=
        \begin{aligned}[t]
        &\frac{1}{2}\int_{\reals^d\times\reals^d}\int_{\reals^d\times\reals^d}U(x'-y')\d\gamma(x,x')\d\gamma(y,y')
        \\
        &-\frac{1}{2}\int_{\reals^d\times\reals^d}\int_{\reals^d\times\reals^d}U(x-y)\d\gamma(x,x')\d\gamma(y,y')
        \end{aligned}
        \\
        &=
        \frac{1}{2}\int_{\reals^d\times\reals^d}\int_{\reals^d\times\reals^d}U(x'-y')-U(x-y)\d\gamma(x,x')\d\gamma(y,y')
        \\
        &=
        \begin{aligned}[t]
        &\frac{1}{2}\int_{\reals^d\times\reals^d}\int_{\reals^d\times\reals^d}\innerProduct{\gradient{}U(x-y)}{(x'-y')-(x-y)}\d\gamma(x,x')\d\gamma(y,y') \\
        &+ \frac{1}{2}\int_{\reals^d\times\reals^d}\int_{\reals^d\times\reals^d}R_{x-y}((x'-y')-(x-y))\,\d\gamma(x,x')\d\gamma(y,y'),
        \end{aligned}
    \end{align*}
    where $R_{x-y}(\cdot)$ is the remainder term of Taylor's series expansion of $U$ at $x-y\in\reals^d$.
    Since the Hessian of $U$ is uniformly bounded, by Taylor's (multivariate) remainder formula we have that there exists $z\in\reals^d$ so that 
    \begin{equation*}
        R_{x-y}((x'-y')-(x-y))
        =
        \frac{1}{2}\innerProduct{(x'-y')-(x-y)}{\gradient{}^2U(z)((x'-y')-(x-y))},
    \end{equation*}
    and so 
    \begin{equation*}
        |R_{x-y}((x'-y')-(x-y))|
        \leq
        \frac{M}{2}\norm{(x'-y')-(x-y)}^2
        \leq 
        M\left(\norm{x-x'}^2+\norm{y-y'}^2\right).
    \end{equation*}
    Thus, 
    \begin{align*}
        \bigg|\int_{\reals^d\times\reals^d}\int_{\reals^d\times\reals^d}&R_{x-y}((x'-y')-(x-y))\,\d\gamma(x,x')\d\gamma(y,y')\bigg|
        \\
        &\leq 
        \int_{\reals^d\times\reals^d}\int_{\reals^d\times\reals^d}|R_{x-y}((x'-y')-(x-y))|\,\d\gamma(x,x')\d\gamma(y,y')
        \\
        &\leq 
        M\int_{\reals^d\times\reals^d}\int_{\reals^d\times\reals^d}\norm{x-x'}^2+\norm{y-y'}^2\,\d\gamma(x,x')\d\gamma(y,y')
        \\
        &\leq 
        M\left(\int_{\reals^d\times\reals^d}\norm{x-x'}^2\d\gamma(x,x')+\int_{\reals^d\times\reals^d}\norm{y-y'}^2\d\gamma(y,y')\right)
        \\
        &=2M\cdot \wassersteinDistance{2}{\mu}{\nu}^2,
    \end{align*}
    which clearly approaches 0 as $\wassersteinDistance{2}{\mu}{\nu}\to 0$. Hence, 
    \begin{align*}
        \cost(\ifbool{compact}{&}{}\nu)-\cost(\mu)
        \ifbool{compact}{\\}{}&=
        \frac{1}{2}\int_{\reals^d\times\reals^d}\int_{\reals^d\times\reals^d}\innerProduct{\gradient{}U(x-y)}{(x'-y')-(x-y)}\d\gamma(x,x')\d\gamma(y,y')
        + \onotation{\wassersteinDistance{2}{\mu}{\nu}}
        \\
        &=
        \begin{aligned}[t]
        &\frac{1}{2}\int_{\reals^d\times\reals^d}\int_{\reals^d\times\reals^d}\innerProduct{\gradient{}U(x-y)}{x'-x}\d\gamma(x,x')\d\gamma(y,y')
        \\
        &-\frac{1}{2}\int_{\reals^d\times\reals^d}\int_{\reals^d\times\reals^d}\innerProduct{\gradient{}U(x-y)}{y'-y}\d\gamma(x,x')\d\gamma(y,y')
        + \onotation{\wassersteinDistance{2}{\mu}{\nu}}
        \end{aligned}
        \\
        &=
        \begin{aligned}[t]
        &\frac{1}{2}\int_{\reals^d\times\reals^d}\innerProduct{\int_{\reals^d}\gradient{}U(x-y)\d\mu(y)}{x'-x}\d\gamma(x,x')
        \\
        &+\frac{1}{2}\int_{\reals^d\times\reals^d}\innerProduct{\int_{\reals^d}-\gradient{}U(x-y)\d\mu(x)}{y'-y}\d\gamma(y,y')
        + \onotation{\wassersteinDistance{2}{\mu}{\nu}}
        \end{aligned}
        \\
        &=
        \begin{aligned}[t]
        &\frac{1}{2}\int_{\reals^d\times\reals^d}\innerProduct{\int_{\reals^d}\gradient{}U(x-y)\d\mu(y)}{x'-x}\d\gamma(x,x')
        \\
        &+\frac{1}{2}\int_{\reals^d\times\reals^d}\innerProduct{\int_{\reals^d}\gradient{}U(y-x)\d\mu(x)}{y'-y}\d\gamma(y,y')
        + \onotation{\wassersteinDistance{2}{\mu}{\nu}}
        \end{aligned}
        \\
        &=
        \begin{aligned}[t]
        &\frac{1}{2}\int_{\reals^d\times\reals^d}\innerProduct{(\gradient{}U\conv\mu)(x)}{x'-x}\d\gamma(x,x')
        \\
        &+\frac{1}{2}\int_{\reals^d\times\reals^d}\innerProduct{(\gradient{}U\conv\mu)(y)}{y'-y}\d\gamma(y,y')
        + \onotation{\wassersteinDistance{2}{\mu}{\nu}}
        \end{aligned}
        \\
        &=
        \int_{\reals^d\times\reals^d}\innerProduct{(\gradient{}U\conv\mu)(x)}{x'-x}\d\gamma(x,x')
        +
        \onotation{\wassersteinDistance{2}{\mu}{\nu}},
    \end{align*}
    where we used that the gradient of an even function is odd (i.e., $-\gradient{}U(x-y)=\gradient{}U(y-x)$). 
    By \cref{def:wasserstein differentiable}, this directly gives subdifferential and superdifferential. These coincide and thus $\gradient{\mu}\cost(\mu)=\gradient{x}U\circ\mu$.
    \qedhere 
\end{enumerate}
\end{proof}}

\newbool{showproofvariance}

\newcommand{\proofvariance}[1]{
\ifbool{#1}{\begin{proof}}{\begin{proof}[Proof of \cref{cor:variance}]}
First, observe that for a real random variable $X$ the variance can be rewritten as follows: 
\begin{align*}
    \variance{\mu}{X}
    &=\int_{\reals^d}x\transpose{x}\d\mu(x)-\expectedValue{\mu}{X}\transpose{\expectedValue{\mu}{X}}
    \\
    &=
    \frac{1}{2}\int_{\reals^d}x\transpose{x}\d\mu(x)+\frac{1}{2}\int_{\reals^d}y\transpose{y}\d\mu(y)
    -\int_{\reals^d}x\d\mu(x)\int_{\reals^d}\transpose{x}\d\mu(x)
    \\
    &=
    \begin{aligned}[t]
    \frac{1}{2}\Bigg(&\int_{\reals^d}x\transpose{x}\d\mu(x)+\int_{\reals^d}y\transpose{y}\d\mu(y) \\
    &-\int_{\reals^d}x\d\mu(x)\int_{\reals^d}\transpose{y}\d\mu(y)-\int_{\reals^d}y\d\mu(y)\int_{\reals^d}\transpose{x}\d\mu(x)\Bigg)
    \end{aligned}
    \\
    &=
    \frac{1}{2}\int_{\reals^d\times\reals^d}x\transpose{x}+y\transpose{y}-x\transpose{y}-y\transpose{x}\d(\mu\productMeasure\mu)(x,y)
    \\
    &=
    \frac{1}{2}\int_{\reals^d\times\reals^d}(x-y)\transpose{(x-y)}\d(\mu\productMeasure\mu)(x,y).
\end{align*}
In our case, 
\begin{equation*}
\begin{aligned}
    \cost(\mu)
    =
    \frac{1}{2}\int_{\reals^d\times\reals^d}\innerProduct{a}{x-y}^2\d(\mu\productMeasure\mu)(x,y).
\end{aligned}
\end{equation*}
Clearly, $U(z)=\innerProduct{a}{z}^2$ is differentiable, convex, and has bounded Hessian. So, \cref{prop:interaction} directly gives continuity w.r.t. weak convergence and convexity. To compute the gradient, we first show that $U(x+y)\leq 2(U(x)+U(y))$. Indeed, 
\begin{equation*}
\begin{aligned}
    U(x+y)
    =
    \innerProduct{a}{x+y}^2
    \leq
    \left(\innerProduct{a}{x}+\innerProduct{a}{y}\right)^2
    \leq
    2(\innerProduct{a}{x}^2+\innerProduct{a}{y}^2)
    =
    2\left(U(x)+U(y)\right).
\end{aligned}
\end{equation*}
Hence, by \cref{prop:interaction},
\begin{equation*}
\begin{aligned}
    \gradient{}\cost(\mu)(x)
    &=
    (\gradient{} U \conv \mu)(x)
    =
    2a\innerProduct{a}{\int_{\reals^d}x-y\d\mu(y)}
    =
    2a\innerProduct{a}{x-\expectedValue{\mu}{x}},
\end{aligned}
\end{equation*}
where we used $\gradient{}U(z)=2a\innerProduct{a}{z}$.
\end{proof}}

\newbool{showproofinternalenergy}

\newcommand{\proofinternalenergy}[1]{
\ifbool{#1}{\begin{proof}}{\begin{proof}[Proof of \cref{prop:internal energy}]}
All proofs are in \cite{Ambrosio2008a}. In particular: 
\begin{enumerate}
    \item The statement is an obvious consequence of the definition of $\cost$. To show that $\cost$ is proper, consider the uniform probability measure over the $d$-dimensional ball in $\reals$, for which $\cost$ is finite. 
    
    \item It follows from (ii). 
    
    \item See \cite[Lemma 9.4.3]{Ambrosio2008a}.
    
    \item See \cite[Proposition 9.3.9]{Ambrosio2008a}, by noting that $\negativePart{F}$ is trivially integrable.
    
    \item See \cite[Theorem 10.4.6]{Ambrosio2008a} for existence of a subgradient; uniqueness follows as in~\cite[Proposition 4.3]{erbar2010heat} (in the special case of the relative entropy).
    \qedhere 
\end{enumerate}
\end{proof}}

\newbool{showproofdivergence}

\newcommand{\proofdivergence}[1]{
\ifbool{#1}{\begin{proof}}{\begin{proof}[Proof of \cref{prop:divergence}]}
All proofs are in~\cite{Ambrosio2008a}. In particular: 
\begin{enumerate}
    \item The statement is an obvious consequence of the definition of $\cost$. To show that $\cost$ is proper, consider $\cost(\bar\mu)=\int_{\reals^d}F(1)\d\bar\mu(x)=F(1)<+\infty$.
    
    \item It follows from (ii). 
    
    \item See~\cite[Lemma 9.4.3]{Ambrosio2008a}.
    
    \item See~\cite[Theorem 9.4.12]{Ambrosio2008a}.
    
    \item See~\cite[Theorem 10.4.9]{Ambrosio2008a} for existence of a subgradient; uniqueness follows as in~\cite[Proposition 4.3]{erbar2010heat} (in the special case of the relative entropy).
    \qedhere 
\end{enumerate}
\end{proof}}
\newbool{showprooflemmavariations}

\newcommand{\prooflemmavariations}[1]{
\ifbool{#1}{\begin{proof}}{\begin{proof}[Proof of \cref{lemma:variations}]}
Let $\varepsilon>0$. Then, by definition of $\tangentSpace{\mu}{\Pp{2}{\reals^d}}$, there exists $\varphi_\varepsilon\in\Ccinf{\reals^d}$ such that 
\begin{equation*}
    \LpNorm{H-\gradient{}\varphi_\varepsilon}{2}{\reals^d,\reals^d;\mu}^2
    <
    \varepsilon.
\end{equation*}
Consider \eqref{eq:lemma variations} with $\psi=\varphi_\varepsilon$, then
\begin{equation}\label{eq:proof lemma variations equality 1}
\begin{aligned}
    0=-2\int_{\reals^d}\innerProduct{H}{\gradient{}\varphi_\varepsilon}\d\mu(x)
    &=
    \int_{\reals^d}\Vert H-\gradient{}\varphi_\varepsilon\Vert^2-\Vert H\Vert^2-\Vert\gradient{}\varphi_\varepsilon\Vert^2\d\mu(x)
    \\
    &=\LpNorm{H-\gradient{}\varphi_\varepsilon}{2}{\reals^d,\reals^d;\mu}^2-\LpNorm{H}{2}{\reals^d,\reals^d;\mu}^2-\LpNorm{\gradient{}\varphi_\varepsilon}{2}{\reals^d,\reals^d;\mu}^2,
\end{aligned}
\end{equation}
and thus
\begin{equation*}
\begin{aligned}
    \LpNorm{H}{2}{\reals^d,\reals^d;\mu}^2
    \overset{\eqref{eq:proof lemma variations equality 1}}&{=}
    \LpNorm{H-\gradient{}\varphi_\varepsilon}{2}{\reals^d,\reals^d;\mu}^2
    -\LpNorm{\gradient{}\varphi_\varepsilon}{2}{\reals^d,\reals^d;\mu}^2
    \\
    &\leq
    \LpNorm{H-\gradient{}\varphi_\varepsilon}{2}{\reals^d,\reals^d;\mu}^2
    \\
    &<\varepsilon.
\end{aligned}
\end{equation*}
Let $\varepsilon\to 0$ to get $H\equiv 0$.
\end{proof}}

\newbool{showproofnecessaryconditions}
\newbool{showprooffirstorderconditions}

\newcommand{\prooffirstorderconditions}[1]{
\ifbool{#1}{\begin{proof}}{\begin{proof}[Proof of \cref{thm:first order conditions}]}
Let $T=\gradient{}\psi$ for some $\psi\in \Ccinf{\reals^d}$.
By \cref{lemma:perturbation transport map}, there exists $\bar s>0$ such that $\Id+\varepsilon T$ are the optimal transport maps between $\mu^\ast$ and $\mu\coloneqq\pushforward{(\Id+\varepsilon T)}\mu^\ast$ for all $\varepsilon\in(-\bar s,\bar s)$. Moreover, observe that $\mu\in\Pp{2}{\reals^d}$ and 
\begin{equation*}
    \wassersteinDistance{2}{\mu^\ast}{\mu}=|\varepsilon|\LpNorm{T}{2}{\reals^d,\reals^d;\mu^\ast}.
\end{equation*}
Thus, by local minimality of $\mu^\ast$, for $\varepsilon$ sufficiently small it holds $\cost(\mu)-\cost(\mu^\ast)\geq 0$. So, for all $\varepsilon>0$, differentiability gives 
\begin{equation*}
\begin{aligned}
    0\leq \frac{\cost(\pushforward{(\Id+\varepsilon T)}\mu^\ast) - \cost(\mu^\ast)}{\varepsilon}
    &=
    \int_{\reals^d}\innerProduct{\gradient{\mu}\cost(\mu^\ast)(x)}{T(x)}\d\mu^\ast(x)
    +\frac{\onotation{\varepsilon}}{\varepsilon}.
\end{aligned}
\end{equation*}
Letting $\varepsilon\to 0$ we get 
\begin{equation*}
    \int_{\reals^d}\innerProduct{\gradient{\mu}\cost(\mu^\ast)(x)}{T(x)}\d\mu^\ast(x)\geq 0.
\end{equation*}
Similarly, for $\varepsilon<0$, we get 
\begin{equation*}
    \int_{\reals^d}\innerProduct{\gradient{\mu}\cost(\mu^\ast)(x)}{T(x)}\d\mu^\ast(x)\leq 0,
\end{equation*}
and thus 
\begin{equation*}
    \int_{\reals^d}\innerProduct{\gradient{\mu}\cost(\mu^\ast)(x)}{T(x)}\d\mu^\ast(x)= 0.
\end{equation*}
Finally, by \cref{lemma:variations}, we conclude $\gradient{\mu}\cost(\mu^\ast)(x)=0$, for $\mu^\ast$-almost every $x\in\reals^d$.
\end{proof}}

\newbool{showprooffirstordersufficientconditions}

\newcommand{\prooffirstordersufficientconditions}[1]{
\ifbool{#1}{\begin{proof}}{\begin{proof}[Proof of~\cref{thm:first order sufficient conditions}]}
Consider a competitor $\nu\in\Pp{2}{\reals^d}$, $\nu\neq\mu$.
Then, by \cref{prop:gradients geodesically convex},
\begin{equation*}
    \cost(\nu)-\cost(\mu)
    \geq 
    \frac{\alpha}{2}\wassersteinDistance{2}{\mu}{\nu}^2.
\end{equation*}
By non-negativity of $\alpha$, $\cost(\nu)\geq \cost(\mu)$; if additionally $\alpha>0$, $\cost(\nu)>\cost(\mu)$. Since $\nu$ is arbitrary, we conclude the proof. 
\end{proof}}

\newbool{showproofadmissiblevariation}

\newcommand{\proofadmissiblevariation}[1]{
\ifbool{#1}{\begin{proof}}{\begin{proof}[Proof of~\cref{lemma:admissibility variation}]}
Define the function
\begin{equation*}
\begin{aligned}
    \chi:
    \reals\times\reals & \to \reals \\
(\varepsilon,\delta) & \mapsto \constraint(\pushforward{(\Id+\varepsilon T+\delta H)}\mu^\ast).
\end{aligned}
\end{equation*}
By definition of $\mu^\ast$ and $K$, $\chi(0,0)=\constraint(\mu^\ast)=0$. We now prove that $\chi$ is continuously differentiable around $(0,0)$. We do so by showing that all partial derivatives exist and are continuous at all $(\varepsilon,\delta)$ sufficiently close to $(0,0)$. 
We start with the partial derivative with respect to the first variable. Let $\Delta\varepsilon>0$.  Since $T$ is the gradient of a smooth compactly supported function, we know, by \cref{lemma:perturbation transport map}, that for $\Delta\varepsilon$ sufficiently small $\Id+\Delta\varepsilon T$ is an optimal transport map from $\pushforward{(\Id+\varepsilon T+\delta H)}\mu^\ast$ to $\pushforward{(\Id+(\varepsilon+\Delta\varepsilon) T+\delta H)}\mu^\ast\in\Pp{2}{\reals^d}$ and 
\begin{equation*}
    \wassersteinDistance{2}{\pushforward{(\Id+\varepsilon T+\delta H)}\mu^\ast}{\pushforward{(\Id+(\varepsilon+\Delta\varepsilon) T+\delta H)}\mu^\ast}=\Delta\varepsilon\norm{T}_{\Lp{2}{\reals^d,\reals^d;\pushforward{(\Id+\varepsilon T+\delta H)}\mu^\ast}}.
\end{equation*}
Without loss of generality, assume $T$ is normalized; i.e., $\norm{T}_{\Lp{2}{\reals^d, \reals^d; \pushforward{(\Id+\varepsilon T+\delta H)}\mu^\ast}}=1$. 
By Wasserstein differentiability of $K$ at $\pushforward{(\Id+\varepsilon T+\delta H)}\mu^\ast$ (cf. \cref{def:wasserstein differentiable}) we have
\begin{multline*}
    \frac{\constraint(\pushforward{(\Id+(\varepsilon+\Delta\varepsilon) T+\delta H)}\mu^\ast)-\constraint(\pushforward{(\Id+\varepsilon T+\delta H)}\mu^\ast)}{\Delta\varepsilon}
    \\
    =
    \frac{1}{\Delta\varepsilon}\int_{\reals^d\times\reals^d}\innerProduct{\gradient{\mu}\constraint\pushforward{(\Id+\varepsilon T+\delta H)}\mu^\ast)(x)}{y-x}\d\gamma(x,y)+\frac{\onotation{\Delta\varepsilon}}{\Delta\varepsilon}
\end{multline*}
for any $\gamma\in\setOptimalPlans{\pushforward{(\Id+\varepsilon T+\delta H)}\mu^\ast}{\pushforward{(\Id+(\varepsilon+\Delta\varepsilon) T+\delta H)}\mu^\ast}$.  As $\Id+\Delta\varepsilon T$ is an optimal transport map, we choose $\gamma=\pushforward{(\Id,\Id+\Delta\varepsilon T)}\pushforward{(\Id+\varepsilon T+\delta H)}\mu^\ast$ to get 
\begin{multline*}
    \frac{\constraint(\pushforward{(\Id+(\varepsilon+\Delta\varepsilon) T+\delta H)}\mu^\ast)-\constraint(\pushforward{(\Id+\varepsilon T+\delta H)}\mu^\ast)}{\Delta\varepsilon}
    \\
    =
    \int_{\reals^d}\innerProduct{\gradient{\mu}\constraint(\pushforward{(\Id+\varepsilon T+\delta H)}\mu^\ast)(x)}{T(x)}\d(\pushforward{(\Id+\varepsilon T+\delta H)}\mu^\ast)(x)+\frac{\onotation{\Delta\varepsilon}}{\Delta\varepsilon},
\end{multline*}
which proves that the limit $\Delta\varepsilon\to 0$ exists (and, by definition, equals the partial derivative). Hence, 
\begin{equation*}
\begin{aligned}
    \left.\diff{\chi}{\varepsilon}\right|_{(\varepsilon,\delta)}
    &=
    \lim_{\Delta\varepsilon\to 0}\int_{\reals^d}\innerProduct{\gradient{\mu}\constraint\pushforward{(\Id+\varepsilon T+\delta H)}\mu^\ast)(x)}{T(x)}\d(\pushforward{(\Id+\varepsilon T+\delta H)}\mu^\ast)(x)+\frac{\onotation{\Delta\varepsilon}}{\Delta\varepsilon}
    \\
    &=
    \int_{\reals^d}\innerProduct{\gradient{\mu}\constraint(\pushforward{(\Id+\varepsilon T+\delta H)}\mu^\ast)(x)}{T(x)}\d(\pushforward{(\Id+\varepsilon T+\delta H)}\mu^\ast)(x),
\end{aligned}
\end{equation*}
which is well-defined since $T$ has compact support and Wasserstein gradients belong to the tangent space $\tangentSpace{\mu}\Pp{2}{\reals^d}$.
Analogously, we obtain that
\begin{equation*}
    \left.\diff{\chi}{\delta}\right|_{(\varepsilon,\delta)}
    =
    \int_{\reals^d}\innerProduct{\gradient{\mu}\constraint(\pushforward{(\Id+\varepsilon T+\delta H)}\mu^\ast)(x)}{H(x)}\d(\pushforward{(\Id+\varepsilon T+\delta H)}\mu^\ast)(x).
\end{equation*}
By assumption, partial derivatives are continuous at all $(\varepsilon,\delta)$ sufficiently close to $(0,0)$. Thus, $\chi$ is continuously differentiable in a neighborhood of $(0,0)$. 
Also, by assumption on $H$,  
\begin{equation*}
\begin{aligned}
    \left.\diff{\chi}{\delta}\right|_{(0,0)}
    =\int_{\reals^d}\innerProduct{\gradient{\mu}\constraint(\mu^\ast)(x)}{H(x)}\d\mu^\ast(x)
    \neq 0.
\end{aligned}
\end{equation*}
Thus, by the implicit function theorem~\cite[Theorem 1.3.1]{krantz2002implicit}, there exists $\bar\varepsilon>0$ and a continuously differentiable function $\sigma:(-\bar\varepsilon,\bar\varepsilon)\to\reals$ such that $\sigma(0)=0$ and 
\begin{equation*}
    K(\pushforward{(\Id+\varepsilon T+\sigma(\varepsilon)H)}\mu^\ast)=0
\end{equation*}
for all $\varepsilon\in(-\bar\varepsilon,\bar\varepsilon)$.

\end{proof}}

\newbool{showprooffcontinuousdifferentiability}

\newcommand{\prooffcontinuousdifferentiability}[1]{
\ifbool{#1}{\begin{proof}}{\begin{proof}[Proof of~\cref{prop:continuous differentiability functionals}]}\,
\begin{enumerate}
    \item
    We start with Condition (i) of~\cref{def:continous differentiability}.
    Since for $\varepsilon_1$ and $\varepsilon_2$ sufficiently small, the map $\Id+\varepsilon_1T_1+\varepsilon_2T_2$ is invertible (cf.~\cref{lemma:perturbation transport map}) and the pushforward of an absolutely continuous measure through an invertible map remains absolutely continuous~\cite[Lemma 5.5.3]{Ambrosio2008a}, $\mu_{\varepsilon_1,\varepsilon_2}\coloneqq \pushforward{(\Id+\varepsilon_1T_1+\varepsilon_2T_2)}\mu$ is absolutely continuous. By~\cref{prop:optimal transport discrepancy}, $\cost$ is differentiable at $\mu_{\varepsilon_1,\varepsilon_2}$ and so Condition (i) holds. For Condition (ii), we need to prove that
    \begin{align*}
        (\varepsilon_1,\varepsilon_2)\mapsto
        h_1(\varepsilon_1,\varepsilon_2)
        &\coloneqq \int_{\reals^d}\innerProduct{\gradient{x}c(x,\optMap{\mu_{\varepsilon_1,\varepsilon_2}}{\bar\mu}(x))}{T_1(x)}\d\mu_{\varepsilon_1,\varepsilon_2}(x)
        \\
        &=
        \int_{\reals^d}\innerProduct{\gradient{x}c(x,y)}{T_1(x)}\d\gamma_{\varepsilon_1,\varepsilon_2}(x,y),
    \end{align*}
    where $\gamma_{\varepsilon_1,\varepsilon_2}\in\setOptimalPlans{\mu_{\varepsilon_1,\varepsilon_2}}{\bar\mu}$, is continuous at $(\varepsilon_1,\varepsilon_2)$ sufficiently close to 0. To start, observe that
    \begin{align*}
        \wassersteinDistance{2}{\mu_{\varepsilon_1,\varepsilon_2}}{\mu_{\varepsilon_1',\varepsilon_2'}}^2
        &\leq
        \int_{\reals^d}\norm{x+\varepsilon_1T_1(x)+\varepsilon_2T_2(x)-x-\varepsilon_1'T_1(x)-\varepsilon_2'T_2(x)}^2\d\mu(x)
        \\
        &\leq 
        2(\varepsilon_1-\varepsilon_1')^2\LpNorm{T_1}{2}{\reals^d,\reals^d;\mu}^2+2(\varepsilon_2-\varepsilon_2')^2\LpNorm{T_2}{2}{\reals^d,\reals^d;\mu}^2
    \end{align*}
    and so the map $(\varepsilon_1,\varepsilon_2)\mapsto\mu_{\varepsilon_1,\varepsilon_2}$ is continuous w.r.t. weak convergence in $\Pp{2}{\reals^d}$ (and so also continuous w.r.t. narrow convergence).
    By standard stability results of the optimal transport problem~\cite[Theorem 5.20]{Villani2009a}, we therefore have that the optimal transport plans $\gamma_{\varepsilon_1',\varepsilon_2'}\in\setOptimalPlans{\mu_{\varepsilon_1',\varepsilon_2'}}{\bar\mu}$ converge narrowly to the optimal transport plan $\gamma_{\varepsilon_1,\varepsilon_2}\in\setOptimalPlans{\mu_{\varepsilon_1,\varepsilon_2}}{\bar\mu}$ as $(\varepsilon_1',\varepsilon_2')\to(\varepsilon_1,\varepsilon_2)$. We now prove that this convergence is also weak in $\Pp{2}{\reals^d}$, by showing convergence of the second moment (see~\cite[Remark 7.1.6]{Ambrosio2008a} and~\cite[Definition 6.8]{Villani2009a}):
    \begin{align*}
        \lim_{(\varepsilon_1',\varepsilon_2')\to(\varepsilon_1,\varepsilon_2)}
        \int_{\reals^d\times\reals^d}
        &\norm{x}^2+\norm{y}^2\d\gamma_{\varepsilon_1',\varepsilon_2'}(x,y)
        \\
        &=
        \lim_{(\varepsilon_1',\varepsilon_2')\to(\varepsilon_1,\varepsilon_2)}
        \int_{\reals^d}\norm{x}^2+\norm{\optMap{\mu_{\varepsilon_1',\varepsilon_2'}}{\bar\mu}(x)}^2\d\mu_{\varepsilon_1',\varepsilon_2'}(x)
        \\
        \overset{\eqref{eq:integral and pushforward}}&{=}
        \lim_{(\varepsilon_1',\varepsilon_2')\to(\varepsilon_1,\varepsilon_2)}
        \int_{\reals^d}\norm{x}^2\d\mu_{\varepsilon_1',\varepsilon_2'}(x)+\int_{\reals^d}\norm{y}^2\d\bar\mu(y)
        \\
        \overset{\diamondsuit}&{=}
        \int_{\reals^d}\norm{x}^2\d\mu_{\varepsilon_1,\varepsilon_2}(x)+\int_{\reals^d}\norm{y}^2\d\bar\mu(y)
        \\
        \overset{\eqref{eq:integral and pushforward}}&{=}
        \int_{\reals^d}\norm{x}^2+\norm{\optMap{\mu_{\varepsilon_1,\varepsilon_2}}{\bar\mu}(x)}^2\d\mu_{\varepsilon_1,\varepsilon_2}(x)
        \\
        &=
        \int_{\reals^d\times\reals^d}
        \norm{x}^2+\norm{y}^2\d\gamma_{\varepsilon_1,\varepsilon_2}(x,y),
    \end{align*}
    where $\diamondsuit$ follows from weak convergence in $\Pp{2}{\reals^d}$ of $\mu_{\varepsilon_1',\varepsilon_2'}$ to $\mu_{\varepsilon_1,\varepsilon_2}$. 
    Then,
    \begin{align*}
        \lim_{(\varepsilon_1',\varepsilon_2')\to(\varepsilon_1,\varepsilon_2)}
        h_1(\varepsilon_1',\varepsilon_2')
        &=
        \lim_{(\varepsilon_1',\varepsilon_2')\to(\varepsilon_1,\varepsilon_2)}
        \int_{\reals^d}\innerProduct{\gradient{x}c(x,y)}{T_1(x)}\d\gamma_{\varepsilon_1',\varepsilon_2'}(x,y)
        \\
        &=
        \int_{\reals^d}\innerProduct{\gradient{x}c(x,y)}{T_1(x)}\d\gamma_{\varepsilon_1,\varepsilon_2}(x,y)
        \\
        &=
        h_1(\varepsilon_1,\varepsilon_2),
    \end{align*}
    where we used that $\innerProduct{\gradient{x}c(x,y)}{T_1(x)}$ hat at most quadratic growth (since, by assumption, $T_1$ has compact support and $\norm{\gradient{x}c(x,y)}^2\leq M(1+\norm{x}^2+\norm{y}^2)$), which establishes continuity of $h_1$. Since the case of $h_2$ is analogous, Condition (ii) holds true.

    \item
    Condition (i) of~\cref{def:continous differentiability} holds trivially. For Condition (ii) it suffices to observe that the maps $(\varepsilon_1,\varepsilon_2)\mapsto\mu_{\varepsilon_1,\varepsilon_2}$ and $\mu\mapsto\int_{\reals^d}\innerProduct{\gradient{}V(x)}{T_i(x)}\d\mu(x)$ are continuous w.r.t. weak convergence in $\Pp{2}{\reals^d}$ and so their composition is continuous too. Continuity of $(\varepsilon_1,\varepsilon_2)\mapsto\mu_{\varepsilon_1,\varepsilon_2}$ follows from (i); continuity of $\mu\mapsto\int_{\reals^d}\innerProduct{\gradient{}V(x)}{T_i(x)}\d\mu(x)$, instead, follows from continuity of all functions and the compact support of $T_i$.

    \item 
    Again, Condition (i) of~\cref{def:continous differentiability} holds trivially. Condition (ii) follows from the observation that, by Fubini's theorem,  
    \begin{equation*}
        \int_{\reals^d}\innerProduct{(\gradient{}U\conv\mu)(x)}{T_i(x)}\d\mu(x)
        =
        \int_{\reals^d\times\reals^d}\innerProduct{\gradient{}U(x-y)}{T_i(x)}d(\mu\productMeasure\mu)(x,y).
    \end{equation*}
    Continuity follows then as in (ii), with the additional considerations that $\mu\mapsto\mu\productMeasure\mu$ is also continuous (since $\wassersteinDistance{2}{\mu\productMeasure\mu}{\nu\productMeasure\nu}^2\leq 2\wassersteinDistance{2}{\mu}{\nu}^2$) and the fact that $(x,y)\mapsto\innerProduct{\gradient{}U(x-y)}{T_i(x)}$ has at most quadratic growth (by boundedness of the Hessian of $U$ and compact support of $T_i$).
\end{enumerate}
\end{proof}}

\newbool{showprooffirstordernecessaryconditionsconstrained}

\newcommand{\prooffirstordernecessaryconditionsconstrained}[1]{
\ifbool{#1}{\begin{proof}}{\begin{proof}[Proof of~\cref{thm:wasserstein lagrange multipliers}]}
The proof goes by variations. As in the proof of \cref{thm:first order conditions}, we will consider variations via gradients of compactly supported smooth functions; we will then deploy \cref{lemma:admissibility variation} to make sure that variations are admissible (i.e., they satisfy the constraint).

To start, we argue that we can without loss of generality assume that $H=\gradient{}\varphi$ for some $\varphi\in\Ccinf{\reals^d}$. Indeed, let $H\in\tangentSpace{\mu^\ast}\Pp{2}{\reals^d}$ so that $\beta\coloneqq\int_{\reals^d}\innerProduct{\gradient{\mu}\constraint(\mu^\ast)(x)}{H(x)}\d\mu^\ast(x)\neq 0$. Up to replacing $H$ with $-H$ we can assume $\beta>0$. 
By density, let $\varphi\in\Ccinf{\reals^d}$ so that $\LpNorm{H-\gradient{}\varphi}{2}{\reals^d,\reals^d;\mu^\ast}<\beta/\LpNorm{\gradient{\mu}K(\mu^\ast)}{2}{\reals^d,\reals^d;\mu^\ast}$. Then, linearity and the Cauchy-Schwarz inequality~\cite[Theorem 4.2]{Rudin1987} give
\begin{align*}
    \int_{\reals^d}&\innerProduct{\gradient{\mu}\constraint(\mu^\ast)(x)}{\gradient{}\varphi(x)}\d\mu^\ast(x)
    \\
    &=
    \int_{\reals^d}\innerProduct{\gradient{\mu}\constraint(\mu^\ast)(x)}{H(x)}\d\mu^\ast(x)
    +
    \int_{\reals^d}\innerProduct{\gradient{\mu}\constraint(\mu^\ast)(x)}{\gradient{}\varphi(x)-H(x)}\d\mu^\ast(x)
    \\
    &\geq 
    \int_{\reals^d}\innerProduct{\gradient{\mu}\constraint(\mu^\ast)(x)}{H(x)}\d\mu^\ast(x)
    -
    \LpNorm{\gradient{\mu}K(\mu^\ast)}{2}{\reals^d,\reals^d;\mu^\ast}
    \LpNorm{\gradient{}\varphi-H}{2}{\reals^d,\reals^d;\mu^\ast}
    \\
    &>
    \beta-\beta=0.
\end{align*}

We can now proceed with the variational analysis. Let $T=\gradient{}\psi$ for some $\psi\in\Ccinf{\reals^d}$.
The variation $\pushforward{(\Id+\varepsilon T)}\mu^\ast$ is \emph{not} admissible, as generally $K(\pushforward{(\Id+\varepsilon T)}\mu^\ast)\neq 0$.
Yet, since by assumption $K$ is Wasserstein continuously differentiable at $\mu^\ast$, $K(\mu^\ast)=0$, and its Wasserstein gradient does not vanish at $\mu^\ast$, we can use \cref{lemma:admissibility variation} to make the variation admissible. In particular, there exists $\bar\varepsilon>0$ and $\sigma\in\C{1}{(-\bar\varepsilon,\bar\varepsilon)}$ such that the variation $\pushforward{(\Id+\varepsilon T+\sigma(\varepsilon)H)}\mu^\ast$ is admissible for all $\varepsilon\in(-\bar\varepsilon,\bar\varepsilon)$.
As $\sigma(\varepsilon)=\sigma'(0)\varepsilon + \onotation{\varepsilon}$, \cref{lemma:perturbation transport map} gives that, for $\varepsilon$ sufficiently small, $\Id+\varepsilon T+\sigma(\varepsilon)H$ is an optimal transport map from $\mu^\ast$ to $\pushforward{(\Id+\varepsilon T+\sigma(\varepsilon)H)}\mu^\ast$.
So,
\begin{equation*}
\begin{aligned}
    \wassersteinDistance{2}{\pushforward{(\Id+\varepsilon T+\sigma(\varepsilon)H)}\mu^\ast}{\mu^\ast}
    &=
    \LpNorm{\varepsilon T+\sigma(\varepsilon)H}{2}{\reals^d,\reals^d;\mu^\ast} \\
    &\leq |\varepsilon|\LpNorm{T}{2}{\reals^d,\reals^d;\mu^\ast}
    +|\sigma(\varepsilon)|\LpNorm{H}{2}{\reals^d,\reals^d;\mu^\ast}.
\end{aligned}
\end{equation*}
As $\sigma(\varepsilon)\to 0$ as $\varepsilon\to 0$, local minimality of $\mu^\ast$ with $\varepsilon>0$, together with differentiability, gives
\begin{align*}
    0
    &\leq \frac{
    \cost(\pushforward{(\Id+\varepsilon T+\sigma(\varepsilon)H)}\mu^\ast)
    -
    \cost(\mu^\ast)}{\varepsilon}
    \\
    &=
    \frac{1}{\varepsilon}
    \int_{\reals^d}\innerProduct{\gradient{\mu}\cost(\mu^\ast)(x)}{x+\varepsilon T(x)+\sigma(\varepsilon)H(x)-x}\d\mu^\ast(x)+\frac{\onotation{\varepsilon}}{\varepsilon}
    \\
    &=
    \int_{\reals^d}\innerProduct{\gradient{\mu}\cost(\mu^\ast)}{ T+\frac{\sigma(\varepsilon)}{\varepsilon}H}\d\mu^\ast+\frac{\onotation{\varepsilon}}{\varepsilon}.
\end{align*}
Let $\varepsilon\to 0$ to conclude 
\begin{equation*}
    \int_{\reals^d}\innerProduct{\gradient{\mu}\cost(\mu^\ast)}{ T+\sigma'(0)H}\d\mu^\ast\geq 0.
\end{equation*}
If we repeat the argument with $\varepsilon<0$, we get 
\begin{equation*}
    \int_{\reals^d}\innerProduct{\gradient{\mu}\cost(\mu^\ast)}{ T+\sigma'(0)H}\d\mu^\ast\leq 0,
\end{equation*}
and thus
\begin{equation}\label{eq:proof multipliers gradient cost}
\begin{aligned}
    0
    &=
    \int_{\reals^d}\innerProduct{\gradient{\mu}\cost(\mu^\ast)}{ T+\sigma'(0)H}\d\mu^\ast(x) \\
    &=
    \int_{\reals^d}\innerProduct{\gradient{\mu}\cost(\mu^\ast)}{ T}\d\mu^\ast
    +
    \sigma'(0)\int_{\reals^d}\innerProduct{\gradient{\mu}\cost(\mu^\ast)}{H}\d\mu^\ast.
\end{aligned}
\end{equation}
Moreover, for all $\varepsilon\in(-\bar\varepsilon,\bar\varepsilon)$ the function $\varepsilon\mapsto K(\pushforward{(\Id+\varepsilon T+\sigma(\varepsilon)H)}\mu)$ is zero and so
\begin{equation*}
    \frac{
    K(\pushforward{(\Id+\varepsilon T+\sigma(\varepsilon)H)}\mu^\ast)
    -
    K(\mu^\ast)}{\varepsilon}
    =
    0.
\end{equation*}
We can now proceed as above to get
\begin{equation*}
    \int_{\reals^d}\innerProduct{\gradient{\mu}\constraint(\mu^\ast)}{ T}\d\mu^\ast
    +
    \sigma'(0)\int_{\reals^d}\innerProduct{\gradient{\mu}\constraint(\mu^\ast)}{H}\d\mu^\ast
    =
    0,
\end{equation*}
which leads, as by assumption $\int_{\reals}\innerProduct{\gradient{\mu}\constraint(\mu^\ast)}{H}\d\mu^\ast\neq 0$, to 
\begin{equation*}
    \sigma'(0)=-\frac{\int_{\reals^d}\innerProduct{\gradient{\mu}\constraint(\mu^\ast)}{ T}\d\mu^\ast}{\int_{\reals^d}\innerProduct{\gradient{\mu}\constraint(\mu^\ast)}{H}\d\mu^\ast}\in\reals.
\end{equation*}
So, \eqref{eq:proof multipliers gradient cost} reads
\begin{equation}\label{eq:proof multipliers before lambda}
    \int_{\reals^d}\innerProduct{\gradient{\mu}\cost(\mu^\ast)}{T}\d\mu^\ast
    -
    \frac{\int_{\reals^d}\innerProduct{\gradient{\mu}\constraint(\mu^\ast)}{ T}\d\mu^\ast}{\int_{\reals}\innerProduct{\gradient{\mu}\constraint(\mu^\ast)}{H}\d\mu^\ast}
    \int_{\reals^d}\innerProduct{\gradient{\mu}\cost(\mu^\ast)}{H}\d\mu^\ast=0.
\end{equation}
Define the multiplier as  
\begin{equation*}
    \lambda
    \coloneqq
    -
    \frac{
    \int_{\reals^d}\innerProduct{\gradient{\mu}\cost(\mu^\ast)}{H}\d\mu^\ast
    }{
    \int_{\reals^d}\innerProduct{\gradient{\mu}\constraint(\mu^\ast)}{H}\d\mu^\ast
    }\in\reals.
\end{equation*}
Then, \eqref{eq:proof multipliers before lambda} becomes
\begin{equation*}
    \int_{\reals^d}\innerProduct{\gradient{\mu}\cost(\mu^\ast)+\lambda\gradient{\mu}\constraint(\mu^\ast)}{T}\d\mu^\ast=0.
\end{equation*}
As $T$ is arbitrary, \cref{lemma:variations} gives 
\begin{equation*}
    \gradient{\mu}\cost(\mu^\ast)+\lambda\gradient{\mu}\constraint(\mu^\ast)\equiv 0,
\end{equation*}
or equivalently $\gradient{\mu}\cost(\mu^\ast)(x)+\lambda\gradient{\mu}\constraint(\mu^\ast)(x)=0$ for $\mu^\ast$-almost every $x\in\reals^d$. Finally, to show uniqueness, assume that there exists another $\lambda'\in\reals$ so that 
\begin{equation*}
    \gradient{\mu}\cost(\mu^\ast)(x)+\lambda'\gradient{\mu}\constraint(\mu^\ast)(x)=0
\end{equation*}
for $\mu^\ast$-almost every $x\in\reals^d$. Then, $\lambda'\gradient{\mu}\constraint(\mu^\ast)=\lambda\gradient{\mu}\constraint(\mu^\ast)$, and
\begin{align*}
    \lambda'\int_{\reals^d}\innerProduct{\gradient{\mu}\constraint(\mu^\ast)(x)}{H(x)}\d\mu^\ast(x)
    =
    \lambda\int_{\reals^d}\innerProduct{\gradient{\mu}\constraint(\mu^\ast)(x)}{H(x)}\d\mu^\ast(x).
\end{align*}
Since, by assumption, $\int_{\reals^d}\innerProduct{\gradient{\mu}\constraint(\mu^\ast)(x)}{H(x)}\d\mu^\ast(x)\neq 0$, we conclude $\lambda'=\lambda$.
\end{proof}}

\newbool{showprooffirstordernecessaryconditionsconstrainedinequality}

\newcommand{\prooffirstordernecessaryconditionsconstrainedinequality}[1]{
\ifbool{#1}{\begin{proof}}{\begin{proof}[Proof of~\cref{thm:inequality:wasserstein lagrange multipliers}]}
We distinguish two cases.

Suppose initially that $\mu^\ast$ lies at the boundary of the feasible set; i.e., $\constraint(\mu^\ast)=0$. Then, $\mu^\ast$ is also a local minimizer of \eqref{eq:constrained optimization}. Thus,~\cref{thm:wasserstein lagrange multipliers} establishes existence of $\lambda\in\reals$ so that $\gradient{\mu}\cost(\mu^\ast)+\lambda\gradient{}\constraint(\mu^\ast)=0$ holds. Clearly, $\lambda K(\mu^\ast)=0$. Thus, we only need to prove $\lambda\geq 0$.
Suppose $\lambda<0$. For $\varepsilon>0$, consider the probability measure $\mu_\varepsilon\coloneqq\pushforward{(\Id-\varepsilon H)}\mu^\ast$, where $H\in\tangentSpace{\mu}\Pp{2}{\reals^d}$ is as in the statement. As in the proof of~\cref{thm:wasserstein lagrange multipliers}, we can without loss of generality assume that $H=\gradient{}\varphi$ for $\varphi\in\Ccinf{\reals}$ and, up to replacing $H$ with $-H$ and rescaling, that $\LpNorm{H}{2}{\reals^d,\reals^d;\mu^\ast}=1$ and 
\begin{equation*}
    \beta\coloneqq\int_{\reals^d}\innerProduct{\gradient{\mu}\constraint(\mu^\ast)(x)}{H(x)}\d\mu^\ast(x)>0.
\end{equation*}
By~\cref{lemma:perturbation transport map}, for $\varepsilon$ sufficiently small $\Id-\varepsilon H$ is an optimal transport map between $\mu^\ast$ and $\mu_\varepsilon$ and, in particular, $\wassersteinDistance{2}{\mu^\ast}{\mu_\varepsilon}=\varepsilon$.
By Wasserstein differentiability of $\constraint$ at $\mu^\ast$, we then have
\begin{align*}
    \constraint(\mu_\varepsilon)
    &=\constraint(\mu^\ast)
    +
    \int_{\reals^d}\innerProduct{\gradient{\mu}\constraint(\mu^\ast)(x)}{x-\varepsilon H(x)-x}\d\mu^\ast(x)+\onotation{\varepsilon}
    =
    -\varepsilon\beta+\onotation{\varepsilon}.
\end{align*}
Thus, for $\varepsilon$ sufficiently small, $\constraint(\mu_\varepsilon)<0$ and $\mu_\varepsilon$ is feasible for~\eqref{eq:inequality constrained optimization}. We now show that, if $\lambda<0$, $\mu_\varepsilon$ achieves a lower cost compared to $\mu^\ast$, which yields the desired contradiction.  
Wasserstein differentiability of $\cost$ gives
\begin{align*}
    \cost(\mu_\varepsilon)
    &=
    \cost(\mu^\ast)
    +
    \int_{\reals^d}\innerProduct{\gradient{\mu}\cost(\mu^\ast)(x)}{x-\varepsilon H(x)-x}\d\mu^\ast(x)+\onotation{\varepsilon}
    \\
    \overset{\diamondsuit}&{=}
    \cost(\mu^\ast)
    +
    \int_{\reals^d}\innerProduct{-\lambda\gradient{\mu}\constraint(\mu^\ast)(x)}{-\varepsilon H(x)}\d\mu^\ast(x)+\onotation{\varepsilon}
    \\
    \overset{\diamondsuit}&{=}
    \cost(\mu^\ast)
    +\lambda\varepsilon\beta+\onotation{\varepsilon},
\end{align*}
where in $\diamondsuit$ we used that at optimality $\gradient{\mu}\cost(\mu^\ast)+\lambda\gradient{}\constraint(\mu^\ast)=0$. If $\lambda<0$, for all $\varepsilon$ sufficiently small, we conclude $\cost(\mu_\varepsilon)<\cost(\mu^\ast)$. This contradicts local minimality of $\mu^\ast$ and shows that it necessarily holds $\lambda\geq 0$.

Suppose now that $\mu^\ast$ resides in the interior of the feasible set; i.e., $\constraint(\mu^\ast)<0$. Then, by continuity of $\constraint$ at $\mu^\ast$ (cf.~\cref{prop:differentiable continuous}), $\mu^\ast$ is also a local minimizer of the unconstrained problem. Accordingly,~\cref{thm:first order conditions} gives $\gradient{\mu}\cost(\mu^\ast)=0$. In this case, we can therefore set $\lambda=0$, which trivially yields $\gradient{\mu}\cost(\mu^\ast)+\lambda\gradient{\mu}\constraint(\mu^\ast)=0$ and $\lambda\constraint(\mu^\ast)=0$.
\end{proof}}

\newbool{showprooffirstordersufficientconditionsconstrained}

\newcommand{\prooffirstordersufficientconditionsconstrained}[1]{
\ifbool{#1}{\begin{proof}}{\begin{proof}[Proof of~\cref{thm:first order sufficient conditions constrained}]}
Without loss of generality, $\lambda>0$; else, the result follows directly from \cref{thm:first order sufficient conditions}.
Let $\alpha_{\cost}\geq 0$ be the convexity parameter of $\cost$ and $\alpha_{\constraint}\geq 0$ be the convexity parameter of $\constraint$, then $\cost+\lambda\constraint$ is $(\alpha_{\cost}+\lambda\alpha_{\constraint})$-geodesically convex.
By \cref{thm:first order sufficient conditions}, $\mu^\ast$ is a (global) minimizer of $\mu\mapsto\cost(\mu)+\lambda K(\mu)$.
Then, let $\nu\in\Pp{2}{\reals^d}$ such that $\nu\neq\mu^\ast$ and $\constraint(\nu)\leq 0$:  
\begin{equation*}
\begin{aligned}
    \cost(\nu)
    &\geq
    \cost(\nu)+\lambda\constraint(\nu)
    \\
    &\geq 
    \cost(\mu^\ast)+\lambda\constraint(\mu^\ast) + \frac{\alpha_{\cost}+\lambda\alpha_{\constraint}}{2}\wassersteinDistance{2}{\mu^\ast}{\nu}^2
    \\
    &=\cost(\mu^\ast) + \frac{\alpha_{\cost}+\lambda\alpha_{\constraint}}{2}\wassersteinDistance{2}{\mu^\ast}{\nu}^2,
\end{aligned}
\end{equation*}
where the second inequality follows $(\alpha_{\cost}+\lambda\alpha_{\constraint})$-geodesically convexity, existence of a subgradient $\xi\in\subdifferential\cost(\mu^\ast)$ so that $\xi+\lambda\gradient{\mu}\constraint(\mu^\ast)=0$, and~\cref{prop:gradients geodesically convex}.
Since $\nu$ is arbitrary, $\mu^\ast$ is a global minimizer of~\eqref{eq:inequality constrained optimization}. If additionally $\alpha_{\cost}+\lambda\alpha_{\constraint}>0$, $\mu^\ast$ is the strict minimizer.
\end{proof}}

\newbool{showproofcorfirstordersufficientconditionsconstrained}

\newcommand{\proofcorfirstordersufficientconditionsconstrained}[1]{
\ifbool{#1}{\begin{proof}}{\begin{proof}[Proof of~\cref{cor:first order sufficient conditions constrained}]}
The proof follows directly from \cref{thm:first order sufficient conditions constrained} with $\mu\mapsto\constraint(\mu)$ and $\mu\mapsto-\constraint(\mu)$.
\end{proof}}

\newbool{showprooffirstordersufficientconditionsconstrainedwasserstein}

\newcommand{\prooffirstordersufficientconditionsconstrainedwasserstein}[1]{
\ifbool{#1}{\begin{proof}}{\begin{proof}[Proof of~\cref{thm:first order sufficient conditions constrained wasserstein}]}
By~\cref{prop:optimal transport discrepancy}, convexity along interpolation curves of $\cost$, and non-negativity of $\lambda$, the function $\mathcal{L}(\mu)=\cost(\mu)+\lambda\optimalTransportDiscrepancy{c}{\bar\mu}{\mu}$ is $(\lambda\alpha_c+\alpha_J)$-convex along the interpolation curve between $\mu^\ast$ and $\nu$ ``centered'' at $\bar\mu$, defined by
\begin{equation*}
    \mu_t
    \coloneqq
    \pushforward{((1-t)\proj_2+ t\proj_3)}\gamma,
\end{equation*}
with 
\begin{equation*}
    \gamma\in\Pp{2}{\reals^d\times\reals^d\times\reals^d},
    \qquad
    \pushforward{(\proj_{12})}\gamma\in\setOptimalPlans{\bar\mu}{\mu^\ast},
    \qquad 
    \pushforward{(\proj_{13})}\gamma\in\setOptimalPlans{\bar\mu}{\nu},
\end{equation*}
where optimality is intended w.r.t. transport cost $c$. 
Note that $\pushforward{(\proj_{23})}\gamma\in\setPlans{\mu^\ast}{\nu}$ and
\begin{equation*}
    \pushforward{(\proj_2,(1-t)\proj_2+ t\proj_3)}\gamma\in\setPlans{\mu^\ast}{\mu_t},
\end{equation*}
where both plans are generally not optimal between their marginals.
Moreover, by~\cref{prop:sum and multiplication rule}, $\mathcal{L}$ is subdifferentiable at $\mu^\ast$ and $\subdifferential\mathcal{L}(\mu^\ast)=\subdifferential\cost(\mu^\ast)+\lambda\gradient{x}c(\cdot,\optMap{\mu^\ast}{\bar\mu}(x))$. Let $\zeta\in\subdifferential\mathcal{L}(\mu^\ast)$. Then, 
\begin{equation}\label{eq:first order sufficient conditions constrained wasserstein strong differential 1}
\begin{aligned}
    \int_{\reals^d\times\reals^d}&\innerProduct{\zeta(x)}{y-x}\d(\pushforward{(\proj_2,(1-t)\proj_2+ t\proj_3)}\gamma)(x,y)
    \\
    &=
    \int_{\reals^d\times\reals^d}\innerProduct{\zeta(x)}{(1-t)x+ty-x}\d(\pushforward{(\proj_2,\proj_3)}\gamma)(x,y)
    \\
    &=
    t\int_{\reals^d\times\reals^d}\innerProduct{\zeta(x)}{y-x}\d(\pushforward{(\proj_2,\proj_3)}\gamma)(x,y)
\end{aligned}
\end{equation}
and
\begin{equation}\label{eq:first order sufficient conditions constrained wasserstein strong differential 2}
\begin{aligned}
    \int_{\reals^d\times\reals^d}&\norm{y-x}^2\d(\pushforward{(\proj_2,(1-t)\proj_2+ t\proj_3)}\gamma)(x,y)
    \\
    &=
    \int_{\reals^d\times\reals^d}\norm{(1-t)x+ty-x}^2\d(\pushforward{(\proj_2,\proj_3)}\gamma)(x,y)
    \\
    &=
    t^2\int_{\reals^d\times\reals^d}\norm{y-x}^2\d(\pushforward{(\proj_2,\proj_3)}\gamma)(x,y).
\end{aligned}
\end{equation}
Then, \cref{prop:strong differentiability} gives 
\begin{equation}\label{eq:first order sufficient conditions constrained wasserstein strong differential}
\begin{aligned}
    \mathcal{L}(\mu_t)-\mathcal{L}(\mu^\ast)
    &\geq
    \begin{aligned}[t]
    &\int_{\reals^d\times\reals^d}\innerProduct{\zeta(x)}{y-x}\d(\pushforward{(\proj_2,(1-t)\proj_2+ t\proj_3)}\gamma)(x,y)
    \\
    &+\onotation[adapt]{\sqrt{\int_{\reals^d\times\reals^d}\norm{y-x}^2\d(\pushforward{(\proj_2,(1-t)\proj_2+ t\proj_3)}\gamma)(x,y)}}
    \end{aligned}
    \\
    \overset{\eqref{eq:first order sufficient conditions constrained wasserstein strong differential 1},\eqref{eq:first order sufficient conditions constrained wasserstein strong differential 2}}&{=}
    t\int_{\reals^d\times\reals^d}\innerProduct{\zeta(x)}{y-x}\d(\pushforward{(\proj_2,\proj_3)}\gamma)(x,y) + \onotation{t}.
\end{aligned}
\end{equation}
Since $\mathcal{L}$ is $(\alpha_J+\lambda\alpha_c)$-convex along the interpolation curve $\mu_t$, for all $t\in(0,1)$ we have 
\begin{equation*}
\begin{aligned}
    \mathcal{L}(\mu_t)
    &\leq 
    \begin{aligned}[t]
    &(1-t)\mathcal{L}(\mu^\ast)+t\mathcal{L} (\nu)
    -\frac{\alpha_J+\lambda\alpha_c}{2}t(1-t)\int_{\reals^d\times\reals^d}\norm{y-x}^2\d(\pushforward{(\proj_2,\proj_3)}\gamma)(x,y)
    \end{aligned}
    \\
    &\leq
    (1-t)\mathcal{L}(\mu^\ast)+t\mathcal{L} (\nu)-\frac{\alpha_J+\lambda\alpha_c}{2}t(1-t)\wassersteinDistance{2}{\mu^\ast}{\nu}^2,
\end{aligned}
\end{equation*}
where the second inequality follows from $\lambda\alpha_c+\alpha_J\geq 0$ and suboptimality of the transport plan $\pushforward{(\proj_2,\proj_3)}\gamma\in\setPlans{\mu^\ast}{\nu}$. Hence,
\begin{equation}\label{eq:first order sufficient conditions constrained wasserstein convexity}
    \frac{\mathcal{L}(\mu_t)-\mathcal{L}(\mu^\ast)}{t}\leq\mathcal{L}(\nu)-\mathcal{L}(\mu^\ast)-\frac{\alpha_J+\lambda\alpha_c}{2}(1-t)\wassersteinDistance{2}{\mu^\ast}{\nu}^2.    
\end{equation}
Thus,
\begin{align*}
    \mathcal{L}(\nu)-\mathcal{L}(\mu^\ast)-&\frac{\alpha_J+\lambda\alpha_c}{2}\wassersteinDistance{2}{\mu^\ast}{\nu}^2
    \\
    &=
    \liminf_{t\downarrow 0} \mathcal{L}(\nu)-\mathcal{L}(\mu^\ast)-\frac{\alpha_J+\lambda\alpha_c}{2}(1-t)\wassersteinDistance{2}{\mu^\ast}{\nu}^2
    \\
    \overset{\eqref{eq:first order sufficient conditions constrained wasserstein convexity}}&{\geq}
    \liminf_{t\downarrow 0} \frac{\mathcal{L}(\mu_t)-\mathcal{L}(\mu^\ast)}{t}
    \\
    \overset{\eqref{eq:first order sufficient conditions constrained wasserstein strong differential}}&{\geq}
    \liminf_{t\downarrow 0}
    \frac{
    t\int_{\reals^d\times\reals^d}\innerProduct{\zeta(x)}{y-x}\d(\pushforward{(\proj_2,\proj_3))}\gamma)(x,y) + \onotation{t}}{t}
    \\
    &=
    \int_{\reals^d\times\reals^d}\innerProduct{\zeta(x)}{y-x}\d(\pushforward{(\proj_2,\proj_3))}\gamma)(x,y).
\end{align*}
By assumption, there exists $\xi\in\subdifferential J(\mu^\ast)$ so that $\xi+\lambda\gradient{x}c(\cdot,\optMap{\mu}{\bar\mu}(\cdot))=0$. That is, $\zeta=0$ lies in the subdifferential of $\mathcal{L}$ at $\mu^\ast$, which yields
\begin{equation}\label{eq:first order sufficient conditions constrained wasserstein convexity lagrangian final}
    \mathcal{L}(\nu)-\mathcal{L}(\mu^\ast)\geq\frac{\alpha_J+\lambda\alpha_c}{2}\wassersteinDistance{2}{\mu^\ast}{\nu}^2.
\end{equation}
Suppose now that $\nu$ is feasible; i.e.,  $\optimalTransportDiscrepancy{c}{\nu}{\bar\mu}\leq\varepsilon$. Then,
\begin{equation*}
\begin{aligned}
    \cost(\nu)
    &\geq
    \cost(\nu)+\lambda\left(\optimalTransportDiscrepancy{c}{\nu}{\bar\mu}-\varepsilon\right)
    \\
    \overset{\eqref{eq:first order sufficient conditions constrained wasserstein convexity lagrangian final}}&{\geq}
    \cost(\mu^\ast)+\lambda\left(\optimalTransportDiscrepancy{c}{\mu^\ast}{\bar\mu}-\varepsilon\right)+\frac{\alpha_J+\lambda\alpha_c}{2}\wassersteinDistance{2}{\mu^\ast}{\nu}^2
    \\
    &=\cost(\mu^\ast)+\frac{\alpha_J+\lambda\alpha_c}{2}\wassersteinDistance{2}{\mu^\ast}{\nu}^2.
\end{aligned}
\end{equation*}
Since $\nu$ is arbitrary, $\mu^\ast$ is a global minimizer. If $\alpha_J+\lambda\alpha_c>0$, then $\mu^\ast$ is the strict minimizer.
\end{proof}}

\newbool{showprooffirstordersufficientconditionsconstrainedwassersteincorollary}

\newcommand{\prooffirstordersufficientconditionsconstrainedwassersteincorollary}[1]{
\ifbool{#1}{\begin{proof}}{\begin{proof}[Proof of~\cref{cor:first order sufficient conditions constrained wasserstein}]}
In the case of the squared Wasserstein distance, we have $c(x,y)=\norm{x-y}^2$, and thus $\alpha_c=2$. By~\cref{cor:wasserstein distance} and convexity along generalized geodesics of $\cost$, the function $\mathcal{L}(\mu)=\cost(\mu)+\lambda\wassersteinDistance{2}{\bar\mu}{\mu}^2$ is $(2\lambda+\alpha)$-convex along the generalized geodesic curve between $\mu^\ast$ and $\nu$ ``centered'' at $\bar\mu$. Then, the proof is identical to the proof of~\cref{thm:first order sufficient conditions constrained wasserstein}.
\end{proof}}

\newbool{showproofabsolutelycontinuityref}

\newcommand{\proofabsolutelycontinuityref}[1]{
\ifbool{#1}{\begin{proof}}{\begin{proof}[Proof of~\cref{prop:absolutely continuity ref}]}
Existence of $\refmu'_\delta$ is ensured by density (w.r.t. the Wasserstein distance) of $\Ppabs{2}{\reals^d}$ in $\Pp{2}{\reals^d}$~\cite[Theorem 2.2.7]{Panaretos2020AnSpace}. Then, the result follows from triangle inequality.
\end{proof}}

\newbool{showproofwassersteinballs}

\newcommand{\proofwassersteinballs}[1]{
\ifbool{#1}{\begin{proof}}{\begin{proof}[Proof of~\cref{prop:wasserstein balls}]}
We prove the two statements for completeness; alternative (but slightly more sophisticated) proofs are in~\cite{Ambrosio2008a,Yue2021}.
\begin{enumerate}
    \item Closedness follows from triangle inequality. To see that $\closedWassersteinBall{2}{\radius}{\refmu}$ is generally not compact, let $d=1$, $\refmu=\diracMeasure{0}$, and $\radius=1$. Define $(\mu_n)_{n\in\naturals}\subset\closedWassersteinBall{2}{1}{\diracMeasure{0}}$ by 
    $\mu_n\coloneqq(1-\frac{1}{n^2})\diracMeasure{0}+\frac{1}{n^2}\diracMeasure{n}$. Assume that there exists a subsequence $\mu_{k_n}\weakconvergence\mu\in\closedWassersteinBall{2}{1}{\diracMeasure{0}}$. Since $\mu_n$ converges narrowly to $\diracMeasure{0}$, it necessarily holds $\mu=\diracMeasure{0}$; else, since weak convergence in $\Pp{2}{\reals^d}$ implies narrow convergence, $\mu_{k_n}$ would converge narrowly to two limits, which is absurd (since narrow convergence is induced by a distance~\cite[Remark 5.1.1]{Ambrosio2008a} and in metric spaces limits are unique). However, by definition of weak convergence in $\Pp{2}{\reals^d}$,
    \begin{equation*}
        \wassersteinDistance{2}{\mu_{k_n}}{\diracMeasure{0}}
        =
        \int_{\reals}x^2\d\mu_{k_n}(x)=\frac{1}{k_n^2}k_n^2=1\to \int_{\reals}x^2\d\mu(x)
        =
        \wassersteinDistance{2}{\mu}{\diracMeasure{0}}.
    \end{equation*}
    That is, $\wassersteinDistance{2}{\mu}{\diracMeasure{0}}=1$.
    Since the Wasserstein distance metrizes weak convergence in $\Pp{2}{\reals^d}$, this contradicts $\mu_{k_n}\weakconvergence\diracMeasure{0}$. So, $(\mu_n)_{n\in\naturals}$ does not admit a convergent subsequence, and $\closedWassersteinBall{2}{1}{\diracMeasure{0}}$ is not compact. 
    
    \item By \cref{cor:wasserstein distance}, $\mu\mapsto\wassersteinDistance{2}{\mu}{\refmu}$ is lower semi-continuous w.r.t. narrow convergence, thus its level sets $\closedWassersteinBall{2}{\radius}{\refmu}$ are closed: Let $(\mu_n)_{n\in\naturals}\subset\closedWassersteinBall{2}{\radius}{\refmu}$ narrowly converge to $\mu\in\Pp{2}{\reals^d}$, then
    \begin{equation*}
        \wassersteinDistance{2}{\mu}{\refmu}
        \leq 
        \liminf_{n\to\infty}\wassersteinDistance{2}{\mu_n}{\refmu}
        \leq \radius,
    \end{equation*}
    and so $\mu\in\closedWassersteinBall{2}{\radius}{\refmu}$. Thus, it suffices to prove that Wasserstein balls are relatively compact (i.e., the closure w.r.t. narrow convergence is compact). 
    By Prokohorov's theorem (\cref{thm:prokhorov} in~\cref{app:preliminaries}), relative compactness is equivalent to $\closedWassersteinBall{2}{\radius}{\refmu}$ being tight. Since 
    \begin{equation*}
        \sup_{\mu\in\closedWassersteinBall{2}{\radius}{\refmu}}\int_{\reals^d}\norm{x}^2\d\mu(x)=
        \sup_{\mu\in\closedWassersteinBall{2}{\radius}{\refmu}}\wassersteinDistance{2}{\mu}{\diracMeasure{0}}^2
        \leq(\varepsilon+\wassersteinDistance{2}{\refmu}{\diracMeasure{0}})^2<+\infty.
    \end{equation*}
    tightness follows from~\cref{app:criterion tightness}.
    \qedhere 
\end{enumerate}
\end{proof}}

\newbool{showproofdrolinearcost}

\newcommand{\proofdrolinearcost}[1]{
\ifbool{#1}{\begin{proof}}{\begin{proof}[Proof of~\cref{prop:dro linear cost}]}

The proof consists of two parts. 
First, we assume absolute continuity of $\mu^\ast$ to ensure differentiability of the Wasserstein distance, and use the necessary conditions for optimality to construct candidate solutions. We then prove its optimality via sufficient conditions.

Assume $\mu^\ast$ is absolutely continuous. 
By \cref{cor:wasserstein distance}, together with absolutely continuity of $\mu^\ast$ and~\cref{prop:sum and multiplication rule}, the Wasserstein gradient of $K$ reads
\begin{equation*}
    \gradient{\mu}K(\mu)=2\left(\Id-\optMap{\mu^\ast}{\refmu}\right).
\end{equation*}
We can assume that there exists $\varphi\in\Ccinf{\reals^d}$ such that 
\begin{equation*}
    2\int_{\reals^d}\innerProduct{\Id-\optMap{\mu^\ast}{\refmu}}{\gradient{}\varphi}\d\mu^\ast\neq 0.
\end{equation*}
Indeed, if
\begin{equation}
    \int_{\reals^d}\innerProduct{\Id-\optMap{\mu^\ast}{\refmu}}{\gradient{}\varphi}\d\mu^\ast=0
\end{equation}
for all $\varphi\in\Ccinf{\reals^d}$, then, by \cref{lemma:variations}, $\optMap{\mu^\ast}{\refmu}=\Id$. In this case, however, the optimal transport map from $\mu^\ast$ to $\refmu$ would coincide with the identity map, which is clearly suboptimal. 
Thus, we can deploy \cref{thm:inequality:wasserstein lagrange multipliers} to conclude that at optimality there exists $\lambda\geq 0$ such that for $\mu^\ast$-almost every $x\in\reals^d$
\begin{equation*}
    -w+2\lambda\left(x-\optMap{\mu^\ast}{\refmu}(x)\right)=0,
\end{equation*}
where the negative sign appears because we are maximizing instead of minimizing. 
As $w\neq 0$, $\lambda=0$ is certainly not a solution and, by complementary slackness, the maximum lies at the boundary of the Wasserstein ball, so that $\wassersteinDistance{2}{\mu^\ast}{\refmu}=\varepsilon$. Thus, we have
\begin{equation*}
    \optMap{\mu^\ast}{\refmu}(x)
    =
    x-\frac{1}{2\lambda}w.
\end{equation*}
As $\optMap{\mu^\ast}{\refmu}$ is invertible, \cref{prop:inverse optimal transport map} suggests that the optimal transport map from $\refmu$ to $\mu^\ast$ reads
\begin{equation*}
    \optMap{\refmu}{\mu^\ast}(x)
    =
    x+\frac{1}{2\lambda}w.
\end{equation*}
Note that, as a consequence, $\mu^{\ast}$ is absolutely continuous whenever $\refmu$ is absolutely continuous. 
Moreover, 
\begin{equation*}
\begin{aligned}
    \radius^2
    =\wassersteinDistance{2}{\refmu}{\mu^\ast}^2
    =
    \int_{x\in\reals^d}\norm{x-\optMap{\hat\mu}{\mu^\ast}(x)}^2\d\refmu(x)
    =\frac{1}{4\lambda^2}\norm{w}^2.
\end{aligned}
\end{equation*}
Hence, the Wasserstein ball constraint reduces to
\begin{equation*}
    \frac{1}{2\lambda}=\pm\frac{\radius}{\norm{w}}.
\end{equation*}
Since $\lambda$ is non-negative, we have $\lambda=\frac{\norm{w}}{2\varepsilon}$ and thus $\optMap{\refmu}{\mu^\ast}(x)=x+\varepsilon\frac{w}{\norm{w}}$.
Overall, this gives
\begin{align*}
    \cost^\ast
    =
    \expectedValue{\mu^\ast}{\innerProduct{w}{x}}
    =
    \int_{\reals^d}\innerProduct{w}{x+\varepsilon\frac{w}{\norm{w}}}\d\refmu(x)
    =
    \expectedValue{\refmu}{\innerProduct{w}{x}} + \radius\norm{w}.
\end{align*}

We can now use \cref{cor:first order sufficient conditions constrained wasserstein} to prove optimality of $\mu^\ast=\pushforward{(\optMap{\refmu}{\mu^\ast})}\refmu$. Since $\lambda>0$, $\mu^\ast$ is absolutely continuous (since $\optMap{\refmu}{\mu^\ast}$ is invertible and $\refmu$ is absolutely continuous), and $-\cost$ is $0$-convex along generalized geodesics (by~\cref{prop:expectation}), we directly have optimality and uniqueness of $\mu^\ast$. 
\end{proof}}

\newbool{showprooflemmapropertiesdrovariance}

\newcommand{\prooflemmapropertiesdrovariance}[1]{
\ifbool{#1}{\begin{proof}}{\begin{proof}[Proof of~\cref{lemma:properties dro variance}]}
Continuity w.r.t. weak convergence in $\Pp{2}{\reals^d}$ follows from~\cref{prop:expectation} and~\cref{cor:variance}.
For geodesic convexity, consider $\mu_0, \mu_1\in\Pp{2}{\reals^d}$, $\gamma\in\setPlans{\mu_0}{\mu_1}$, and $\mu_t=\pushforward{((1-t)\proj_1+t\proj_2)}\gamma$. Then,
\begin{align*}
    -&\variance{\mu_t}{\innerProduct{w}{x}}
    \\
    &=
    -\frac{1}{2}\int_{\reals^{d}\times\reals^{d}}\norm{\innerProduct{w}{x-y}}^2\d(\mu_t\productMeasure\d\mu_t)(x,y)
    \\
    &=
    -\frac{1}{2}\int_{\reals^{d}\times\reals^{d}}\int_{\reals^{d}\times\reals^{d}}\norm{\innerProduct{w}{(1-t)x+tx'-(1-t)y-ty'}}^2\d\gamma(x,x')\d\gamma(y,y')
    \\
    &=
    \begin{aligned}[t]
    &-\frac{1}{2}\int_{\reals^{d}\times\reals^{d}}\int_{\reals^{d}\times\reals^{d}}(1-t)\norm{\innerProduct{w}{x-y}}^2 +t\norm{\innerProduct{w}{x'-y'}}^2\d\gamma(x,x')\d\gamma(y,y')
    \\
    &+\frac{1}{2}t(1-t)\int_{\reals^{d}\times\reals^{d}}\int_{\reals^{d}\times\reals^{d}}\norm{\innerProduct{w}{x-x'-y+y'}}^2\d\gamma(x,x')\d\gamma(y,y')
    \end{aligned}
    \\
    &= 
    \begin{aligned}[t]
    &-(1-t)\variance{\mu_0}{\innerProduct{w}{x}}-t\variance{\mu_1}{\innerProduct{w}{x}} 
    \\
    &+\frac{1}{2}t(1-t)\Bigg(\int_{\reals^{d}\times\reals^d}
    \norm{\innerProduct{w}{x-x'}}^2\d\gamma(x,x')+\int_{\reals^{d}\times\reals^d}\norm{\innerProduct{w}{y-y'}}^2\d\gamma(y,y')\Bigg)
    \\
    &-\frac{1}{2}t(1-t)\int_{\reals^{d}\times\reals^{d}}\int_{\reals^{d}\times\reals^{d}}2\innerProduct{w}{x-x'}\innerProduct{w}{y-y'}\d\gamma(x,x')\d\gamma(y,y')
    \end{aligned}
    \\
    \overset{\diamondsuit}&{\leq}
    -(1-t)\variance{\mu_0}{\innerProduct{w}{x}}-t\variance{\mu_1}{\innerProduct{w}{x}} 
    + t(1-t)\int_{\reals^{d}\times\reals^d}\norm{\innerProduct{w}{x-x'}}^2\d\gamma(x,x')
    \\
    \overset{\heartsuit}&{\leq}
    -(1-t)\variance{\mu_0}{\innerProduct{w}{x}}-t\variance{\mu_1}{\innerProduct{w}{x}} 
    +\norm{w}^2 t(1-t)\int_{\reals^{d}\times\reals^d}\norm{x-x'}^2\d\gamma(x,x'),
\end{align*}
where in $\diamondsuit$ we used 
\begin{equation*}
\begin{aligned}
    \int_{\reals^{d}\times\reals^{d}}\int_{\reals^{d}\times\reals^{d}}\!\innerProduct{w}{x-x'}\innerProduct{w}{y-y'}\d\gamma(x,x')\d\gamma(y,y')
    =
    \left(\int_{\reals^d\times\reals^d}\!\innerProduct{w}{x-x'}\d\gamma(x,x')\right)^2\geq 0,
\end{aligned}
\end{equation*}
and in $\heartsuit$ we used the Cauchy-Schwarz inequality~\cite[Theorem 4.2]{Rudin1987}. 
Since $\expectedValue{\mu}{\innerProduct{w}{x}}$ is 0-convex (by~\cref{prop:expectation}), $-\cost$ is $(-2\weight\norm{w}^2)$-convex along any interpolating curve.
\end{proof}}

\newbool{showproofdrovariance}

\newcommand{\proofdrovariance}[1]{
\ifbool{#1}{\begin{proof}}{\begin{proof}[Proof of~\cref{thm:dro variance}]}
We split the proof into two parts.
First, we assume enough regularity (in particular, absolute continuity of $\mu^\ast$) to use the necessary conditions for optimality to construct candidate solutions.
Second, we use sufficient conditions for optimality to show that the probability measure built before is indeed optimal. 

Let us assume that $\mu^\ast$ exists and that it is absolutely continuous with respect to the Lebesgue measure. 
By \cref{thm:wasserstein lagrange multipliers}, there exists $\lambda\geq 0$ such that
\begin{equation}\label{eq:dro variance lagrange multiplier equation}
    -w-2\weight w\innerProduct{w}{\Id{}-\expectedValue{\mu^\ast}{x}}+2\lambda\left(\Id-\optMap{\mu^\ast}{\refmu}\right)=0,
\end{equation}
where the gradient of $-\cost$ is computed via~\cref{prop:expectation,cor:variance} as
\begin{equation*}
    \gradient{\mu}(-\cost)(\mu)
    =
    -w - 2\weight w\innerProduct{w}{x-\expectedValue{\mu^\ast}{x}}
    =
    -w\left(1 + 2\weight\innerProduct{w}{x-\expectedValue{\mu^\ast}{x}}\right).
\end{equation*}
and the gradient of the Wasserstein distance stems from \cref{cor:wasserstein distance}.
With~\eqref{eq:dro variance lagrange multiplier equation}, we can already conclude $\lambda>0$. Else,
$\innerProduct{w}{x}= -\frac{1}{2\weight} + \expectedValue{\mu^\ast}{\innerProduct{w}{x}}$, which implies that 
$\innerProduct{w}{x}$ is constant $\mu^\ast$-almost everywhere (say, $\innerProduct{w}{x}=\beta\in\reals$ for $\mu^\ast$-almost every $x\in\reals^d$) and
\begin{align*}
    \beta
    =-\frac{1}{2\weight}+\int_{\reals^d}\innerProduct{w}{x}\d\mu^\ast(x)
    =-\frac{1}{2\weight}+\beta,
\end{align*}
which is a contradiction. Thus, $\lambda>0$ and complementary slackness yields $\wassersteinDistance{2}{\refmu}{\mu^\ast}=\radius$.
We now characterize the optimal transport map $\optMap{\mu^\ast}{\refmu}$. By~\eqref{eq:dro variance lagrange multiplier equation}, we have the affine optimal transport map
\begin{equation*}
    \optMap{\mu^\ast}{\refmu}
    =\underbrace{\left(\Id-\frac{\weight}{\lambda}w\transpose{w}\right)}_{\eqqcolon A}x \underbrace{-\frac{w}{2\lambda} + \frac{\rho}{\lambda}\expectedValue{\mu^\ast}{\innerProduct{w}{x}}w}_{\eqqcolon b},
\end{equation*}
where $A=\transpose{A}$ and $\spec(A)=\{1,\ldots,1,1-\frac{\weight}{\lambda}\norm{w}^2\}$.
Since $\mu^\ast$ is, by assumption, absolutely continuous, the optimal transport map $\optMap{\mu^\ast}{\refmu}$ is the gradient of a convex function, which implies $\det(A)\geq 0$, and so $\lambda\geq \weight\norm{w}^2$.
Moreover, if $\lambda=\rho\norm{w}^2$, then $A$ is not invertible, which contradicts $\refmu$ being absolutely continuous (since all probability mass would be concentrated on a linear subspace, which has measure 0). 
Thus, $\lambda>\weight\norm{w}^2$, $A$ is invertible, and by \cref{prop:inverse optimal transport map}:
\begin{equation}\label{eq:dro variance opt map}
    \optMap{\refmu}{\mu^\ast}(x)=(\optMap{\mu^\ast}{\refmu})^{-1}(x)=A^{-1}x-A^{-1}b.
\end{equation}
It is readily verified that
\begin{equation*}
    A^{-1}=\eye{}+\frac{\frac{\weight}{\lambda}}{1-\frac{\weight}{\lambda}\norm{w}^2}w\transpose{w}.
\end{equation*}
Since $1-\weight\norm{w}^2/\lambda>0$, $A^{-1}$ is well defined.
By definition, $A+\frac{\weight}{\lambda}w\transpose{w}=\eye{}$, so
\begin{equation}\label{eq:dro variance equality Ainv}
\begin{aligned}
    A^{-1}w
    =\left(1+\frac{\frac{\weight}{\lambda}\norm{w}^2}{1-\frac{\weight}{\lambda}\norm{w}^2}\right)w
    =\frac{1}{1-\frac{\weight}{\lambda}\norm{w}^2}w.
\end{aligned}
\end{equation}
By definition of $b$, we have
\begin{align*}
    b
    &=
    -\frac{w}{2\lambda} + \frac{\rho}{\lambda}\expectedValue{\mu^\ast}{\innerProduct{w}{x}}w
    \\
    &=
    -\frac{w}{2\lambda} + \frac{\rho}{\lambda}\expectedValue{\refmu}{\innerProduct{w}{\optMap{\refmu}{\mu^\ast}(x)}}w
    \\
    \overset{\eqref{eq:dro variance opt map}}&{=}
    -\frac{w}{2\lambda}+\frac{\weight}{\lambda}\left(\expectedValue{\refmu}{\innerProduct{w}{A^{-1}x}}-\innerProduct{w}{A^{-1}b}\right)w,
    \\
    &=
    -\frac{w}{2\lambda}+\frac{\weight}{\lambda}\left(\expectedValue{\refmu}{\innerProduct{A^{-1}w}{x}}-\innerProduct{A^{-1}w}{b}\right)w
    \\
    \overset{\eqref{eq:dro variance equality Ainv}}&{=}
    -\frac{w}{2\lambda}+\frac{\frac{\weight}{\lambda}}{1-\frac{\weight}{\lambda}\norm{w}^2}\expectedValue{\refmu}{\innerProduct{w}{x}}w-\frac{\weight}{\lambda}w\transpose{w}A^{-1}b
    \\
    \overset{\diamondsuit}&{=}
    -\frac{w}{2\lambda}+\frac{\frac{\weight}{\lambda}}{1-\frac{\weight}{\lambda}\norm{w}^2}\expectedValue{\refmu}{\innerProduct{w}{x}}w+b-A^{-1}b,
\end{align*}
where in $\diamondsuit$ we used
\begin{equation*}
    \frac{\weight}{\lambda}w\transpose{w}\inv{A}
    =
    \frac{\weight}{\lambda}w\transpose{(A^{-1}w)}
    \overset{\eqref{eq:dro variance equality Ainv}}{=}
    \frac{\frac{\weight}{\lambda}}{1-\frac{\weight}{\lambda}\norm{w}^2} w\transpose{w}=\eye{}-\inv{A}.
\end{equation*}
Thus,
\begin{equation}\label{eq:dro variance A times b}
    A^{-1}b
    =
    -\frac{w}{2\lambda}+\frac{\frac{\weight}{\lambda}}{1-\frac{\weight}{\lambda}\norm{w}^2}\expectedValue{\refmu}{\innerProduct{w}{x}}w.
\end{equation}
We can now leverage the constraint (which, by complementary slackness, is active) to find $\lambda$:
\begin{equation}\label{eq:dro variance distance}
\begin{aligned}
    \radius^2
    &=
    \wassersteinDistance{2}{\refmu}{\mu^\ast}^2
    \\
    \overset{\eqref{eq:dro variance opt map}}&{=}
    \int_{\reals^d}\norm{x-(A^{-1}x-A^{-1}b)}^2\d\refmu(x)
    \\
    &=
    \int_{\reals^d}\norm{(I-A^{-1})x+A^{-1}b}^2\d\refmu(x)
    \\
    &=
    \int_{\reals^d}\norm{-\frac{\weight}{\lambda}w\transpose{w}A^{-1}x+A^{-1}b}^2\d\refmu(x)
    \\
    \overset{\eqref{eq:dro variance equality Ainv}}&{=}
    \int_{\reals^d}\norm{-\frac{\frac{\weight}{\lambda}}{1-\frac{\weight}{\lambda}\norm{w}^2}\innerProduct{w}{x}w+A^{-1}b}^2\d\refmu(x)
    \\
    \overset{\eqref{eq:dro variance A times b}}&{=}
    \int_{\reals^d}\norm{\frac{\frac{\weight}{\lambda}}{1-\frac{\weight}{\lambda}\norm{w}^2}\left(\innerProduct{w}{x}-\expectedValue{\refmu}{\innerProduct{w}{x}}\right)w+\frac{w}{2\lambda}}^2\d\refmu(x)
    \\
    &=
    \left(\frac{\frac{\weight}{\lambda}\norm{w}}{1-\frac{\weight}{\lambda}\norm{w}^2}\right)^2\variance{\refmu}{\innerProduct{w}{x}}
    +
    \frac{\frac{\weight}{\lambda^2}\norm{w}^2}{1-\frac{\weight}{\lambda}\norm{w}^2}
    \left(\int_{\reals^d}\innerProduct{w}{x}\d\refmu(x)-\expectedValue{\refmu}{\innerProduct{w}{x}}\right)+\frac{\norm{w}^2}{4\lambda^2}
    \\
    &=
    \left(\frac{\frac{\weight}{\lambda}\norm{w}}{1-\frac{\weight}{\lambda}\norm{w}^2}\right)^2\variance{\refmu}{\innerProduct{w}{x}}
    +
    \frac{\norm{w}^2}{4\lambda^2}.
\end{aligned}
\end{equation}
Finally, we multiply \eqref{eq:dro variance distance} by $\lambda^2(1-\frac{\weight}{\lambda}\norm{w}^2)^2/\radius^2$ and rearrange to obtain the fourth-order polynomial equation $f(\lambda)=0$ with 
\begin{equation*}
    f(\lambda)
    \coloneqq 
    \lambda^4
    -2\weight\norm{w}^2 \lambda^3
    +\left(-\frac{\weight^2\norm{w}^2}{\radius^2}\variance{\refmu}{\innerProduct{w}{x}} - \frac{\norm{w}^2}{4\radius^2}+\weight^2\norm{w}^4\right)\lambda^2
    +\frac{\weight\norm{w}^4}{2\radius^2}\lambda
    -\frac{\weight^2\norm{w}^6}{4\radius^2}.
\end{equation*}
To ease the notation, let
\begin{equation*}
    \Lambda\coloneqq\{\lambda>\weight\norm{w}^2: f(\lambda)=0\}.
\end{equation*}
The set $\Lambda$ contains all real roots of $f$ such that $1-\weight\norm{w}^2/\lambda>0$. Values of $\lambda<\weight\norm{w}^2$ would contradict the optimality of the transport map or non-negativity of $\lambda$. Clearly, $\cardinality{\Lambda}\leq 4$; we will now prove $\Lambda$ is non-empty. 
First, $f(\lambda)\to+\infty$ as $\lambda\to+\infty$.
Second, $f(\lambda)\to-\norm{w}^6\weight^4\variance{\refmu}{\innerProduct{w}{x}}/\radius^2<0$ (since $\refmu$ is absolutely continuous, $\variance{\refmu}{\innerProduct{w}{x}}>0$) for $\lambda\downarrow\weight\norm{w}^2>0$. So, by the intermediate value theorem, $\Lambda$ is non-empty and $\cardinality{\Lambda}\geq 1$.
The resulting cost is
\begin{align*}
    \cost^\ast
    \ifbool{compact}{&=}{=}
    \cost(\mu^\ast)
    \ifbool{compact}{\\ &=}{&=}
    \expectedValue{\mu^\ast}{\innerProduct{w}{x}}+\weight\variance{\mu^\ast}{\innerProduct{w}{x}}
    \\
    \overset{\eqref{eq:dro variance opt map}}&{=}
    \expectedValue{\refmu}{\innerProduct{w}{A^{-1}x-A^{-1}b}}+\weight\variance{\refmu}{\innerProduct{w}{A^{-1}x-A^{-1}b}}
    \\
    &=
    \expectedValue{\refmu}{\innerProduct{w}{A^{-1}x}}-\innerProduct{w}{A^{-1}b}+\weight\int_{\reals^d}\innerProduct{w}{A^{-1}x-A^{-1}b-A^{-1}\expectedValue{\refmu}{x}+A^{-1}b}^2\d\refmu(x)
    \\
    \overset{\eqref{eq:dro variance A times b}}&{=}
    \expectedValue{\refmu}{\innerProduct{w}{A^{-1}x}}
    +\frac{\norm{w}^2}{2\lambda}-\frac{\frac{\weight}{\lambda}\norm{w}^2}{1-\frac{\weight}{\lambda}\norm{w}^2}\expectedValue{\refmu}{\innerProduct{w}{x}}
    +\weight\int_{\reals^d}\innerProduct{A^{-1}w}{x-\expectedValue{\refmu}{x}}^2\d\refmu(x)
    \\
    \overset{\eqref{eq:dro variance equality Ainv}}&{=}
    \frac{1-\frac{\weight}{\lambda}\norm{w}^2}{{1-\frac{\weight}{\lambda}\norm{w}^2}}\expectedValue{\refmu}{\innerProduct{w}{x}}+\frac{1}{2\lambda}\norm{w}^2+\weight\left(\frac{1}{1-\frac{\weight}{\lambda}\norm{w}^2}\right)^2\int_{\reals^d}\innerProduct{w}{x-\expectedValue{\refmu}{x}}^2\d\refmu(x)
    \\
    &=
    \expectedValue{\refmu}{\innerProduct{w}{x}}
    +\frac{1}{2\lambda}\norm{w}^2+\weight\left(\frac{1}{1-\frac{\weight}{\lambda}\norm{w}^2}\right)^2\variance{\refmu}{\innerProduct{w}{x}}.
\end{align*}
Overall, this provides with $1\leq\cardinality{\Lambda}\leq 4$ candidate optimal solutions, one for each $\lambda\in\Lambda$.

To show that one of them is indeed optimal, we will use sufficient conditions. 
From above, there exists $\mu^\ast$ absolutely continuous and $\lambda>\weight\norm{w}^2$ such that 
\begin{equation*}
    \gradient{\mu}\cost(\mu^\ast)+\lambda\gradient{\mu}\left(\wassersteinDistance{2}{\mu^\ast}{\refmu}^2-\radius\right)=0.
\end{equation*}
Thanks to \cref{lemma:properties dro variance,cor:first order sufficient conditions constrained wasserstein}, $\mu^\ast$ is the unique global minimizer. In particular, this also proves that $\Lambda$ contains at most one element; else, we would contradict uniqueness. 
\end{proof}}

\newbool{showprooflemmapropertiesdrostd}

\newcommand{\prooflemmapropertiesdrostd}[1]{
\ifbool{#1}{\begin{proof}}{\begin{proof}[Proof of~\cref{lemma:properties dro std}]}
Continuity follows from~\cref{prop:expectation} and~\cref{cor:variance}, together with continuity of $x\mapsto-\sqrt{x}$. To prove the upper bound, let $\mu\in\closedWassersteinBall{2}{\radius}{\refmu}$ be arbitrary. Let $\nu=\pushforward{(\innerProduct{w}{x})}\mu\in\Pp{2}{\reals}$ and $\hat\nu=\pushforward{(\innerProduct{w}{x})}\refmu\in\Pp{2}{\reals}$. Then, the Wasserstein distance between $\nu$ and $\hat\nu$ reads
\begin{equation}\label{eq:lemmapropertiesdrostd:preliminary}
\begin{aligned}
    \wassersteinDistance{2}{\nu}{\hat\nu}
    &=
    \wassersteinDistance{2}{\pushforward{(\innerProduct{w}{x})}\mu}{\pushforward{(\innerProduct{w}{x})}\refmu}
    \\
    \overset{\triangle}&{=}
    \sqrt{\min_{\gamma\in\setPlans{\mu}{\refmu}}\int_{\reals^d\times\reals^d}\abs{\innerProduct{w}{x}-\innerProduct{w}{y}}^2\d\gamma(x,y)}
    \\
    \overset{\square}&{\leq}
    \norm{w}\wassersteinDistance{2}{\mu}{\hat\mu}, 
\end{aligned}
\end{equation}
where 
\begin{itemize}
    \item in $\triangle$ we use~\cite[Lemma 6.1]{terpin2023dynamic} or~\cite[Proposition 3]{Aolaritei2023DistributionalTransport}; and
    \item in $\square$ we use linearity of the inner product and Cauchy-Schwarz inequality~\cite[Theorem 4.2]{Rudin1987}.
\end{itemize}
Then, 
\begin{align*}
    \cost(\mu)
    &=
    \expectedValue{\mu}{\innerProduct{w}{x}}
    + \weight\standardDeviation{\mu}{\innerProduct{w}{x}}
    \\
    \overset{\triangle}&{=} 
    \expectedValue{\nu}{y} + \weight\standardDeviation{\nu}{y}
    \\
    &\leq 
    \expectedValue{\hat\nu}{y} + \abs{\expectedValue{\nu}{y}-\expectedValue{\hat\nu}{y}}
    +\weight\standardDeviation{\hat\nu}{y} + \weight\abs{\standardDeviation{\nu}{y}-\standardDeviation{\hat\nu}{y}}
    \\
    \overset{\square}&{\leq}
    \cost(\refmu) + \sqrt{1+\weight^2}\sqrt{\abs{\expectedValue{\nu}{y}-\expectedValue{\hat\nu}{y}}^2 + \abs{\standardDeviation{\nu}{y}-\standardDeviation{\hat\nu}{y}}^2}
    \\
    \overset{\heartsuit}&{\leq}
    \cost(\refmu) + \sqrt{1+\weight^2}\wassersteinDistance{2}{\nu}{\hat\nu}
    \\
    \overset{\eqref{eq:lemmapropertiesdrostd:preliminary}}&{\leq}
    \cost(\refmu) + \norm{w}\sqrt{1+\weight^2}\wassersteinDistance{2}{\mu}{\refmu}
    \\
    &\leq 
    \cost(\refmu) + \radius\norm{w}\sqrt{1+\weight^2},
\end{align*}
where  
\begin{itemize}
    \item in $\triangle$ we use the definition of $\nu$;
    \item in $\square$ we use that
        $\sqrt{a^2+b^2}\geq \frac{a+c b}{\sqrt{1+c^2}}$ for all non-negative $a,b,c\in\nonnegativeReals$; and
    \item in $\heartsuit$ we use (a corollary of) Gelbrich bound~\cite{Gelbrich1990}: 
    If $\mu,\nu\in\Pp{2}{\reals}$ with mean $m_1,m_2\in\reals$ and standard deviation $\sigma_1,\sigma_2\in\nonnegativeReals$, then
    $
    \wassersteinDistance{2}{\mu}{\nu}
        \geq 
        \sqrt{\abs{ m_1-m_2}^2 + \abs{\sigma_1-\sigma_2}^2}
    $.
    \qedhere 
\end{itemize}
\end{proof}}

\newbool{showproofdrostd}

\newcommand{\proofdrostd}[1]{
\ifbool{#1}{\begin{proof}}{\begin{proof}[Proof of~\cref{thm:dro std}]}
Again, we assume sufficient regularity to use necessary conditions for optimality to construct a solution. We will then prove that this solution is indeed optimal.

Let us assume that $\mu^\ast$ exists and that it is absolutely continuous. Being absolutely continuous, $\standardDeviation{\mu^\ast}{\innerProduct{w}{x}}>0$. Thus, $\mu\mapsto-\standardDeviation{\mu}{\innerProduct{w}{x}}=-\sqrt{\variance{\mu}{\innerProduct{w}{x}}}$ is Wasserstein differentiable at $\mu^\ast$ and,~\cref{prop:expectation,cor:variance} and the sum, multiplication, and chain rules (\cref{prop:sum and multiplication rule,prop:chain rule}), its Wasserstein gradient reads
\begin{equation*}
-w - \frac{\weight}{\standardDeviation{\mu^\ast}{\innerProduct{w}{x}}}w\innerProduct{w}{\Id{}-\expectedValue{\mu^\ast}{x}},
\end{equation*}
which is well-defined since $\standardDeviation{\mu^\ast}{\innerProduct{w}{x}}>0$. As above, the Wasserstein gradient of the constraint follows directly from~\cref{cor:wasserstein distance,prop:sum and multiplication rule} and reads $2(\Id{} - \optMap{\mu^\ast}{\refmu})$. Then, by \cref{thm:inequality:wasserstein lagrange multipliers}, there exists $\lambda\geq 0$ such that
\begin{equation}\label{eq:dro std lagrange multiplier equation}
    -w - \frac{\weight}{\standardDeviation{\mu^\ast}{\innerProduct{w}{x}}}w\innerProduct{w}{\Id{}-\expectedValue{\mu^\ast}{x}}
    +
    2\lambda\left(\Id{} - \optMap{\mu^\ast}{\refmu}\right)=0.
\end{equation}
As for the mean-variance, the case $\lambda=0$ yields a contradiction. Thus, $\lambda>0$ and complementary slackness yields $\wassersteinDistance{2}{\refmu}{\mu^\ast}=\radius$.
Then, \eqref{eq:dro std lagrange multiplier equation} yields the optimal transport map 
\begin{equation}
     \optMap{\mu^\ast}{\refmu}(x)=
     \underbrace{\left(\Id{}-\alpha w\transpose{w}\right)}_{\eqqcolon A}x
     \underbrace{-\frac{w}{2\lambda}+\alpha\expectedValue{\mu^\ast}{\innerProduct{w}{x}}}_{\eqqcolon b}w,
\end{equation}
where $\alpha\coloneqq \frac{\weight}{2\lambda\standardDeviation{\mu^\ast}{\innerProduct{w}{x}}}$, and $A=\transpose{A}$ and $\spec(A)=\{1,\ldots,1,1-\alpha\norm{w}^2\}$.
Since $\mu^\ast$ is, by assumption, absolutely continuous, the optimal transport map $\optMap{\mu^\ast}{\refmu}$ is the gradient of a convex function, and $1-\alpha\norm{w}^2\geq 0$. Moreover, if $1-\alpha\norm{w}^2=0$, then $A$ is not invertible, which contradicts $\refmu$ being absolutely continuous (since all probability mass would be concentrated on a linear subspace, which has measure 0). Thus, by \cref{prop:inverse optimal transport map},
\begin{equation}\label{eq:dro std opt map}
    \optMap{\refmu}{\mu^\ast}(x)=(\optMap{\mu^\ast}{\refmu})^{-1}(x)
    =A^{-1}x-A^{-1}b.
\end{equation}
It is readily verified that 
\begin{equation*}
    A^{-1}=\eye{}+\frac{\alpha}{1-\alpha\norm{w}^2}w\transpose{w},
\end{equation*}
well defined since $1-\alpha\norm{w}^2>0$. Thus, 
\begin{equation}\label{eq:dro std A inverse w}
    A^{-1}w
    =\left(1+\frac{\alpha\norm{w}^2}{1-\alpha\norm{w}^2}\right)w 
    =\frac{1}{1-\alpha\norm{w}^2}w
\end{equation}
and 
\begin{equation}\label{eq:dro std std star}
    \standardDeviation{\mu^\ast}{\innerProduct{w}{x}}
    =
    \sqrt{\transpose{w}A^{-1}\variance{\refmu}{x}A^{-1}w}
    =
    \frac{1}{1-\alpha\norm{w}^2}\standardDeviation{\refmu}{\innerProduct{w}{x}}
\end{equation}
Thus, the definition and \eqref{eq:dro std std star} yield 
\begin{equation*}
    \alpha
    =
    \frac{\weight}{2\lambda\standardDeviation{\mu^\ast}{\innerProduct{w}{x}}}
    =
    \frac{\weight}{2\lambda\standardDeviation{\refmu}{\innerProduct{w}{x}}}(1-\alpha\norm{w}^2),
\end{equation*}
which gives 
\begin{equation}\label{eq:dro std alpha}
    \alpha
    =
    \frac{\weight}{2\lambda\standardDeviation{\refmu}{\innerProduct{w}{x}} + \weight\norm{w}^2}.
\end{equation}
and 
\begin{equation}\label{eq:dro std:lambda non-zero}
    1-\alpha\norm{w}^2
    =
    \frac{2\lambda\standardDeviation{\refmu}{\innerProduct{w}{x}}}{2\lambda\standardDeviation{\refmu}{\innerProduct{w}{x}} + \weight\norm{w}^2}>0.
\end{equation}
Thus, 
\begin{equation*}
\begin{aligned}
    A
    &=
    \eye{}-\frac{\weight}{2\lambda\standardDeviation{\refmu}{\innerProduct{w}{x}} + \weight\norm{w}^2}w\transpose{w}.
\end{aligned}
\end{equation*}
Similarly, 
\begin{equation*}
\begin{aligned}
    b 
    &=
    -\frac{w}{2\lambda}+\alpha\expectedValue{\mu^\ast}{\innerProduct{w}{x}} w
    \\
    \overset{\eqref{eq:dro std opt map}}&{=}
    -\frac{w}{2\lambda}+\alpha \left(\expectedValue{\refmu}{\innerProduct{w}{A^{-1}x}}-\innerProduct{w}{A^{-1}b})\right)w
    \\
    &=
    -\frac{w}{2\lambda}+\alpha \left(\expectedValue{\refmu}{\innerProduct{A^{-1}w}{x}}-\innerProduct{A^{-1}w}{b})\right)w
    \\
    &=
    -\frac{w}{2\lambda}+\frac{\alpha}{1-\alpha\norm{w}^2}\expectedValue{\refmu}{\innerProduct{w}{x}}w-\alpha w\transpose{w}A^{-1}b
    \\
    &=
    -\frac{w}{2\lambda}+\frac{\alpha}{1-\alpha\norm{w}^2}\expectedValue{\refmu}{\innerProduct{w}{x}}w+b-A^{-1}b,
\end{aligned}
\end{equation*}
and so
\begin{equation*}
    A^{-1}b
    =
    -\frac{w}{2\lambda}+\frac{\alpha}{1-\alpha\norm{w}^2}\expectedValue{\refmu}{\innerProduct{w}{x}}w
    \overset{\eqref{eq:dro std alpha}}{=}
    -\frac{w}{2\lambda}+\frac{\weight}{2\lambda\standardDeviation{\refmu}{\innerProduct{w}{x}}} \expectedValue{\refmu}{\innerProduct{w}{x}}w.
\end{equation*}
By~\eqref{eq:dro std:lambda non-zero}, $\lambda$ cannot be zero. Thus, complementary slackness yields $\wassersteinDistance{2}{\refmu}{\mu^\ast}=\radius$.
We can now use the constraint to find $\lambda$. We proceed as in~\eqref{eq:dro variance distance} to get
\begin{equation*}
    \radius^2=\wassersteinDistance{2}{\refmu}{\mu^\ast}^2=\frac{1}{4\lambda^2}\norm{w}^2(1+\weight^2),
\end{equation*}
and so $\lambda^\ast=\frac{1}{2\radius}\norm{w}\sqrt{1+\weight^2}$. The resulting worst-case cost reads
\begin{align*}
    \cost^\ast
    =
    \cost(\mu^\ast)
    &=
    \expectedValue{\mu^\ast}{\innerProduct{w}{x}} + \weight\standardDeviation{\mu^\ast}{\innerProduct{w}{x}}
    \\
    \overset{\eqref{eq:dro std opt map}}&{=}
    \expectedValue{\refmu}{\innerProduct{w}{A^{-1}x-A^{-1}b}} 
    + 
    \weight\standardDeviation{\mu^\ast}{\innerProduct{w}{x}}
    \\
    \overset{\eqref{eq:dro std std star}}&{=}
    \expectedValue{\refmu}{\innerProduct{w}{A^{-1}x}}-\innerProduct{w}{A^{-1}b}
    +
    \frac{2\lambda^\ast\standardDeviation{\refmu}{\innerProduct{w}{x}} + \weight^2\norm{w}^2}{2\lambda^\ast}
    \\
    \overset{\eqref{eq:dro variance A times b}}&{=}
    \begin{aligned}[t]
    &\expectedValue{\refmu}{\innerProduct{A^{-1}w}{x}}+
    \frac{\norm{w}^2}{2\lambda^\ast}-\frac{\weight\norm{w}^2}{2\lambda^\ast\standardDeviation{\refmu}{\innerProduct{w}{x}}} \expectedValue{\refmu}{\innerProduct{w}{x}}
    \\
    &+
    \frac{2\lambda^\ast\standardDeviation{\refmu}{\innerProduct{w}{x}} + \weight^2\norm{w}^2}{2\lambda^\ast}
    \end{aligned}
    \\
    \overset{\eqref{eq:dro std A inverse w}}&{=}
    \begin{aligned}[t]
    &\frac{2\lambda^\ast\standardDeviation{\refmu}{\innerProduct{w}{x}} + \weight\norm{w}^2}{2\lambda^\ast\standardDeviation{\refmu}{\innerProduct{w}{x}}}\expectedValue{\refmu}{\innerProduct{w}{x}}
    \\&+
    \frac{\norm{w}^2}{2\lambda^\ast}-\frac{\weight\norm{w}^2}{2\lambda^\ast\standardDeviation{\refmu}{\innerProduct{w}{x}}} \expectedValue{\refmu}{\innerProduct{w}{x}}
    +
    \frac{2\lambda^\ast\standardDeviation{\refmu}{\innerProduct{w}{x}} + \weight^2\norm{w}^2}{2\lambda^\ast}
    \end{aligned}
    \\
    &=
    \expectedValue{\refmu}{\innerProduct{w}{x}}
    +
    \standardDeviation{\refmu}{\innerProduct{w}{x}} + \frac{1}{2\lambda^\ast}\norm{w}^2(1+\weight^2).
\end{align*}
It is clearly optimal to select $\lambda^\ast=+\frac{1}{2\radius}\norm{w}\sqrt{1+\weight^2}$. The resulting worst-case cost reads
\begin{equation*}
    \cost^\ast
    =
    \cost(\mu^\ast)
    =
    \expectedValue{\refmu}{\innerProduct{w}{x}}
    +
    \standardDeviation{\refmu}{\innerProduct{w}{x}} + \radius\norm{w}\sqrt{1+\weight^2}.
\end{equation*}
Overall, this gives us the candidate optimal solution $\mu^\ast=\pushforward{T}\mu$ with 
\begin{equation*}
    T(x)=
    \left(\eye{}+\frac{\weight\radius}{\norm{w}\sqrt{1+\weight^2}\standardDeviation{\refmu}{\innerProduct{w}{x}}}w\transpose{w}\right)x
    +
    \left(1-\frac{\weight\expectedValue{\refmu}{\innerProduct{w}{x}}}{\standardDeviation{\refmu}{\innerProduct{w}{x}}}\right)\frac{\radius}{\sqrt{1+\weight^2}}\frac{w}{\norm{w}},
\end{equation*}
which is well-defined since $\refmu$ is absolutely continuous and so $\standardDeviation{\refmu}{\innerProduct{w}{x}}>0$. Since $T$ is invertible, $\mu^\ast$ is absolutely continuous too. Thus, the probability measure $\mu^\ast$ indeed satisfies the first-order necessary condition for optimality.

To prove that $\mu^\ast$ is optimal, it suffices to observe that it attains the (a priori) upper bound \eqref{eq:dro std upper bound}.
\end{proof}}

\newbool{showprooflemmaexistencekl}

\newcommand{\prooflemmaexistencekl}[1]{
\ifbool{#1}{\begin{proof}}{\begin{proof}[Proof of~\cref{lemma:min entropy wasserstein ball existence}]}
The proof follows by lower semi-continuity of the \gls{acr:kl} divergence (\cref{prop:divergence}) and compactness of Wasserstein balls (\cref{prop:wasserstein balls}) w.r.t. narrow convergence. Since the \gls{acr:kl} divergence evaluates to infinity for non-absolutely continuous measures, the minimizer is necessarily absolutely continuous. 
\end{proof}}

\newbool{showproofoptimalsolutionkl}

\newcommand{\proofoptimalsolutionkl}[1]{
\ifbool{#1}{\begin{proof}}{\begin{proof}[Proof of~\cref{prop:conditions for optimality kl}]}
By~\cref{prop:divergence} (and~\cref{ex:kl divergence}), $\frac{\gradient{}\rho^\ast}{\rho^\ast}$ is a Wasserstein subgradient of $\kullbackLeibler{\cdot}{\muprior}$ at $\mu^\ast=\rho^\ast\lebesgueMeasure{d}$.
Since both $\muref$ and $\mu^\ast$ are absolutely continuous,~\cite[Proposition 6.2.12]{Ambrosio2008a} gives $\det\gradient{}\optMap{\mu^\ast}{\muref}(x)>0$ for $\mu^\ast$-almost every $x\in\reals^d$.
Thus, the usual change of variables for integrals~\cite[Lemma 5.3.3]{Ambrosio2008a} yields
\begin{equation*}
    \det\gradient{}\optMap{\mu^\ast}{\muref}(x)
    =
    \frac{\rho^\ast(x)}{\rho_\mathrm{r}(\optMap{\mu^\ast}{\muref}(x))}.
\end{equation*}
Then, 
\begin{align*}
    \frac{\gradient{}\rho^\ast(x)}{\rho^\ast(x)}
    &=
    \frac{\gradient{}\left(
    \det\gradient{}\optMap{\mu^\ast}{\muref}(x)\rho_\mathrm{r}(\optMap{\mu^\ast}{\muref}(x))
    \right)}{\det\gradient{}\optMap{\mu^\ast}{\muref}(x)\rho_\mathrm{r}(\optMap{\mu^\ast}{\muref}(x))}
    \\
    &=
    \frac{\gradient{}\det\gradient{}\optMap{\mu^\ast}{\muref}(x)}{\det\gradient{}\optMap{\mu^\ast}{\muref}(x)}
    +
    \frac{\gradient{}(\rho_\mathrm{r}(\optMap{\mu^\ast}{\muref}(x)))}{\rho_\mathrm{r}(\optMap{\mu^\ast}{\muref}(x))}
    \\
    &=
    \sum_{i=1}^d\trace\left((\gradient{}\optMap{\mu^\ast}{\muref}(x))^{-1}
    \diff{\gradient{}\optMap{\mu^\ast}{\muref}(x)}{x_i}
    \right)e_i
    -
    \gradient{}\optMap{\mu^\ast}{\muref}(x)\gradient{}\Vref(\optMap{\mu^\ast}{\muref}(x)),
\end{align*}
where we used Jacobi's formula~\cite{horn2012matrix} for the derivative of the determinant and 
\begin{equation*}
\begin{aligned}
    \gradient{}\rho_\mathrm{r}(\optMap{\mu^\ast}{\muref}(x))
    =
    \gradient{}\left(e^{-\Vref(\optMap{\mu^\ast}{\muref}(x))}\right)
    &=
    -\gradient{}\optMap{\mu^\ast}{\muref}(x)\gradient{}\Vref(\optMap{\mu^\ast}{\muref}(x))e^{-\Vref(\optMap{\mu^\ast}{\muref}(x))}
    \\
    &=
    -\gradient{}\optMap{\mu^\ast}{\muref}(x)\gradient{}\Vref(\optMap{\mu^\ast}{\muref}(x))\rho_\mathrm{r}(\optMap{\mu^\ast}{\muref}(x)).
\end{aligned}
\end{equation*}
The statement then follows directly from~\cref{cor:first order sufficient conditions constrained wasserstein}, considering that the \gls{acr:kl} divergence is convex along generalized geodesics of  (by~\cref{prop:divergence}) and the optimal coupling between two absolutely continuous probability measures is unique and induced by an optimal transport map (by~\cref{thm:brenier}).
In the case of Gaussian probability measures, we first show that \eqref{eq:minimum entropy necessary condition} admits a solution of the form $\optMap{\mu^\ast}{\muref}(x)=Ax+b$ for $A\in\reals^{d\times d}$ symmetric and positive definite and $b\in\reals^d$.  In this case, by~\cref{prop:brenier2}, the optimal transport from $\muref$ to $\mu^\ast$ is
\begin{equation}\label{eq:minimum entropy inverse optimal map}
    \optMap{\muref}{\mu^\ast}(x)=(\optMap{\mu^\ast}{\muref})^{-1}(x)=A^{-1}x-A^{-1}b.
\end{equation}
Since $\gradient{}\optMap{\mu^\ast}{\muref}=A$, \eqref{eq:minimum entropy necessary condition} reduces to
\begin{equation}\label{eq:minimum entropy optimality condition ae}
    \Sprior^{-1} x - \Sprior^{-1}\mprior - A\left(\Sref^{-1}(Ax+b) - \Sref^{-1}\mref\right) + 2\lambda(x - (Ax+b))=0,
\end{equation}
for $\mu^\ast$-almost every $x\in\reals^d$.
By the absolute continuity of $\mu^\ast$, \eqref{eq:minimum entropy optimality condition ae} reduces to 
\begin{equation*}
\begin{aligned}
    \Sprior^{-1} - A\Sref^{-1} A +2\lambda(\eye{}-A) &=0
    \\
    -\Sprior^{-1}\mprior + A\Sref^{-1}b -A\Sref^{-1}\mref -2\lambda b &= 0.
\end{aligned}
\end{equation*}
Basic algebraic manipulations then lead to 
\begin{subequations}\label{eq:minimum entropy eqs}
\begin{align}
    0&=-2\lambda A - A\Sref^{-1}A +  \Sprior^{-1}+2\lambda\eye{},\label{eq:minimum entropy eq A}
    \\
    b&=A(\Sprior^{-1}+2\lambda\eye{})^{-1}(-\Sprior^{-1}\mprior+A\Sref^{-1}\mref).\label{eq:minimum entropy eq b}
\end{align}
\end{subequations}
For any fixed $\lambda\geq 0$, the system of equations~\eqref{eq:minimum entropy eqs} fully characterizes $A$ and $b$. Indeed,~\eqref{eq:minimum entropy eq A} is a continuous-time algebraic Riccati equation (with $A=-\lambda\eye{}$, $N=\Sref$, $B=\eye{}$, and $C=(\Sprior+2\lambda\eye{})^{1/2}$ in the notation in~\cite{wonham1968matrix}). Thus, since $(A,B)$ is stabilizable and $(C,A)$ observable, there exists a unique positive definite solution $A\in\reals^{d\times d}$ for any $\lambda\geq 0$~\cite[Theorem 4.1]{wonham1968matrix}. Moreover, for any $\lambda\geq 0$, $\Sprior^{-1}+2\lambda\eye{}$ is positive definite and thus invertible, and $b$ is well defined by~\eqref{eq:minimum entropy eq b}.We now characterize the Lagrange multiplier. If $\lambda=0$ and, by complementary slackness, the minimizer possibly lies in the Wasserstein ball, \eqref{eq:minimum entropy eq A} yields 
\begin{equation}\label{eq:minimum entropy eq A lambda 0}
    A\Sref^{-1}A=\Sprior^{-1},
\end{equation}
which has a unique positive definite solution
\begin{equation}\label{eq:minimum entropy sol A lambda 0}
    A
    =
    \Sref^{1/2}(\Sref^{-1/2}\Sprior^{-1}\Sref^{-1/2})^{1/2}\Sref^{1/2}
    =
    \Sref^{1/2}(\Sref^{1/2}\Sprior\Sref^{1/2})^{-1/2}\Sref^{1/2}.
\end{equation}
Overall, 
\begin{align*}
    \optMap{\muref}{\mu^\ast}(x)
    =
    (\optMap{\mu^\ast}{\muref})^{-1}(x)
    \overset{\eqref{eq:minimum entropy inverse optimal map}}&{=}
    A^{-1}x-A^{-1}b
    \\
    \overset{\eqref{eq:minimum entropy eq b}}&{=}
    A^{-1}x-\Sprior(-\Sprior^{-1}\mprior+A\Sref^{-1}\mref)
    \\
    \overset{\eqref{eq:minimum entropy eq A lambda 0}}&{=}
    A^{-1}x-\Sprior(-\Sprior^{-1}\mprior+\Sprior^{-1}A^{-1}\mref)
    \\
    &=
    A^{-1}(x-\mref)+\mprior
    \\
    \overset{\eqref{eq:minimum entropy sol A lambda 0}}&{=}
    \Sref^{-1/2}(\Sref^{1/2}\Sprior\Sref^{1/2})^{1/2}\Sref^{-1/2}(x-\mref)+\mprior. 
\end{align*}
It is readily verified~\cite[§1.6.3]{Panaretos2020AnSpace} that $\optMap{\muref}{\mu^\ast}$ coincides with the optimal transport map from $\muref$ to $\muprior$. Thus, we would have $\mu^\ast=\muprior$, which however is not feasible since, by assumption, $\wassersteinDistance{2}{\muprior}{\muref}>\varepsilon$.
Thus, we restrict our attention to $\lambda>0$, which, by complementary slackness, implies that $\mu^\ast$ lies at the boundary of the Wasserstein ball (i.e., $\wassersteinDistance{2}{\mu^\ast}{\muref}=\radius$). Thus,
\begin{align*}
    \radius^2
    &=\int_{\reals^d}\norm{x-(A^{-1}x - A^{-1}b)}^2\d\muref(x)
    \\
    &=
    \int_{\reals^d}\norm{(\eye{}-A^{-1})x + (\Sprior^{-1}+2\lambda\eye{})^{-1}(-\Sprior^{-1}\mprior+A\Sref^{-1}\mref)}^2\d\muref(x)
    \\
    &=
    \begin{aligned}[t]
    &\trace\left((\eye{}-A^{-1})(\Sref+\mref\transpose{\mref})(\eye{}-A^{-1})\right)
    \\&+
    2\innerProduct{(\Sprior^{-1}+2\lambda\eye{})^{-1}(-\Sprior^{-1}\mprior+A\Sref^{-1}\mref)}{(\eye{}-A^{-1})\mref}
    \\&+
    \norm{(\Sprior^{-1}+2\lambda\eye{})^{-1}(-\Sprior^{-1}\mprior+A\Sref^{-1}\mref)}^2
    \end{aligned}
    \\
    &=
    \begin{aligned}[t]
    &\trace\left((\eye{}-A^{-1})^2(\Sref+\mref\transpose{\mref})\right)
    \\&+
    2\innerProduct{(\eye{}-A^{-1})(\Sprior^{-1}+2\lambda\eye{})^{-1}(-\Sprior^{-1}\mprior+A\Sref^{-1}\mref)}{\mref}
    \\&+
    \norm{(\Sprior^{-1}+2\lambda\eye{})^{-1}(-\Sprior^{-1}\mprior+A\Sref^{-1}\mref)}^2,
    \end{aligned}
\end{align*}
where we used~\eqref{eq:minimum entropy inverse optimal map} to express the optimal transport map between $\muref$ and $\mu^\ast$. 
Equivalently, we look for the positive roots of 
\begin{align*}
    f(\lambda)
    &\coloneqq
    \int_{\reals^d}\norm{x-(A^{-1}x - A^{-1}b)}^2\d\muref(x)-\varepsilon^2
    \\
    &=
    \begin{aligned}[t]
    &\trace\left((\eye{}-A^{-1})^2(\Sref+\mref\transpose{\mref})\right)
    \ifbool{compact}{\\&+}{+}
    2\innerProduct{(\eye{}-A^{-1})(\Sprior^{-1}+2\lambda\eye{})^{-1}(-\Sprior^{-1}\mprior+A\Sref^{-1}\mref)}{\mref}
    \\&+
    \norm{(\Sprior^{-1}+2\lambda\eye{})^{-1}(-\Sprior^{-1}\mprior+A\Sref^{-1}\mref)}^2-\varepsilon^2.
\end{aligned}
\end{align*}
We now prove that at least one positive root exists using the intermediate value theorem. Note that $f$ is continuous in $\lambda$, since $A$ is continuous $\lambda$, by~\eqref{eq:minimum entropy eq A} and~\eqref{eq:minimum entropy eq b}.
In particular, for $\lambda\to+\infty$ we have $A\to\eye{}$. Thus, $\lim_{\lambda\to+\infty}f(\lambda)=\radius^2<0$.
Conversely, we know from above that $\lambda=0$ yields $\mu^\ast=\muprior$ and thus
\begin{equation*}
    f(0)
    =\wassersteinDistance{2}{\mu^\ast}{\muref}^2-\varepsilon^2
    =\wassersteinDistance{2}{\muprior}{\muref}^2-\varepsilon^2
    >0.
\end{equation*}
Hence, the intermediate value theorem establishes the existence of $\lambda>0$ so that $f(\lambda)=0$.
Thus, the probability measure
\begin{align*}
    \mu^\ast
    =
    \pushforward{(A^{-1}x-A^{-1}b)}\muref
    \overset{\eqref{eq:minimum entropy eq b}}{=}
    \pushforward{(A^{-1}x-(\Sprior^{-1}+2\lambda\eye{})^{-1}(-\Sprior^{-1}\mprior+A\Sref^{-1}\mref))}\muref,
\end{align*}
where $A$ and $b$ result from~\eqref{eq:minimum entropy eq A} and~\eqref{eq:minimum entropy eq b}, satisfies sufficient conditions for optimality. Clearly, $\mu^\ast$ is also Gaussian, being the pushforward to a Gaussian probability measure through a linear (and invertible) map. 
Finally, the \gls{acr:kl} divergence is convex along generalized geodesics (by~\cref{prop:divergence}) and so its convexity parameter is $\alpha=0$.  Thus, $\lambda>0$ implies uniqueness of $\mu^\ast$, which is therefore the strict minimizer of~\eqref{eq:min entropy wasserstein ball}. As a byproduct, we also have uniqueness of $A$ and, by~\eqref{eq:minimum entropy eq A}, of $\lambda$. Finally, given that $\lambda>0$, complementary slackness implies that $\mu^\ast$ necessarily lies at the boundary of the Wasserstein ball.
\end{proof}}

\usepackage{epstopdf}
\ifpdf
  \DeclareGraphicsExtensions{.eps,.pdf,.png,.jpg}
\else
  \DeclareGraphicsExtensions{.eps}
\fi

\usepackage[shortlabels]{enumitem}
\setlist[enumerate]{leftmargin=*,label=(\roman*)}
\setlist[itemize]{leftmargin=*}

\crefname{subsection}{section}{sections}
\Crefname{subsection}{Section}{Sections}

\newsiamremark{remark}{Remark}
\newsiamremark{hypothesis}{Hypothesis}
\crefname{hypothesis}{Hypothesis}{Hypotheses}
\newsiamthm{claim}{Claim}
\newsiamthm{example}{Example}
\newsiamthm{assumption}{Assumption}
\headers{First-order Conditions in the Wasserstein Space}{N. Lanzetti, S. Bolognani, F. Dörfler}

\title{First-order Conditions for Optimization in the Wasserstein Space\thanks{Published in the \emph{SIAM Journal on Mathematics of Data Science}, 7(1), 274-300, 2025. \\ \url{https://doi.org/10.1137/23M156687X}
\funding{This work was supported as a part of NCCR Automation, a National Centre of Competence in Research, funded by the Swiss National Science Foundation (grant number 51NF40\_225155).}}}

\author{Nicolas Lanzetti\thanks{Automatic Control Laboratory, ETH Zürich
  (\email{\{lnicolas,bsaverio,dorfler\}@ethz.ch}, \url{http://people.ee.ethz.ch/\~lnicolas}).}
\and Saverio Bolognani\footnotemark[2]
\and Florian Dörfler\footnotemark[2]}

\usepackage{amsopn}

\def\cost{J}
\def\constraint{K}
\def\multiplier{\lambda}
\def\refmu{\hat\mu}
\def\radius{\varepsilon}
\def\weight{\rho}
\def\muprior{\mu_\mathrm{p}}
\def\muref{\mu_\mathrm{r}}
\def\Vprior{V_\mathrm{p}}
\def\Vref{V_\mathrm{r}}
\def\Sprior{\Sigma_\mathrm{p}}
\def\Sref{\Sigma_\mathrm{r}}
\def\mprior{m_\mathrm{p}}
\def\mref{m_\mathrm{r}}

\ifpdf
\hypersetup{
  pdftitle={First-order Conditions for Optimization in the Wasserstein Space},
  pdfauthor={N. Lanzetti, S. Bolognani, F. Dörfler}
}
\fi

\newbool{showproofs} 
\boolfalse{showproofs}
\newbool{arxiv}
\boolfalse{arxiv}
\newbool{compact}
\booltrue{compact}
\ifbool{showproofs}{\boolfalse{compact}}{}
\ifbool{arxiv}{\boolfalse{compact}}{}

\begin{document}

\maketitle

% REQUIRED
\begin{abstract}
We study first-order optimality conditions for constrained optimization in the Wasserstein space, whereby one seeks to minimize a real-valued function over the space of probability measures endowed with the Wasserstein distance. Our analysis combines recent insights on the geometry and the differential structure of the Wasserstein space with more classical calculus of variations. We show that simple rationales such as ``setting the derivative to zero'' and ``gradients are aligned at optimality'' carry over to the Wasserstein space. We deploy our tools to study and solve optimization problems in the setting of distributionally robust optimization and statistical inference. The generality of our methodology allows us to naturally deal with functionals, such as mean-variance, Kullback-Leibler divergence, and Wasserstein distance, which are traditionally difficult to study in a unified framework. 
\end{abstract}

% REQUIRED

% REQUIRED

\section{Introduction}\label{sec:intro}
Many problems in artificial intelligence, machine learning, and optimization under uncertainty can be cast as optimization problems over the space of probability measures of the form
\begin{equation}\label{eq:problem intro}
    \inf_{\mu\in\Pp{}{\reals^d}}\{J(\mu): K(\mu)\leq 0\},
\end{equation}
where $\Pp{}{\reals^d}$ is the space of Borel probability measures on $\reals^d$, and $J:\Pp{}{\reals^d}\to\reals$ and $K:\Pp{}{\reals^d}\to\reals$ are real-valued functions over the space of probability measures.
Examples of interest include: 

\begin{example}[Distributionally robust optimization]
\label{ex:dro}
In decision-making problems, one seeks a decision $w\in\mathcal{W}\subseteq\reals^n$ which minimizes the expected cost $\expectedValue{\mu}{f(w,x)}$ of a function $f:\mathcal{W}\times\reals^d\to\reals$, where the quantity $x$ is random with distribution $\mu\in\Pp{}{\reals^d}$. Since the probability measure $\mu$ is rarely known in practice, \gls{acr:dro} aims to minimize the \emph{worst-case} cost over the ambiguity set $\mathcal{P}\subset\Pp{}{\reals^d}$:
\begin{equation*}
    \inf_{w\in\mathcal{W}}\sup_{\mu\in\mathcal{P}}\expectedValue{\mu}{f(w,x)}.
\end{equation*}
Ambiguity sets are usually designed to ensure (statistical) performance guarantees and computational tractability; see~\cite{Blanchet2019,gao2022distributionally,kuhn2019wasserstein,peyman2018datadriven,Rahimian2019} and references therein. A popular choice of ambiguity sets (also subject to study in this work, see~\cref{sec:dro}) are Wasserstein balls, which include all probability measures within a given Wasserstein distance from a prescribed probability measure.  
\end{example}

\begin{example}[Statistical inference]
According to Kullback's principle of minimum cross-entropy~\cite{kullback1959information,shore1980axiomatic,shore1981properties}, the inference of a probability measure $\mu$ given a prior $\bar\mu$ over a class of probability measures $\mathcal{P}$ results from the minimization of the \gls{acr:kl} divergence: 
\begin{equation*}
    \inf_{\mu\in\mathcal{P}}\kullbackLeibler{\mu}{\bar\mu}\coloneqq\int_{\reals^d}\log\left(\frac{\rho(x)}{\bar\rho(x)}\right)\bar\rho(x)\d x,
\end{equation*}
where $\rho$ and $\bar\rho$ are the densities of $\mu$ and $\bar\mu$. Here, $\mathcal{P}$ usually encodes moment constraints, the class of the probability measure, or a ball around a prescribed probability measure.
\end{example}

\begin{example}[Maximum likelihood deconvolution]
Consider the problem of estimating an unknown probability measure $\mu\in\Pp{}{\reals^d}$ based on corrupted observations $Y_1,\ldots,Y_n$, where $Y_i=X_i+Z_i$, $X_1,\ldots,X_n$ are independent copies of $X\sim\mu$, and the errors $Z_1,\ldots,Z_n$ are independent copies of $Z\sim\bar\mu$ (with probability density $\bar\rho$) and independent of $X_1,\ldots,X_n$. A natural candidate is the maximum-likelihood estimator
\begin{equation}\label{eq:mle}
    \hat\mu
    \coloneqq\argmax_{\mu\in\mathcal{P}}
    \sum_{i=1}^n\log\left((\bar\rho\conv\mu)(Y_i)\right)
    \coloneqq
    \sum_{i=1}^n\log\left(\int_{\reals^d}\bar\rho(Y_i-x)\d\mu(x)\right),
\end{equation}
where $\mathcal{P}\subseteq\mathcal{P}(\reals^d)$ is a class of probability measures. 
It is well known~\cite{Rigollet2018} that~\eqref{eq:mle} can be reformulated as the (unconstrained) minimization of the entropic optimal transport distance~\cite{carlier2017convergence,peyre2019computational,rigollet2022sample} from the empirical probability measure $\frac{1}{n}\sum_{i=1}^n\diracMeasure{Y_i}$.
\end{example}

\begin{comment}
\begin{example}[Reinforcement learning]
In policy-based reinforcement learning, policies are represented as probability measures over the available actions. 
To maximize the expected discounted reward, it is common to iteratively solve the optimization problem
\begin{equation}\label{eq:rl}
\begin{aligned}
    \sup_{\mu_s}\left\{\expectedValue{\rho,\mu_s}{A(s,a)}:
    \expectedValue{\rho}{d(\mu_s,\hat\mu_s)}\leq\radius\right\},
\end{aligned}
\end{equation}
where $A:\mathcal{S}\times\mathcal{A}\to\reals$ is (an estimate of) the advantage function, $\hat\mu_s$ is the current policy at state $s\in\mathcal{S}$, $\rho$ is (an estimate of) the state visitation frequency, $d:\Pp{}{\mathcal{A}}\times\Pp{}{\mathcal{A}}\to\nonnegativeReals$ is a function quantifying the distance between two policies, and $\mathcal{S}, \mathcal{A}$ are the state space and the action space.
Efficient routines to solve~\eqref{eq:rl} are crucial to maintain stability and computational tractability of the policy update, and thus of the reinforcement learning algorithms~\cite{Anonymous2022,Schulman2015}.
\end{example}
\end{comment}

In general, the optimization problem~\eqref{eq:problem intro} can be formulated (or relaxed) as an infinite-dimensional optimization problem in the (vector) space of signed measures, whereby non-negativity and normalization are introduced as independent constraints. This augmentation permits to access the rich theory of optimization in Banach spaces; e.g., see~\cite{guignard1969generalized,Kurcyusz1976,Leuenberger1997,Maurer1979,Molchanov2004} and references therein.
For instance, when both $J$ and $K$ are restricted to be the expected value of some real-valued function, \eqref{eq:problem intro} culminates in the so-called problem of moments, a particular instance of conic linear problems in the space of signed measures for which powerful dual reformulations exist~\cite{Shapiro2001,Lasserre2009MomentsApplications,klerk2019survey,Schmuedgen2017}.
Recent years also witnessed significant efforts in studying~\eqref{eq:problem intro} in the context of \gls{acr:dro}, whereby one seeks to evaluate the worst-case expected value of a real-valued function over an ambiguity set, oftentimes defined via moments, the \gls{acr:kl} divergence, or the Wasserstein distance~\cite{Ben-Tal2013,Blanchet2019,gao2022distributionally,peyman2018datadriven,Rahimian2019,Wiesemann2014}. In many cases, \gls{acr:dro} problems admit dual reformulations, which sometimes even reduce to a tractable finite-dimensional convex optimization problem. 
By using the expected value as a local linear approximation of a more general sufficiently well-behaved functional over the probability space,~\cite{Kent2021} develops a Frank-Wolfe algorithm in probability spaces for the unconstrained version of \eqref{eq:problem intro}. In most cases, though, dual reformulations only hold when $\cost$ is an expected value, and do not readily generalize to arbitrary constraints.
Alternatively,~\eqref{eq:problem intro} can be studied via von Mises calculus~\cite{mises1947asymptotic}, first introduced in the field of asymptotic statistics, and recently leveraged to study various machine learning algorithms~\cite{chu2019probability} and nonlinear \gls{acr:dro} problems~\cite{sheriff2023nonlinear}.
Other approaches are based on input convex neural network~\cite{alvarez2021optimizing}, embedding of the probability space into a Hilbert space~\cite{dai2017learning}, and Stein variational gradient descent~\cite{liu2016stein}.

In this work, we adopt a different approach, and study~\eqref{eq:problem intro} through the lens of optimal transport and its differential properties. 
The theory of optimal transport, dating back to the seminal works of Monge~\cite{Monge1781MemoireRemblais} and Kantorovich~\cite{Kantorovich1942OnRussian}, provides us with a way to quantify the distance between probability measures, and thus to define a metric space (the so-called Wasserstein space) of probability measures~\cite{Ambrosio2008a,Santambrogio2015,Villani2009a}.
Notably, this is sufficient to endow the probability space with a differential structure. This enlightening theory, pioneered by~\cite{Jordan1998} and later formalized by~\cite{Ambrosio2008a}, culminates in the rigorous formulation of gradient flows in probability spaces; see~\cite{Santambrogio2017EuclideanOverview} for an introduction. Initially, this theory has been widely used to study partial differential equations~\cite{carrillo2022primal,figalli2018global,Otto1996DoublyDescent,Otto2001TheEquation}. More recently, there has been growing interest in the applications of Wasserstein gradient flows for optimization purposes.

For instance, \cite{Chewi2020} relies on Wasserstein gradients to devise computational methods to compute the ``average'' (so-called Wasserstein barycenter) of a collection of probability measures, implicitly defined as an unconstrained optimization problem over the space of probability measures~\cite{Agueh2010,Panaretos2020AnSpace}. 
Other applications in machine learning include the analysis of (over-parametrized) neural network~\cite{bach2021gradient,Chizat2018OnTransport,chizat2022infinite}, approximate inference~\cite{frogner2020approximate}, variational inference~\cite{lambert2022variational}, back-propagation of gradients through discrete random variables~\cite{cheng2019straight}, and policy optimization in reinforcement learning~\cite{zhang2018policy}.
In~\cite{Arbel2019Maximum,Salim2020TheWasserstein}, Wasserstein gradient flows are used to design iterative algorithms for the unconstrained version of~\eqref{eq:problem intro}.
In~\cite{Bonnet2019a,bonnet2021necessary,Bonnet2019}, (sub-)differential calculus in the Wasserstein space is used to formulate a Pontryagin Maximum Principle for (constrained) optimal control problems in the Wasserstein space, where the system dynamics is a transport equation with non-local velocity.
Our work also hinges upon (sub-)differential calculus in the Wasserstein space to study first-order necessary and sufficient conditions for optimality of the static optimization problem~\eqref{eq:problem intro}. In particular, our definitions of Wasserstein subdifferential and gradient (see~\cref{def:wasserstein differentiable,def:wasserstein subdifferentiable}) coincide with~\cite[Definitions 2.11 and 2.12]{Bonnet2019a} and, on measures with compact support, with their localized versions~\cite[Definitions 3.1 and 3.3]{bonnet2021necessary}. They are special cases of the more general definitions of extended Wasserstein subdifferential and gradient, introduced in~\cite[Definition 10.3.1]{Ambrosio2008a} and used in~\cite[Definitions 5 and 7]{Bonnet2019}.

\begin{comment}
The closest approach to ours is~\cite{Bonnet2019,Bonnet2019a,bonnet2021necessary}, which study necessary conditions ``à la Pontryagin'' for optimality of control problems in the Wasserstein space. While we focus on optimality conditions for the static optimization problem~\eqref{eq:problem intro} and not on a Pontryagin Maximum Principle for optimal control problems, our analysis also relies on the (sub-)differential calculus in the Wasserstein space. In particular, our definitions of Wasserstein subdifferential and gradient (cf.~\cref{def:wasserstein differentiable,def:wasserstein subdifferentiable}) coincide with~\cite[Definitions 2.11 and 2.12]{Bonnet2019a} and, on measures with compact support, with~\cite[Definitions 3.1 and 3.3]{bonnet2021necessary}, and are special cases of the more general~\cite[Definitions 5 and 7]{Bonnet2019} and~\cite[Definition 10.3.1]{Ambrosio2008a}.
%While our setting is different, we share some of the fundamentals of this work. 
\end{comment}

The study of~\eqref{eq:problem intro} in the probability space, and not in the space of signed measures, allows us to formulate simple, intuitive, and interpretable necessary and sufficient conditions for optimality, which oftentimes formally resemble their Euclidean counterparts (e.g., ``set the gradient to zero''). In particular, we do not need to repeatedly deal with normalization and non-negativity of the signed measures.

\paragraph{Contributions}
We derive novel necessary and sufficient KKT-type conditions for optimality of~\eqref{eq:problem intro}. Our analysis combines recent advances in the geometry and the differential structure of the Wasserstein space (for which we also extend several fundamental results) with classical methods from calculus of variations.
We complement our theoretic analysis with many examples of functionals over the Wasserstein space, for which we study the differentiability properties and provide expressions for their Wasserstein (sub)gradient. Finally, we show that our methodology can be utilized to solve mean-variance and mean-standard deviation \gls{acr:dro} and statistical inference with the \gls{acr:kl} divergence, sometimes even in closed form. While not representing the core of this work, the study of these problems is of independent interest. 

\begin{comment}
More specifically, our contributions are as follows. 
\begin{enumerate}
    \item 
    We derive novel results on the differential structure of the Wasserstein and on the differentiability of functionals defined on the probability spaces. 
    
    We provide a concise and self-contained review of the Wasserstein space, its geometry, and its differential structure. Our review bases on both existing results as well as on various novel extensions and it is geared towards solving~\eqref{eq:problem intro}.
    We do not content ourselves with enunciating the theoretic results, but present many examples of functionals over the Wasserstein space, and characterize their properties such as continuity and differentiability. 
    
    \item 
    We derive novel necessary and sufficient KKT-type conditions for optimality of~\eqref{eq:problem intro} in the case of equality constraints and specific inequality constraints. Our analysis combines recent advances in the geometry and the differential structure of the Wasserstein space with classical methods from calculus of variations.
    
    \item 
     We show that our methodology can be used to solve various problems (including unsolved ones) in \gls{acr:dro} and statistical inference with the \gls{acr:kl} divergence, sometimes even in closed form. While not representing the core of this work, the study of these specific problems is of independent interest. 
\end{enumerate}
\end{comment}

\paragraph{Organization}
This paper unfolds as follows. In~\cref{sec:preliminaries}, we recall and extend preliminaries in optimal transport and the Wasserstein space. In~\cref{sec:optimality}, we study necessary and sufficient conditions for optimality of~\eqref{eq:problem intro}. In~\cref{sec:dro,sec:kl}, we deploy our methodology to solve various optimization problems in \gls{acr:dro} and statistical inference with the \gls{acr:kl} divergence. Finally, \cref{sec:conclusions} draws the conclusions of this paper. 
All proofs are deferred to the appendix.
% \ifbool{showproofs}{}{\ifbool{arxiv}{All proofs are deferred to the appendix.}{All proofs are in the supplementary material.}}

\paragraph{Notation}
We will use several classes of real-valued continuous functions $f:\reals^d\to\reals$, whereby $\reals^d$ is endowed with the Euclidean inner product $\innerProduct{\cdot}{\cdot}$ and induced norm $\Vert\cdot\Vert$.
We denote by $\C{p}{\reals^d}$ the space of $p$-time continuously differentiable functions, by $\Cb{\reals^d}$ the space of bounded continuous functions, by $\Ccinf{\reals^d}$ the space of smooth (i.e., infinitely differentiable) functions with compact support, and by $\Wrloc{1}{\reals^d; \mu}$ the (Sobolev) space of locally $\mu$-integrable functions whose gradient is locally $\mu$-integrable. 
A function $f:\reals^d\to\reals$ is $\alpha$-convex, with $\alpha\in\reals$, if $f-\alpha\norm{x}^2$ is convex; in particular, we allow the convexity parameter $\alpha$ to be non-positive.
We denote the gradient of $f:\reals^d\to\reals$ by $\gradient{}f$ and its Hessian by $\gradient{}^2f$.
Similarly, we denote the Jacobian of $g:\reals^d\to\reals^d$ by $\gradient{}g$.
We use the notation $\Lp{2}{\reals^d;\mu}$ to denote the space of real-valued $\mu$-measurable functions with bounded $2$-norm (where integration is w.r.t. $\mu$) and the notation $\Lp{2}{\reals^d,\reals^d;\mu}$ for functions $f:\reals^d\to\reals^d$. We denote the identity function on $\reals^d$ by $\Id$ and canonical projections by $\proj$ (e.g., $\proj_1:\reals^d\times\reals^d\to\reals^d, (x,y)\mapsto x$). 
%A function $f:\reals^d\to\reals$ is \emph{convex} if $f(tx+(1-t)y)\leq tf(x)+(1-t)f(y)$ and \emph{strongly convex} if there exists $\alpha>0$ such that $f-\frac{\alpha}{2}\norm{x}^2$ is convex.
%If $f\in\Ctwo{\reals^d}$, then $f$ is convex if and only if $\gradient{}^2f(x)$ is positive semidefinite for all $x\in\reals^d$ and it is strongly convex if there exists $\alpha>0$ such that $\gradient{}^2f(x)-\frac{\alpha}{2}\eye{}$ is positive semidefinite for all $x\in\reals^d$.
We say $f(x)=\onotation{g(x)}$ for non-negative $f$ and positive $g$ if $\lim_{x\to 0}\frac{f(x)}{g(x)}=0$; in particular, $\lim_{x\to 0}\frac{\onotation{x}}{x}=0$.
%For a square matrix $A\in\reals^{d\times d}$, we denote its determinant by $\det(A)$, its spectrum (i.e., set of eigenvalues) by $\spec(A)$, its maximum and minimum eigenvalues by $\maxEigenvalue{A}$ and $\minEigenvalue{A}$, and its induced norm by $\norm{A}\coloneqq\max_{\norm{x}=1}\norm{Ax}$.
Finally, for a real number $\alpha\in\reals$, we define the negative part via $\negativePart{\alpha}\coloneqq |\min\{\alpha,0\}|\geq 0$.

\section{Optimal Transport and the Wasserstein Space}\label{sec:preliminaries}

\def\cost{J}

In this section, we review and extend various preliminaries in optimal transport and the Wasserstein space, including some background (\cref{subsec:background in mt ot ws}), interpolation and geodesic convexity (\cref{subsec:interpolation and geodesic convexity}), differential calculus (\cref{subsec:differential calculus in the wasserstein space}), and examples of smooth functionals (\cref{subsec:examples of smooth functionals}). 
\ifbool{arxiv}{The proofs of most of the statements of this section are deferred to~\cref{app:proofs}.}{}

\subsection{Background in Measure Theory, Optimal Transport, and Wasserstein Space}\label{subsec:background in mt ot ws}
We start with basics in measure theory and optimal transport. \ifbool{showproofs}{We defer some technical details to~\cref{app:preliminaries}, and w}{W}e refer the reader to~\cite{Ambrosio2008a,figalli2021invitation,Santambrogio2015,Villani2009a} for a comprehensive review. 
We denote by $\spaceProbabilityMeasures{\reals^d}$ the set of Borel probability measures on $\reals^d$, and by
% \ifbool{compact}{$\Pp{2}{\reals^d}
%     \coloneqq 
%     \left\{\mu\in\spaceProbabilityMeasures{\reals^d}: \int_{\reals^d}\norm{x}^2\d\mu(x)<+\infty\right\}$}{
\begin{equation*}
    \Pp{2}{\reals^d}
    \coloneqq 
    \left\{\mu\in\spaceProbabilityMeasures{\reals^d}: \int_{\reals^d}\norm{x}^2\d\mu(x)<+\infty\right\}
\end{equation*}
the set of probability measures with finite second moment. 
We use the notation $\mu\ll\nu$ to indicate that $\mu$ is \emph{absolutely continuous} with respect to $\nu$. In particular, we denote by $\Ppabs{2}{\reals^d}\subset\Pp{2}{\reals^d}$ the space of absolutely continuous probability measures (with respect to the Lebesgue measure on $\reals^d$, denoted by $\lebesgueMeasure{d}$).
The \emph{pushforward} of a measure $\mu\in\spaceProbabilityMeasures{\reals^d}$ through a Borel map $T:\reals^d\to\reals^d$, denoted by $\pushforward{T}\mu$, is defined by $(\pushforward{T}\mu)(A)\coloneqq\mu(T^{-1}(A))$ for all Borel sets $A\subseteq\reals^d$. We recall that for any $f:\reals^d\to\reals$ $\pushforward{T}\mu$-integrable (or non-negative measurable/Borel bounded)
\begin{equation}\label{eq:integral and pushforward}
    \int_{\reals^d}f(x)\d(\pushforward{T}\mu)(x)
    =
    \int_{\reals^d}f(T(x))\d\mu(x).
\end{equation}
Following \cite{Villani2009a}, we define two notions of convergence in $\Pp{2}{\reals^d}$: 
\begin{enumerate}
    \item \emph{narrow convergence}: $(\mu_n)_{n\in\naturals}\subset\Pp{2}{\reals^d}$ converges narrowly to $\mu\in\Pp{2}{\reals^d}$ if
    \begin{equation*}
        \int_{\reals^d}\phi(x)\d\mu_n(x)\to\int_{\reals^d}\phi(x)\d\mu(x)\qquad\forall\phi\in\Cb{\reals^d}.
    \end{equation*}
    
    \item \emph{weak convergence in $\Pp{2}{\reals^d}$}: $(\mu_n)_{n\in\naturals}\subset\Pp{2}{\reals^d}$ converges weakly in $\Pp{2}{\reals^d}$ to $\mu\in\Pp{2}{\reals^d}$, denoted by $\mu_n\weakconvergence\mu$, if
    \begin{equation*}
    \begin{aligned}
        \int_{\reals^d}\phi(x)\d\mu_n(x)\to\int_{\reals^d}\phi(x)\d\mu(x) && \forall\phi\in\Cb{\reals^d} 
        \quad\text{and}\quad 
        \int_{\reals^d}x^2\d\mu_n(x)\to\int_{\reals^d}x^2\d\mu(x).
    \end{aligned}
    \end{equation*}
    Equivalently~\cite[Definition 6.8]{Villani2009a}, $\mu_n\weakconvergence\mu$ if for all continuous functions $\phi\in\C{0}{\reals^d}$ with $|\phi(x)|\leq C(1+\norm{x}^2)$ for $C\in\reals$ we have\ifbool{compact}{
    $\int_{\reals^d}\phi(x)\d\mu_n(x)\to\int_{\reals^d}\phi(x)\d\mu(x).$}{
    \begin{equation*}
        \int_{\reals^d}\phi(x)\d\mu_n(x)\to\int_{\reals^d}\phi(x)\d\mu(x).
    \end{equation*}}
\end{enumerate}
Of course, weak convergence in $\Pp{2}{\reals^d}$ implies narrow convergence. The converse, though, is not true, as $x\mapsto x^2\notin\Cb{\reals^d}$. 
\ifbool{compact}{}{
\begin{example}[Narrow convergence $\not\Rightarrow$ weak convergence in $\Pp{2}{\reals^d}$]
\label{ex:narrow convergence and weak convergence}
Let $(\mu_n)_{n\in\naturals}\subset\Pp{2}{\reals}$ be defined by $\mu_n\coloneqq(1-\frac{1}{n^2})\diracMeasure{0}+\frac{1}{n^2}\diracMeasure{n}$. Then, $\mu_n$ converges narrowly to $\diracMeasure{0}$: for any $\phi\in\Cb{\reals}$
\begin{equation*}
    \int_\reals\phi(x)\d\mu_n(x)=\left(1-\frac{1}{n^2}\right)\phi(0)+\frac{1}{n^2}\phi(n)\to\phi(0)=\int_{\reals}\phi(x)\d\delta_0(x),
\end{equation*}
where we used that $\phi$ is bounded. However, $\mu_n\not\weakconvergence\diracMeasure{0}$ as 
\begin{equation*}
    \int_{\reals}x^2\d\mu_n(x)=\frac{1}{n^2}n^2=1\not\to 0=\int_{\reals}x^2\d\diracMeasure{0}.
    \qedhere 
\end{equation*}
\end{example}}

Given $\mu,\nu\in\Pp{2}{\reals^d}$ we say that $\gamma\in\Pp{2}{\reals^d\times\reals^d}$ is a \emph{plan} (or \emph{coupling}) between $\mu$ and $\nu$ if $\pushforward{(\proj_1)}\gamma=\mu$ and $\pushforward{(\proj_2)}\gamma=\nu$, where $\proj_1,\proj_2:\reals^d\times\reals^d\to\reals^d$ are the canonical projections; equivalently for all $\phi,\psi:\reals^d\to\reals$ integrable
\begin{equation*}
\begin{aligned}
    \int_{\reals^d\times\reals^d}\phi(x)\d\gamma(x,y)
    =\int_{\reals^d}\phi(x)\d\mu(x)
    \quad\text{and}\quad 
    \int_{\reals^d\times\reals^d}\psi(y)\d\gamma(x,y)
    =\int_{\reals^d}\psi(y)\d\nu(y).\\
\end{aligned}
\end{equation*}
We denote by  $\setPlans{\mu}{\nu}\subset\spaceProbabilityMeasures{\reals^d\times\reals^d}$ the set of couplings of $\mu$ and $\nu$. Since the product measure $\mu\productMeasure\nu$ has marginals $\mu$ and $\nu$, respectively, $\setPlans{\mu}{\nu}$ is non-empty. 
We are now ready to recall the definition of the Wasserstein distance: 

\begin{definition}[Wasserstein distance~\protect{\cite[§7.1]{Ambrosio2008a}}]
Let $\mu,\nu\in\Pp{2}{\reals^d}$. We define the (type 2) \emph{Wasserstein distance} between $\mu$ and $\nu$ as
\begin{equation}\label{eq:wasserstein distance}
    \wassersteinDistance{2}{\mu}{\nu}
    \coloneqq
    \left(\min_{\gamma\in\setPlans{\mu}{\nu}}\int_{\reals^d\times\reals^d}\norm{x-y}^2\d\gamma(x,y)\right)^{\frac{1}{2}}.
\end{equation}
We denote by $\setOptimalPlans{\mu}{\nu}$ the set of couplings attaining the minimum in \eqref{eq:wasserstein distance}, and we call any $\gamma\in\setOptimalPlans{\mu}{\nu}$ an \emph{optimal transport plan} between $\mu$ and $\nu$.
\end{definition}

A compactness and lower semi-continuity argument (with respect to the topology induced by \emph{narrow} convergence) shows that $\setOptimalPlans{\mu}{\nu}$ is non-empty: The minimum in \eqref{eq:wasserstein distance} is attained by some optimal transport plan $\gamma$~\cite[§4]{Villani2009a}.
If $\gamma\in\setOptimalPlans{\mu}{\nu}$ is of the form $\pushforward{(\Id,\optMap{\mu}{\nu})}\mu$ for some measurable $\optMap{\mu}{\nu}:\reals^d\to\reals^d$, we say that $\gamma$ is induced by an optimal transport \emph{map} which pushes $\mu$ into $\nu$; i.e., $\nu=\pushforward{(\optMap{\mu}{\nu})}\mu$, $\optMap{\mu}{\nu}$ yields the minimum transportation cost, and $\optMap{\mu}{\nu}$ is an optimal transport map between $\mu$ and $\nu$. 
The celebrated Brenier's theorem~\cite[Theorem 6.2.4]{Ambrosio2008a} sheds light on existence of optimal transport maps. 

\begin{theorem}[Brenier's theorem~\protect{\cite[Theorem 6.2.4]{Ambrosio2008a}}]\label{thm:brenier}
Let $\mu,\nu\in\Pp{2}{\reals^d}$, and assume that $\mu$ is absolutely continuous w.r.t. the Lebesgue measure. Then, there exists a unique optimal transport plan $\gamma$, induced by a unique optimal transport map $\optMap{\mu}{\nu}$, and $\optMap{\mu}{\nu}=\nabla\phi$ $\mu$-almost everywhere for some convex function $\phi:\reals^d\to\reals$.
\end{theorem}

\cref{thm:brenier} suggests that, whenever $\mu\ll\lebesgueMeasure{d}$, it is safe to define \emph{the} optimal transport map (inducing \emph{the} optimal transport plan) between $\mu$ and $\nu$. Its inverse, if it exists, reveals the optimal transport map in ``reverse direction''.

\begin{proposition}[inverse of optimal transport maps]
\label{prop:inverse optimal transport map}
Let $\mu,\nu\in\Pp{}{\reals^d}$. Assume an optimal transport map $\optMap{\mu}{\nu}$ exists, and it is invertible $\mu$-almost everywhere. Then, $\optMap{\nu}{\mu}$ exists, and it equals $(\optMap{\mu}{\nu})^{-1}$ $\nu$-almost everywhere.
\end{proposition}

\ifbool{showproofs}{\booltrue{showproofinverseoptimaltransportmap}}{
\ifbool{arxiv}{\boolfalse{showproofinverseoptimaltransportmap}}{}} % do not show in arxiv
\ifbool{showproofinverseoptimaltransportmap}{\proofinverseoptimaltransportmap{booltrue}}{}

\begin{remark}
When both $\mu$ and $\nu$ are absolutely continuous, invertibility of $\optMap{\mu}{\nu}$ and~\cref{prop:inverse optimal transport map} follow directly from Brenier's theorem (\cref{thm:brenier}). \cref{prop:inverse optimal transport map} suggests that the result continues to hold if the probability measures are not absolutely continuous but still admit a $\mu$-almost everywhere invertible optimal transport map.
% does not provide conditions for ($\mu$-a.e.) invertibility of $\optMap{\mu}{\nu}$, but it assumes it. A sufficient condition for $\mu$-a.e. invertibility is the absolutely continuity of $\nu$. Then, by , $\optMap{\nu}{\mu}$ exists and $\optMap{\nu}{\mu}=(\optMap{\mu}{\nu})^{-1}$ $\nu$-a.e. Nonetheless, the result 
\end{remark}

There is one final important setting where optimal transport maps exist. 

\begin{proposition}[pushforward via gradients of convex functions~\protect{\cite[Theorem 1.48]{Santambrogio2015}}]\label{prop:brenier2}
Let $\mu\in\Pp{}{\reals^d}$ and $T=\gradient{}\phi$ for some $\phi:\reals^d\to\reals$ convex and differentiable $\mu$-almost everywhere. If $\norm{T}_{\Lp{2}{\reals^d,\reals^d;\mu}}<+\infty$, then an optimal transport plan between $\mu$ and $\pushforward{T}\mu\in\Pp{}{\reals^d}$ is given by the optimal transport map $T$.
\end{proposition}

\cref{prop:brenier2} is non-obvious: There might exist another transport map $\tilde T$ (or a transport plan) yielding a lower transport cost between $\mu$ and $\pushforward{T}\mu$.~\cref{prop:brenier2} can be used to define a perturbation of measures, which will be a crucial building block for our variational analysis.

\begin{lemma}[perturbation of probability measures~\protect{\cite[Corollary 1.1]{bonnet2019optimal}}]
\label{lemma:perturbation transport map}
Let $\mu\in\Pp{2}{\reals^d}$ and $\psi\in\Ccinf{\reals^d}$. Then, there exists $\bar s>0$ such that $\Id+s\gradient{}\psi$ is an optimal transport map from $\mu$ to $\pushforward{(\Id+s\gradient{}\psi)}\mu$ for all $s\in(-\bar s,+\bar s)$, i.e.,
\begin{equation*}
    \optMap{\mu}{\pushforward{(\Id+s\gradient{}\psi)}\mu}
    =
    \Id+s\gradient{}\psi.
\end{equation*}
Further, $\pushforward{(\Id+s\gradient{}\psi)}\mu\in\Pp{2}{\reals^d}$, and the Wasserstein distance between $\mu$ and $\pushforward{(\Id+s\gradient{}\psi)}\mu$ is
\begin{equation*}
    \wassersteinDistance{2}{\mu}{\pushforward{(\Id+s\gradient{}\psi)}\mu}
    =
    |s|\LpNorm{\gradient{}\psi}{2}{\reals^d,\reals^d;\mu}.
\end{equation*}
\end{lemma}

\ifbool{showproofs}{\booltrue{showproofperturbationtransportmap}}{
\ifbool{arxiv}{\boolfalse{showproofperturbationtransportmap}}{}} % do not show in arxiv
\ifbool{showproofperturbationtransportmap}{\proofperturbationtransportmap{booltrue}}{}

%---------------------------------------------------------------%
%---------------------------------------------------------------%
\begin{comment}
Finally, we can provide a lower bound on the Wasserstein distance, based on mean and covariance: 

\begin{theorem}[Gelbrich's bound~\cite{Gelbrich1990}]\label{thm:gelbrich}
Let $\mu,\nu\in\Pp{2}{\reals^d}$ with mean $m_1,m_2\in\reals^d$ and covariance matrix $\Sigma_1,\Sigma_2\in\reals^d\times\reals^d$, with $\Sigma_1,\Sigma_2$ symmetric and positive semi-definite.
Then,
\begin{equation*}
    \wassersteinDistance{2}{\mu}{\nu}
    \geq 
    \sqrt{\Vert m_1-m_2\Vert^2 + \Vert\Sigma_1^{1/2}-\Sigma_2^{1/2}\Vert_\mathrm{Fro}^2},
\end{equation*}
where $\Sigma^{1/2}\in\reals^{d\times d}$ is the unique matrix such that $(\Sigma^{1/2})^2=\Sigma$ and $\norm{\cdot}_\mathrm{Fro}$ is the Frobenius norm, defined by $\norm{A}_\mathrm{Fro}\coloneqq\sum_{i,j}|a_{ij}|$.
\end{theorem}
The bound becomes exact in several cases of interest; e.g., elliptical distributions with the same generator. We refer to~\cite{Gelbrich1990,Ngyuyen2021} for more details.
\end{comment}
%---------------------------------------------------------------%
%---------------------------------------------------------------%

Finally, it is well known that the Wasserstein distance $\wassersteinDistance{2}{\cdot}{\cdot}$ defines a distance on $\Pp{2}{\reals^d}$ (but not on $\spaceProbabilityMeasures{\reals^d}$), so that $(\Pp{2}{\reals^d},\wassersteinDistance{2}{\cdot}{\cdot})$ is a metric space, often called  the \emph{Wasserstein space}~\cite{Ambrosio2008a,Villani2009a,Santambrogio2015}. 
We also recall that the Wasserstein distance metrizes weak convergence in $\Pp{2}{\reals^d}$~\cite[Theorem 6.9]{Villani2009a}; i.e., for $(\mu_n)_{n\in\naturals}\subset\Pp{2}{\reals^d}$ and $\mu\in\Pp{2}{\reals^d}$ we have $\lim_{n\to\infty}\wassersteinDistance{2}{\mu_n}{\mu}=0$ if and only if $\mu_n\weakconvergence\mu$.
\ifbool{compact}{}{
\begin{example}[\cref{ex:narrow convergence and weak convergence} revisited]
\label{ex:narrow convergence and weak convergence rivisited}
In~\cref{ex:narrow convergence and weak convergence}, $\wassersteinDistance{2}{\mu_n}{\diracMeasure{0}}=1$ for all $n$, so $\wassersteinDistance{2}{\mu_n}{\diracMeasure{0}}\not\to 0$, and thus $\mu_n$ cannot converge weakly in $\Pp{2}{\reals^d}$ to $\diracMeasure{0}$.
\end{example}}

\subsection{Interpolation and Geodesic Convexity}\label{subsec:interpolation and geodesic convexity}
The theory of optimal transport naturally provides us with a notion of interpolation in the Wasserstein space~\cite[§7]{Ambrosio2008a}. Let $\mu_0,\mu_1\in\Pp{2}{\reals^d}$ and $\gamma\in\setOptimalPlans{\mu_0}{\mu_1}$, and consider the curve
\begin{equation}\label{eq:geodesics}
    \mu_t
    \coloneqq 
    \pushforward{((1-t)\proj_1 + t\proj_2)}\gamma.
\end{equation}
Then, $\mu_t$ interpolates between $\mu_0$ ($t=0$) and $\mu_1$ ($t=1$). Since multiple optimal transport plans might exist, the interpolation $\mu_t$ between $\mu_0$ and $\mu_1$ is generally not unique. 
Curves of the form~\eqref{eq:geodesics} can be shown to be \emph{constant speed geodesics} in $\Pp{2}{\reals^d}$ (i.e., $\wassersteinDistance{2}{\mu_s}{\mu_t}=(t-s)\wassersteinDistance{2}{\mu_0}{\mu_1}$ for all $0\leq s\leq t\leq 1$), and all constant speed geodesics can be written as~\eqref{eq:geodesics}~\cite[Theorem 7.2.2]{Ambrosio2008a}.
Accordingly, a functional $\cost:\Pp{2}{\reals^d}\to\realsBar\coloneqq(-\infty,\infty]$ is $\alpha$-\emph{geodesically convex} (with $\alpha\in\reals$) if for all $\mu_0,\mu_1\in\Pp{2}{\reals^d}$ and all $t\in[0,1]$ there exists $\mu_t$ as in~\eqref{eq:geodesics} such that 
\begin{equation}\label{eq:geodesics equation}
    \cost(\mu_t)\leq (1-t)\cost(\mu_0) + t\cost(\mu_1)-\frac{\alpha}{2}t(1-t)\wassersteinDistance{2}{\mu_0}{\mu_1}^2.
\end{equation}
We will sometimes need a stronger notion of geodesic convexity: geodesic convexity \emph{along generalized geodesics}. Following~\cite[§9.2]{Ambrosio2008a}, a generalized geodesic joining $\mu_0$ and $\mu_1$ with base $\bar\mu$ is a curve 
\begin{equation}\label{eq:generalized geodesics}
    \mu_t
    \coloneqq 
    \pushforward{((1-t)\proj_2 + t\proj_3)}\gamma,
\end{equation}
where $\gamma\in\Pp{2}{\reals^d\times\reals^d\times\reals^d}$ is a coupling of $\bar\mu$, $\mu_0$, and $\mu_1$ which is optimal between $\bar\mu$ and $\mu_0$ (i.e., $\pushforward{(\proj_{1,2})}\gamma\in\setOptimalPlans{\bar\mu}{\mu_0}$) and between $\bar\mu$ and $\mu_1$ (i.e., $\pushforward{(\proj_{1,3})}\gamma\in\setOptimalPlans{\bar\mu}{\mu_2}$), but not necessarily optimal between $\mu_0$ and $\mu_1$. With the choice $\bar\mu=\mu_0$, we recover the geodesics~\eqref{eq:geodesics}. Accordingly, $\cost$ is $\alpha$-\emph{convex along generalized geodesics} if for all $\bar\mu, \mu_0,\mu_1\in\Pp{2}{\reals^d}$ there exists $\mu_t$ as in~\eqref{eq:generalized geodesics} such that 
\begin{equation*}
    \cost(\mu_t)\leq (1-t)\cost(\mu_0) + t\cost(\mu_1)-\frac{\alpha}{2}t(1-t)\int_{\reals^d\times\reals^d\times\reals^d}\norm{y-z}^2\d\gamma(x,y,z).
\end{equation*}
Convexity along generalized geodesics is stronger than geodesic convexity. In particular, if $\cost$ is $\alpha$-convex along generalized geodesics, then it suffices to pick $\bar\mu=\mu_0$ to prove that $\cost$ is $\alpha$-geodesically convex. If $\alpha=0$ we simply say that $\cost$ is geodesically convex (along generalized geodesics).
Finally, if~\eqref{eq:geodesics equation} holds for any $\gamma\in\setPlans{\mu}{\nu}$ (not necessarily optimal) in~\eqref{eq:geodesics}, then $\cost$ is $\alpha$-convex \emph{along any interpolating curve}. Clearly, this implies both geodesic convexity and convexity along generalized geodesics.

\subsection{Differential Calculus in the Wasserstein Space}\label{subsec:differential calculus in the wasserstein space}
Following~\cite[§10]{Ambrosio2008a} and~\cite{Bonnet2019a}, we now present the ``differential structure'' of the Wasserstein space.
Critically, $(\Pp{2}{\reals^d},\wassersteinDistance{2}{\cdot}{\cdot})$ does \emph{not} enjoy a linear structure, so classical notions of differentials (e.g., Fréchet) do not apply.
Henceforth, we consider lower semi-continuous (w.r.t. to weak convergence in $\Pp{2}{\reals^d}$) functionals of the form $\cost:\Pp{2}{\reals^d}\to\realsBar\coloneqq(-\infty,+\infty]$. To ease the notation, we define the effective domain of $\cost$ as
\ifbool{compact}{
$\effectiveDomain{\cost}\coloneqq\{\mu\in\Pp{2}{\reals^d}: \cost(\mu)<+\infty\}$}{
\begin{equation*}
    \effectiveDomain{\cost}
    \coloneqq\{\mu\in\Pp{2}{\reals^d}: \cost(\mu)<+\infty\}
\end{equation*}}
and tacitly assume that $\cost$ is proper (i.e., $\effectiveDomain{\cost}\neq\emptyset$).
Inspired by classical definitions of sub- and super-differentiability in Euclidean settings, sub- and super-differentials are defined as follows. 
\begin{definition}[Wasserstein sub- and super-differential~\protect{\cite[Definition 1.9]{Bonnet2019a}}]\label{def:wasserstein subdifferentiable}
Let $\mu\in\effectiveDomain{\cost}$. We say that a map $\xi\in\Lp{2}{\reals^d,\reals^d; \mu}$ belongs to the \emph{subdifferential} $\subdifferential{\cost}(\mu)$ of $\cost$ at $\mu$ if
\begin{equation}\label{eq:sub and superdifferential wasserstein}
    \cost(\nu)-\cost(\mu)
    \geq 
    \sup_{\gamma\in\setOptimalPlans{\mu}{\nu}}
    \int_{\reals^{d}\times\reals^d}
    \innerProduct{\xi(x)}{y-x}\,\d\gamma(x,y)
    +
    \onotation{\wassersteinDistance{2}{\mu}{\nu}}
\end{equation}
for all $\nu\in\Pp{2}{\reals^d}$. In this case, we say that $\xi$ is a \emph{Wasserstein subgradient} of $\cost$ at $\mu$.
Similarly, $\xi\in\Lp{2}{\reals^d,\reals^d;\mu}$ belongs to the \emph{superdifferential} $\superdifferential{\cost}(\mu)$ of $\cost$ at $\mu$ if $(-\xi)\in\subdifferential{(-\cost)}(\mu)$. In this case, we say that $\xi$ is a \emph{Wasserstein supergradient} of $\cost$ at $\mu$.
\end{definition}
By \cref{def:wasserstein subdifferentiable}, sub- and super-gradients are \emph{functions} in $\Lp{2}{}$. Intuitively, $\xi(x)$, with $\xi\in\subdifferential{\cost}(\mu)$, is the ``subgradient experienced by the particles of $\mu$ located at $x\in\reals^d$''.
It is now easy to define differentiability: 
\begin{definition}[differentiable functions~\protect{\cite[Definition 1.10]{Bonnet2019a}}]\label{def:wasserstein differentiable}
A functional $\cost:\Pp{2}{\reals^d}\to\reals$ is \emph{Wasserstein differentiable} at $\mu\in\Pp{2}{\reals^d}$ if the intersection of sub- and superdifferential is non-empty; i.e., if $\subdifferential{\cost}(\mu)\cap\superdifferential{\cost}(\mu)\neq\emptyset$.
In this case, we say $\gradient{\mu}\cost(\mu)\in\subdifferential{\cost}(\mu)\cap\superdifferential{\cost}(\mu)$ is a \emph{Wasserstein gradient} of $\cost$ at $\mu$, satisfying
\begin{equation*}
    \cost(\nu)-\cost(\mu)
    =
    \int_{\reals^{d}\times\reals^d}\innerProduct{\gradient{\mu}\cost(\mu)(x)}{y-x}\d\gamma(x,y)
    +
    \onotation{\wassersteinDistance{2}{\mu}{\nu}}
\end{equation*}
for any $\nu\in\Pp{2}{\reals^d}$ and any $\gamma\in\setOptimalPlans{\mu}{\nu}$.
\end{definition}

\cref{def:wasserstein differentiable} states that Wasserstein gradients are general functions in $\Lp{2}{\reals^d,\reals^d;\mu}$. As we shall see below, we can impose additional structure and dictate that (sub- and super-)gradients belong to the ``Wasserstein tangent space'', which is defined as follows~\cite{Ambrosio2008a}.

\begin{definition}[tangent space~\protect{\cite[Definition 8.4.1]{Ambrosio2008a}}]\label{def:tangent space}
The \emph{tangent space} at $\mu\in\Pp{2}{\reals^d}$ is the vector space
\begin{equation*}
    \tangentSpace{\mu}\Pp{2}{\reals^d}
    \coloneqq
    \overline{\{
    \gradient{}\psi: \psi\in\Ccinf{\reals^d}
    \}}^{\Lp{2}{\reals^d,\reals^d;\mu}},
\end{equation*}
where the closure is taken with respect to the $\Lp{2}{\reals^d,\reals^d;\mu}$ topology. 
\end{definition}

\begin{remark}\label{rem:tangent space}
A few remarks on~\crefrange{def:wasserstein subdifferentiable}{def:tangent space} are in order. 
First, \cref{def:tangent space} encapsulates the geometry of $(\Pp{2}{\reals^d},\wassersteinDistance{2}{\cdot}{\cdot})$: Gradients of functions $\psi\in\Ccinf{\reals^d}$ generate optimal transport maps by perturbation of the identity via $\pushforward{(\Id+\varepsilon\gradient{}\psi)}\mu$ for $\varepsilon$ sufficiently small; see \cref{lemma:perturbation transport map} above. Thus, they can be thought as tangent vectors to $(\Pp{2}{\reals^d},\wassersteinDistance{2}{\cdot}{\cdot})$.
Second, an equivalent (more related to optimal transport) definition of tangent space is
\begin{equation*}
    \tangentSpace{\mu}\Pp{2}{\reals^d}
    =
    \overline{\{
    \lambda(r-\Id): \pushforward{(\Id, r)}\mu\in\setOptimalPlans{\mu}{\pushforward{r}\mu},\lambda>0
    \}}^{\Lp{2}{\reals^d,\reals^d;\mu}}.
\end{equation*}
We refer to \cite[Theorem 8.5.1]{Ambrosio2008a} for more details.
Third, an alternative definition of sub- and super-differentiability involves taking the \emph{infimum} instead of the \emph{supremum} in~\eqref{eq:sub and superdifferential wasserstein}~\cite[Definition 10.3.1]{Ambrosio2008a}; yet, this definition is equivalent~\cite{Gangbo2019}.
%Fourth, \cref{def:wasserstein differentiable} does \emph{not} establish uniqueness of Wasserstein gradients. Indeed, the set $\subdifferential{\cost}(\mu)\cap\superdifferential{\cost}(\mu)$ generally contains multiple elements. Nonetheless, Wasserstein gradients are, if they exist, unique in $\Lp{2}{\reals^d,\reals^d;\mu}$ modulo the equivalence relation induced by $\tangentSpace{\mu}{\Pp{2}{\reals^d}}^\perp$. That is, Wasserstein gradients are unique in $\tangentSpace{\mu}{\Pp{2}{\reals^d}}$; see~\cref{prop:gradients act on tangent vectors} below. Thus, it is safe to talk about \emph{the} gradient (if it exists).
%Fourth, it is straightforward to show linearity; i.e., for $\alpha,\beta\in\reals$, $\gradient{\mu}(\alpha\cost_1+\beta\cost_2)(\mu)=\alpha\gradient{\mu}\cost_1(\mu)+\beta\gradient{\mu}\cost_2(\mu)$, whenever $\cost_1$ and $\cost_2$ are Wasserstein differentiable at $\mu\in\Pp{2}{\reals^d}$.
Fourth, we defined (sub- and super-)gradients for functions whose domain is $\Pp{2}{\reals^d}$. Yet, our definitions readily extend to any function whose domain is an open or a dense subset of $\Pp{2}{\reals^d}$ (e.g.,  $\Ppabs{2}{\reals^d}$). 
\end{remark}

\cref{def:wasserstein differentiable} does \emph{not} establish uniqueness of Wasserstein gradients. Indeed, the set $\subdifferential{\cost}(\mu)\cap\superdifferential{\cost}(\mu)$ generally contains multiple elements. Nonetheless, Wasserstein gradients are, if they exist, unique in $\Lp{2}{\reals^d,\reals^d;\mu}$ modulo the equivalence relation induced by $\tangentSpace{\mu}{\Pp{2}{\reals^d}}^\perp$. That is, Wasserstein gradients are unique in $\tangentSpace{\mu}{\Pp{2}{\reals^d}}$. In particular, since $\tangentSpace{\mu}\Pp{2}{\reals^d}$ is a \emph{closed} linear subspace of the Hilbert space $\Lp{2}{\reals^d,\reals^d;\mu}$, the Hilbert decomposition theorem~\cite[Theorem 4.11]{Rudin1987} asserts that any element $\xi\in\Lp{2}{\reals^d,\reals^d;\mu}$ can be \emph{uniquely} decomposed as $\xi=\xi^{\parallel} + \xi^{\perp}$, with $\xi^{\parallel}\in\tangentSpace{\mu}\Pp{2}{\reals^d}$ and $\xi^{\perp}\in\tangentSpace{\mu}\Pp{2}{\reals^d}^{\perp}$. The next proposition, which bases on the more general Proposition 8.5.4 in~\cite{Ambrosio2008a}, shows that 
\begin{equation*}
    \int_{\reals^d\times\reals^d} \innerProduct{\xi^\perp(x)}{y-x}\d\gamma(x,y)=0
\end{equation*}
for any $\gamma\in\setOptimalPlans{\mu}{\nu}$ and any $\nu\in\Pp{2}{\reals^d}$, and thus
\begin{equation*}
    \int_{\reals^d\times\reals^d} \innerProduct{\xi(x)}{y-x}\d\gamma(x,y)
    =
    \int_{\reals^d\times\reals^d} \innerProduct{\xi^\parallel(x)}{y-x}\d\gamma(x,y).
\end{equation*}
This, together with \cref{def:wasserstein subdifferentiable,def:wasserstein differentiable}, suggests that we can \emph{always} impose that Wasserstein sub- and super-differential belong to the tangent space $\tangentSpace{\mu}\Pp{2}{\reals^d}$, and thus Wasserstein (sub- and super-)gradients, whenever they exist, are unique and also belong to $\tangentSpace{\mu}\Pp{2}{\reals^d}$. Formally, we have the following proposition.

\begin{proposition}[Wasserstein gradients only act on tangent vectors]\label{prop:gradients act on tangent vectors}
Let $\mu,\nu\in\Pp{2}{\reals^d}$, $\gamma\in\setOptimalPlans{\mu}{\nu}$, and $\xi\in\tangentSpace{\mu}\Pp{2}{\reals^d}^{\perp}$. Then, 
\begin{equation}\label{eq:gradients act on tangent vectors}
    \int_{\reals^d\times\reals^d}\innerProduct{\xi(x)}{y-x}\d\gamma(x,y)=0.
\end{equation}
In particular, we can without loss of generality impose that sub- and super-differential live in $\tangentSpace{\mu}\Pp{2}{\reals^d}$. This way, also the Wasserstein gradient of $\cost$ at $\mu\in\Pp{2}{\reals^d}$ is, if it exists, the unique element  $\gradient{\mu}\cost(\mu)\in\tangentSpace{\mu}{\Pp{2}{\reals^d}}\cap\subdifferential{\cost}(\mu)\cap\superdifferential{\cost}(\mu)$.
\end{proposition}

\ifbool{showproofs}{\booltrue{showproofgradientsactontangentvectors}}{
\ifbool{arxiv}{\boolfalse{showproofgradientsactontangentvectors}}{}} % do not show in arxiv
\ifbool{showproofgradientsactontangentvectors}{\proofgradientsactontangentvectors{booltrue}}{}

More specifically,~\eqref{eq:gradients act on tangent vectors} follows directly from~\cite[Proposition 8.5.4]{Ambrosio2008a}. Uniqueness of the Wasserstein gradient is then a consequence of the definition of Wasserstein gradients and the tangent space.
Henceforth, we will therefore regard the Wasserstein gradient of $\cost$ at $\mu$, if it exists, as the unique element $\gradient{\mu}\cost(\mu)\in\tangentSpace{\mu}{\Pp{2}{\reals^d}}\cap\subdifferential{\cost}(\mu)\cap\superdifferential{\cost}(\mu)$. While apparently of minor importance, imposing that Wasserstein (sub- and super-)gradients live in the tangent space $\tangentSpace{\mu}\Pp{2}{\reals^d}$ (instead of in the more general $\Lp{2}{}$ space) will play a crucial role in later sections.

Next, we show that Wasserstein gradients provide a ``linear approximation'' even if perturbations are not induced by optimal transport plans.
Our proof strategy is inspired by similar results in the setting of absolutely continuous probability measures~\cite[Proposition 4.2]{Ambrosio2007} and probability measures with compact support~\cite[Proposition 3.6]{bonnet2021necessary}.

\begin{proposition}[differentials are ``strong'']\label{prop:strong differentiability}
Let $\mu, \nu\in\Pp{2}{\reals^d}$, $\gamma\in\setPlans{\mu}{\nu}$ (not necessarily optimal), and let $\cost:\Pp{2}{\reals^d}\to\realsBar$ be Wasserstein subdifferentiable at $\mu$ with Wasserstein subgradient $\xi\in\subdifferential{\cost}(\mu)\cap\tangentSpace{\mu}\Pp{2}{\reals^d}$. Then, 
\begin{equation*}
    \cost(\nu)-\cost(\mu)
    \geq 
    \int_{\reals^d\times\reals^d}\innerProduct{\xi(x)}{y-x}\d\gamma(x,y)
    +
    \onotation[adapt]{\sqrt{\int_{\reals^d\times\reals^d}\norm{x-y}^2\d\gamma(x,y)}}.
\end{equation*}
In particular, if $\cost$ is Wasserstein differentiable at $\mu$ with Wasserstein gradient $\gradient{\mu}\cost(\mu)\in\tangentSpace{\mu}\Pp{2}{\reals^d}$, then 
\begin{equation*}
    \cost(\nu)-\cost(\mu)
    =
    \int_{\reals^d\times\reals^d}\innerProduct{\gradient{\mu}\cost(\mu)(x)}{y-x}\d\gamma(x,y)
    +
    \onotation[adapt]{\sqrt{\int_{\reals^d\times\reals^d}\norm{x-y}^2\d\gamma(x,y)}}.
\end{equation*}
\end{proposition}

\ifbool{showproofs}{\booltrue{showproofgradientsarestrong}}{
\ifbool{arxiv}{\boolfalse{showproofgradientsarestrong}}{}} % do not show in arxiv
\ifbool{showproofgradientsarestrong}{\proofgradientsarestrong{booltrue}}{}

We now adapt four well-known facts on (sub)differentiable functions to the Wasserstein space. First, Wasserstein differentiability implies continuity (in the topology induced by the Wasserstein distance). 
\begin{proposition}[differentiability implies continuity]\label{prop:differentiable continuous}
Let $\cost:\Pp{2}{\reals^d}\to\reals$ be Wasserstein differentiable at $\mu\in\Pp{2}{\reals^d}$. Then, $\cost$ is continuous w.r.t. weak convergence in $\Pp{2}{\reals^d}$ at $\mu$.
\end{proposition}
\ifbool{showproofs}{\proofdifferentiablecontinuous{showproofs}}{}
As in the Euclidean case, we can further characterize Wasserstein subgradients of geodesically convex functions (cf.~\cite[Theorem 10.3.6]{Ambrosio2008a}).
\begin{proposition}[gradients of geodesically convex functions]\label{prop:gradients geodesically convex}
Let $\cost$ be $\alpha$-geodesically convex with $\alpha\in\reals$. Suppose that $\cost$ is Wasserstein subdifferentiable at $\mu\in\Pp{2}{\reals^d}$ and let $\xi\in\subdifferential\cost(\mu)$.
Then, for all $\nu\in\Pp{2}{\reals^d}$
\begin{equation}\label{eq:prop gradient convex}
    \cost(\nu)-\cost(\mu)
    \geq
    \sup_{\gamma\in\setOptimalPlans{\mu}{\nu}}\int_{\reals^d\times\reals^d}\innerProduct{\xi(x)}{y-x}\d\gamma(x,y)+\frac{\alpha}{2}\wassersteinDistance{2}{\mu}{\nu}^2.
\end{equation}
Moreover, Wasserstein subgradients are ``monotone'': If $\cost$ is subdifferentiable at $\mu,\nu\in\Pp{2}{\reals^d}$ with $\xi\in\subdifferential\cost(\mu)$ and $\zeta\in\subdifferential\cost(\nu)$, then 
\begin{equation}\label{eq:prop gradient monotone}
    \int_{\reals^d\times\reals^d}\innerProduct{\zeta(y)-\xi(x)}{y-x}\d\gamma(x,y)
    \geq
    \alpha\wassersteinDistance{2}{\mu}{\nu}^2
\end{equation}
for all $\gamma\in\setOptimalPlans{\mu}{\nu}$.
\end{proposition}

\ifbool{showproofs}{\booltrue{showproofgradientsgeodesicallyconvex}}{
\ifbool{arxiv}{\boolfalse{showproofgradientsgeodesicallyconvex}}{}} % do not show in arxiv
\ifbool{showproofgradientsgeodesicallyconvex}{\proofgradientsgeodesicallyconvex{booltrue}}{}

We now study the differentiability properties of the sum and multiplication of functionals:
\begin{proposition}[sum and multiplication rule]\label{prop:sum and multiplication rule}
Let $\cost_1:\Pp{2}{\reals^d}\to\realsBar$ and $\cost_2:\Pp{2}{\reals^d}\to\realsBar$ be Wasserstein subdifferentiable at $\mu\in\Pp{2}{\reals^d}$.
Then, 
\begin{enumerate}
    \item $\subdifferential{(\cost_1+\cost_2)}(\mu)\supset\subdifferential{J_1}(\mu)+\subdifferential{J_2}(\mu)$
    \item $\subdifferential{(\alpha\cost_1)}(\mu)=\alpha\subdifferential{\cost_1}(\mu)$ for all $\alpha\geq 0$.
\end{enumerate}
In particular, if $J_1$ is Wasserstein differentiable at $\mu$, then
\begin{enumerate}
    \item $\subdifferential{(\cost_1+\cost_2)}(\mu)=\{\gradient{\mu}\cost_1(\mu)\}+\subdifferential{J_2}(\mu)$
    \item $\gradient{\mu}{(\alpha\cost_1)}(\mu)=\alpha\gradient{\mu}{\cost_1}(\mu)$ for all $\alpha\in\reals$.
\end{enumerate}
\end{proposition}

\ifbool{showproofs}{\booltrue{showproofsummultiplicationrule}}{
\ifbool{arxiv}{\boolfalse{showproofsummultiplicationrule}}{}} % do not show in arxiv
\ifbool{showproofsummultiplicationrule}{\proofsummultiplicationrule{booltrue}}{}

Finally, we can prove the chain rule.
\begin{proposition}[chain rule]\label{prop:chain rule}
Let $g:\reals\to\reals$ be differentiable, $\cost:\Pp{2}{\reals^d}\to\reals$ be Wasserstein differentiable, and let $\mu\in\Pp{2}{\reals^d}$. Then, $g\circ\cost:\Pp{2}{\reals^d}\to\reals$ is Wasserstein differentiable at $\mu$, and 
\begin{equation*}
    \gradient{\mu}(g\circ J)(\mu)
    =
    g'(\cost(\mu))\gradient{\mu}\cost(\mu).
\end{equation*}
\end{proposition}

\ifbool{showproofs}{\booltrue{showproofchainrule}}{
\ifbool{arxiv}{\boolfalse{showproofchainrule}}{}} % do not show in arxiv
\ifbool{showproofchainrule}{\proofchainrule{booltrue}}{}

\subsection{Examples of Smooth Functionals}\label{subsec:examples of smooth functionals}
After having discussed the differential structure of Wasserstein space, we recall and in some cases further study the topological and differential properties of five classes of important functionals:
\ifbool{compact}{
optimal transport discrepancies, expected values-type functionals, interaction-type functionals, internal-energy-type functionals, and divergences.
}{
\begin{itemize}
    \item optimal transport discrepancies from a reference measure $\bar\mu\in\Pp{2}{\reals^d}$: 
    \begin{equation*}
        \mu\mapsto
        \optimalTransportDiscrepancy{c}{\mu}{\bar\mu}\coloneqq 
        \min_{\gamma\in\setPlans{\mu}{\bar\mu}}\int_{\reals^d\times\reals^d}c(x,y)\d\gamma(x,y),
    \end{equation*}
    for some continuously differentiable transport cost $c:\reals^d\times\reals^d\to\nonnegativeReals$, which also include the (squared) Wasserstein distance. 
    
    \item expected value-type functionals:
    \begin{equation*}
        \mu\mapsto \expectedValue{\mu}{V(x)}\coloneqq\int_{\reals^d}V(x)\d\mu(x)
    \end{equation*}
    for some twice continuously differentiable $V:\reals^d\to\reals$;
    
    \item interaction-type functionals:
    \begin{equation*}
        \mu\mapsto\frac{1}{2}\int_{\reals^d\times\reals^d}U(x-y)\d(\mu\productMeasure\mu)(x,y),
    \end{equation*}
    where $\mu\productMeasure\mu$ denotes the product measure, for some continuously differentiable $U:\reals^d\to\nonnegativeReals$, related, among others, to the variance (see~\cref{cor:variance} below);
    
    \item internal-energy-type functionals: 
    \begin{equation*}
        \mu\mapsto
        \begin{cases}
            \displaystyle\int_{\reals^d}F(\rho(x))\d x & \text{if }\mu=\rho\lebesgueMeasure{d}\in\Ppabs{2}{\reals^d},
            \\
            +\infty & \text{else}
        \end{cases}
    \end{equation*}
    for some continuously differentiable $F:\nonnegativeReals\to\nonnegativeReals$, related, among others, to the entropy;

    \item the divergences relative to a reference measure $\bar\mu\in\Pp{2}{\reals^d}$:
    \begin{equation*}
        \mu\mapsto
        \begin{cases}
            \displaystyle\int_{\reals^d}F\left(\radonNikodyn{\mu}{\bar\mu}\right)\,\d\bar\mu(x) & \text{if } \mu\ll\bar\mu,
            \\
            +\infty & \text{else},
        \end{cases}
    \end{equation*}
    where $\radonNikodyn{\mu}{\bar\mu}$ is the Radon-Nykodin derivative~\cite{Rudin1987}; an example is the \gls{acr:kl} divergence. 
\end{itemize}}

We start by recalling from~\cite{Villani2009a,Santambrogio2015,Ambrosio2008a} and extending some basic properties of optimal transport discrepancies.

\begin{proposition}[optimal transport discrepancy]\label{prop:optimal transport discrepancy}
Let $\bar\mu\in\Pp{2}{\reals^d}$, $c:\reals^d\times\reals^d\to\nonnegativeReals$, and consider the functional 
\begin{equation}\label{eq:optimal transport discrepancy:definition}
    \cost(\mu)=
    \optimalTransportDiscrepancy{c}{\mu}{\bar\mu}
    \coloneqq
    \min_{\gamma\in\setPlans{\mu}{\bar\mu}}\int_{\reals^d\times\reals^d}c(x,y)\d\gamma(x,y).
\end{equation}
Assume that there exists $C>0$ so that $c(y,y)\leq C$ for all $y\in\support{\bar\mu}$, $c$ is twice continuously differentiable in the first argument, and there exists $M>0$ so that $\norm{\gradient{x}^2c(x,y)}\leq M$ for all $x\in\reals^d, y\in\support{\bar\mu}$, and $\norm{\gradient{x}c(x,y)}^2\leq M(1 + \norm{x}^2+\norm{y}^2)$.
Then, 
\begin{enumerate}
    \item $\cost$ is proper and $\cost(\mu)\geq 0$ for all $\mu\in\Pp{2}{\reals^d}$; 
    \item $\cost$ is lower semi-continuous w.r.t. weak convergence in $\Pp{2}{\reals^d}$; 
    \item $\cost$ is lower semi-continuous w.r.t. narrow convergence; 
    \item if $-c(\cdot,y)$ is $\alpha$-convex (with $\alpha\in\reals$) for all $y\in\support{\bar\mu}$, then $-\cost$ is $\alpha$-convex along any interpolating curve;
    \item if $c(\cdot,y)$ is $\alpha$-convex (with $\alpha\in\reals$) for all $y\in\support{\bar\mu}$, then $\cost$ is $\alpha$-convex along the curve
    \begin{equation*}
        \mu_t=\pushforward{((1-t)\proj_2+t\proj_3)}\gamma,
    \end{equation*}
    where $\gamma\in\Pp{2}{\reals^d\times\reals^d\times\reals^d}$, $\pushforward{(\proj_{12})}\gamma\in\setOptimalPlans{\bar\mu}{\mu}$ and $\pushforward{(\proj_{13})}\gamma\in\setOptimalPlans{\bar\mu}{\nu}$ (where optimality is intended w.r.t. the transport cost $c$), and $\mu,\nu\in\Pp{2}{\reals^d}$;
    \item if $\mu\in\effectiveDomain{\cost}$
    and $\setOptimalPlans{\mu}{\bar\mu}$ contains a unique optimal transport plan induced by  an optimal transport map $\optMap{\mu}{\bar\mu}\in\Lp{2}{\reals^d,\reals^d;\mu}$ (w.r.t. the cost $c$),  $\cost$ is Wasserstein differentiable at $\mu$, and its Wasserstein gradient is
    \begin{equation}\label{eq:ot discrepancy:wasserstein gradient}
        \gradient{\mu}\cost(\mu) = \gradient{x}c(\cdot,\optMap{\mu}{\bar\mu}(\cdot));
    \end{equation}
    \item in particular, if $y\mapsto\gradient{x}c(x,y)$ is injective for $\mu$-almost every $x\in\reals^d$ (i.e., $c$ satisfies the so-called twist condition) and $\mu\in\Ppabs{2}{\reals^d}$ is absolutely continuous, then $\cost$ is Wasserstein differentiable at $\mu$ and its Wasserstein gradient is~\eqref{eq:ot discrepancy:wasserstein gradient}.
\end{enumerate}
\end{proposition}

\ifbool{showproofs}{\booltrue{showproofoptimaltransportdiscrepancy}}{
\ifbool{arxiv}{\boolfalse{showproofoptimaltransportdiscrepancy}}{}} % do not show in arxiv
\ifbool{showproofoptimaltransportdiscrepancy}{\proofoptimaltransportdiscrepancy{booltrue}}{}

\begin{remark}
Three remarks are in order. 
First, Statement (i) follows immediately from basic properties of $J$; (ii) and (iii) are a consequence of Kantorovich duality (e.g., see~\cite[Proposition 7.4]{Santambrogio2015}); (iv), (v), and (vi) are proven in~\cite[§9.3 and §10.4]{Ambrosio2008a} for the special case of the Wasserstein distance (i.e., $c(x,y)=\norm{x-y}^2$, see~\cref{cor:wasserstein distance} below) and extended here for general transportation costs $c$; (vii) follows from (vi) and~\cite[Theorem 10.28 and Example 10.35]{Villani2009a}.
Second, Statement (v) does \emph{not} imply that $\cost$ is convex along generalized geodesics, but only along the generalized geodesics whose ``base'' is precisely $\bar\mu$.
Third, if $c(x,y)=\transpose{(x-y)}A(x-y)$ with $A=\transpose{A}\in\reals^{d\times d}$ positive definite or $c(x,y)=\ell(\norm{x-y})$ for $\ell:\nonnegativeReals\to\nonnegativeReals$ strictly convex which grows faster than linearly but so that $c$ has bounded Hessian, then $y\mapsto\gradient{x}c(x,y)$ is injective at all $x\in\reals^d$ and $\mu$ being absolutely continuous is a sufficient condition for existence of a unique optimal transport plan induced by an optimal transport map; see~\cite{Gangbo1996,figalli2007existence} and~\cite[§10]{Villani2009a} for details and generalizations.
\end{remark}

As a corollary, we recover the convexity and differentiability properties of the squared Wasserstein distance presented in~\cite[§7 and §10]{Ambrosio2008a}.
\begin{corollary}[squared Wasserstein distance]\label{cor:wasserstein distance}
Let $\bar\mu\in\Pp{2}{\reals^d}$ and consider the functional 
\begin{equation*}
    \cost(\mu)=\frac{1}{2}\wassersteinDistance{2}{\mu}{\bar\mu}^2.
\end{equation*}
Then, 
\begin{enumerate}
    \item $\effectiveDomain{\cost}=\Pp{2}{\reals^d}$, $\cost$ is proper, and $\cost(\mu)\geq 0$ for all $\mu\in\Pp{2}{\reals^d}$; 
    \item $\mu\mapsto\wassersteinDistance{2}{\mu}{\bar\mu}$ and $\cost$ are continuous w.r.t. weak convergence in $\Pp{2}{\reals^d}$; 
    \item $\mu\mapsto\wassersteinDistance{2}{\mu}{\bar\mu}$ and $\cost$ are lower semi-continuous w.r.t. to narrow convergence; 
    \item $-\cost$ is $(-1)$-convex (and thus $\cost$ is $(-1)$-concave) along any interpolating curve; 
    \item $\cost$ is $1$-convex along the curve 
    \begin{equation*}
        \mu_t=\pushforward{((1-t)\proj_2+t\proj_3)}\gamma,
    \end{equation*}
    where $\gamma\in\Pp{2}{\reals^d\times\reals^d\times\reals^d}$, $\pushforward{(\proj_{12})}\gamma\in\setOptimalPlans{\bar\mu}{\mu}$ and $\pushforward{(\proj_{13})}\gamma\in\setOptimalPlans{\bar\mu}{\nu}$, and $\mu,\nu\in\Pp{2}{\reals^d}$;
    \item if $\mu\in\Ppabs{2}{\reals^d}$ is absolutely continuous, $\cost$ is Wasserstein differentiable at $\mu$ and its Wasserstein gradient is
    \begin{equation*}
    \gradient{\mu}\cost(\mu) = \Id-\optMap{\mu}{\bar\mu},
    \end{equation*}
    where $\Id:\reals^d\to\reals^d$ is the identity map and $\optMap{\mu}{\bar\mu}:\reals^d\to\reals^d$ is the (unique) optimal transport map from $\mu$ to $\bar\mu$.
\end{enumerate}
\end{corollary}

\ifbool{showproofs}{\booltrue{showproofwassersteindistance}}{
\ifbool{arxiv}{\boolfalse{showproofwassersteindistance}}{}} % do not show in arxiv
\ifbool{showproofwassersteindistance}{\proofwassersteindistance{booltrue}}{}

Unlike in Euclidean (or more generally Hilbertian) settings, the distance function is \emph{not} (geodesically) convex: It is in fact semiconcave. This makes the Wasserstein space a \emph{positively curved metric space}; see~\cite[§7.3]{Ambrosio2008a}.
This fact, among others, will complicate sufficient conditions for optimality, oftentimes intimately related to convexity. We later circumvent this technical difficulty in~\cref{thm:first order sufficient conditions constrained wasserstein}.

Next, we recall from~\cite[§9.3 and §10.4]{Ambrosio2008a} some well-known results for expected values:

\begin{proposition}[expected value~\protect{\cite[§9.3 and §10.4]{Ambrosio2008a}}]\label{prop:expectation}
Let $V\in\C{2}{\reals^d}$ with uniformly bounded Hessian (i.e., there exists $M>0$ so that $\norm{\gradient{x}^2V(x)}\leq M$ for all $x\in\reals^d$), and consider the functional
\begin{equation}\label{eq:expected value:definition}
    \cost(\mu)=\expectedValue{\mu}{V(x)}=\int_{\reals^d}V(x)\d\mu(x).
\end{equation}
Then, 
\begin{enumerate}
    \item $\effectiveDomain{\cost}=\Pp{2}{\reals^d}$, $\cost$ is proper, and $\cost(\mu)>-\infty$ for all $\mu\in\Pp{2}{\reals^d}$;
    
    \item $\cost$ is continuous w.r.t. weak convergence in $\Pp{2}{\reals^d}$;
    
    \item if $V(x)^-\leq C(1+\norm{x}^p)$ for $C\in\reals$ and $p<2$, then $\cost$ is lower semi-continuous w.r.t. narrow convergence;
    
    \item $J$ is $\alpha$-convex along any interpolating curve if and only if $V$ is $\alpha$-convex;
    
    \item $\cost$ is Wasserstein differentiable and its Wasserstein gradient at $\mu\in\Pp{2}{\reals^d}$ is 
    \begin{equation*}
        \gradient{\mu}\cost(\mu)=\gradient{x}V.
    \end{equation*}
\end{enumerate}
\end{proposition}

\ifbool{showproofs}{\booltrue{showproofexpectation}}{
\ifbool{arxiv}{\boolfalse{showproofexpectation}}{}} % do not show in arxiv
\ifbool{showproofexpectation}{\proofexpectation{booltrue}}{}

Following~\cite[§9.3 and §10.4]{Ambrosio2008a}, we now consider interaction-type functionals:

\begin{proposition}[interaction-type functionals~\protect{\cite[§9.3 and §10.4]{Ambrosio2008a}}]\label{prop:interaction}
Let  $U\in\C{2}{\reals^d}$  with uniformly bounded Hessian (i.e., there exists $M>0$ so that $\norm{\gradient{z}^2U(z)}\leq M$ for all $z\in\reals^d$), and consider the functional
\begin{equation}\label{eq:interaction:definition}
    \cost(\mu)=\frac{1}{2}\int_{\reals^d\times\reals^d}U(x-y)\d(\mu\productMeasure\mu)(x,y).
\end{equation}
Then,
\begin{enumerate}
    \item $\effectiveDomain{\cost}=\Pp{2}{\reals^d}$, $\cost$ is proper, and  
    $\cost(\mu)>-\infty$ for all $\mu\in\Pp{2}{\reals^d}$; 
    
    \item $\cost$ is continuous w.r.t. weak convergence in $\Pp{2}{\reals^d}$; 
    
    \item if $U^-(z)\leq C(1+\norm{z}^p)$ for $C\in\reals$ and $p<2$, then $\cost$ is lower semi-continuous w.r.t. narrow convergence; 
    
    \item if $U$ is convex, then $\cost$ is convex along any interpolating curve; 

    \item $\cost$ is Wasserstein differentiable and its Wasserstein gradient at $\mu\in\Pp{2}{\reals^d}$ is 
\begin{equation*}
    \gradient{\mu}\cost(\mu)=\gradient{}U\conv\mu,
\end{equation*}
provided that $\gradient{}U\conv\mu\in\Lp{2}{\reals^d,\reals^d; \mu}$, where the convolution is defined as 
\begin{equation*}
    \gradient{\mu}\cost(\mu)(x)=(\gradient{}U\conv\mu)(x)\coloneqq\int_{\reals^d}\nabla U(x-y)\d\mu(y).
\end{equation*}
\end{enumerate}
\end{proposition}

\ifbool{showproofs}{\booltrue{showproofinteraction}}{
\ifbool{arxiv}{\boolfalse{showproofinteraction}}{}} % do not show in arxiv
\ifbool{showproofinteraction}{\proofinteraction{booltrue}}{}

\begin{remark}
Two remarks are in order. 
First, Statements (i) and (ii) follow immediately from the definition of $J$ and the assumptions on $U$; (iii) is a consequence of~\cite[Theorem 3]{Yue2021}; (iv) is proven in~\cite[Proposition 9.3.5]{Ambrosio2008a}; (v) is proven in~\cite[Theorem 10.4.11]{Ambrosio2008a} under convexity of $U$ and a so-called ``doubling'' assumption, while our proof instead bases on smoothness of $U$ and boundedness of the Hessian.
Second, if $U$ is $\alpha$-geodesically convex, then $\cost$ is not necessarily $\alpha$-geodesically convex; e.g., see~\cref{lemma:properties dro variance} below. Nonetheless, if $\alpha\geq 0$, then $\cost$ is ($0$-)geodesically convex.
\end{remark}

We now instantiate \cref{prop:interaction} to study the variance.

\begin{corollary}[variance]\label{cor:variance}
Let $a\in\reals^d$ and $\cost(\mu)=\variance{\mu}{\innerProduct{a}{x}}$. Then, $\cost$ is continuous w.r.t. to weak convergence in $\Pp{2}{\reals^d}$, convex along any interpolating curve, and its Wasserstein gradient at $\mu\in\Pp{2}{\reals^d}$ is
\begin{equation*}
    \gradient{\mu}\cost(\mu)(x)=2a\innerProduct{a}{x-\expectedValue{\mu}{x}}.
\end{equation*}
\end{corollary}

\ifbool{showproofs}{\booltrue{showproofvariance}}{
\ifbool{arxiv}{\boolfalse{showproofvariance}}{}} % do not show in arxiv
\ifbool{showproofvariance}{\proofvariance{showproofs}}{}

\begin{remark}
If $d=1$ and $a=1$, the Wasserstein gradient reduces to the natural expression $\gradient{\mu}\variance{\mu}{x}=2(\Id-\expectedValue{\mu}{x})$. The gradient evaluates to zero when all probability mass is concentrated at the expected value and the variance achieves its global minimum;  see~\cref{sec:optimality}. 
Moreover, \cref{cor:variance} also follows from the chain rule (\cref{prop:chain rule}): Since $\variance{\mu}{\innerProduct{a}{x}}=\expectedValue{\mu}{\innerProduct{a}{x}^2}-\expectedValue{\mu}{\innerProduct{a}{x}}^2$, we conclude\ifbool{compact}{
$\gradient{\mu}\cost(\mu)=2a\innerProduct{a}{\Id} - 2a\expectedValue{\mu}{\innerProduct{a}{x}}=2a\innerProduct{a}{\Id-\expectedValue{\mu}{x}}$.}{
\begin{equation*}
    \gradient{\mu}\cost(\mu)=2a\innerProduct{a}{\Id} - 2a\expectedValue{\mu}{\innerProduct{a}{x}}=2a\innerProduct{a}{\Id-\expectedValue{\mu}{x}}. 
    \qedhere 
\end{equation*}}
\end{remark}

%\begin{remark}
%\nlmargin{The \gls{acr:mmd}\ldots~\cite{Arbel2019}}{todo}, \cite{Kent2021}.
%\end{remark}

We can now recall from~\cite{Ambrosio2008a,erbar2010heat} the properties of internal energy-type functionals: 

\begin{proposition}[internal energy-type functionals~\cite{Ambrosio2008a,erbar2010heat}]\label{prop:internal energy}
Let $F:\nonnegativeReals\to\reals$ be convex, differentiable, with superlinear growth at infinity, and satisfy $F(0)=0$. 
Consider the functional
\begin{equation*}
    \cost(\mu)=
    \begin{cases}
    \displaystyle\int_{\reals^d}F(\rho(x))\d x & \text{if } \mu=\rho\lebesgueMeasure{d}\in\Ppabs{2}{\reals^d},\\
    +\infty & \text{else}. 
    \end{cases}
\end{equation*}
Then, 
\begin{enumerate}
    \item $\effectiveDomain{\cost}\subset\Ppabs{2}{\reals^d}$ and $\cost$ is proper;
    
    \item $\cost$ is lower semi-continuous w.r.t. weak convergence in $\Pp{2}{\reals^d}$; 
    
    \item $\cost$ is lower semi-continuous w.r.t. narrow convergence.
\end{enumerate}
Moreover, if $s\mapsto s^d F(s^{-d})$ is convex and non-increasing on $(0,+\infty)$, then
\begin{enumerate}[resume]
    \item $\cost$ is convex along generalized geodesics; 
\end{enumerate}
If additionally there exists $C>0$ s.t. $F(z+w)\leq C(1+F(z)+F(w))$ for all $z,w\in\reals^d$, then
\begin{enumerate}[resume]
    \item $\cost$ is Wasserstein subdifferentiable at $\mu=\rho\lebesgueMeasure{d}\in\Ppabs{2}{\reals^d}$ and the Wasserstein subdifferential of $\cost$ at $\mu$ is given by
    \begin{equation}\label{eq:internal energy:wasserstein gradient}
        \subdifferential\cost(\mu)
        =
        \left\{\frac{\gradient{}(\rho F'(\rho)-F(\rho))}{\rho}\right\},
    \end{equation}
    provided that all terms all well defined (i.e., $\rho F'(\rho)+\rho\in\Wrloc{1}{\reals^d; \mu}$ and  $\frac{\gradient{}(\rho F'(\rho)-F(\rho))}{\rho}\in\Lp{2}{\reals^d,\reals^d;\mu}$).
\end{enumerate}
\end{proposition}

\ifbool{showproofs}{\booltrue{showproofinternalenergy}}{
\ifbool{arxiv}{\boolfalse{showproofinternalenergy}}{}} % do not show in arxiv
\ifbool{showproofinternalenergy}{\proofinternalenergy{showproofinternalenergy}}{}

\begin{remark}
    Two remarks are in order. 
    First,~\cite{Ambrosio2008a} proves that $\frac{\gradient{}(\rho F'(\rho)-F(\rho))}{\rho}$ is the minimal selection in the subdifferential of $\cost$ at $\mu$~\cite[Theorem 10.4.6]{Ambrosio2008a}, but it does not provide a full characterization of the subdifferential of $\cost$ which was established in~\cite[Proposition 4.3]{erbar2010heat} for the special case of the (relative) entropy in the more general case of Riemannian manifolds (by showing uniqueness of the subgradient). 
    Second, internal energy-type functionals (as well as the divergence presented below) are not Wasserstein differentiable since they attain $+\infty$ at all measures which are not absolutely continuous, which, in turn, implies that the superdifferential is empty.
\end{remark}

As an example, we consider the entropy: 

\begin{example}[entropy]
\label{ex:entropy}
Consider the entropy
\begin{equation*}
\cost(\mu)=
\begin{cases}
\displaystyle\int_{\reals^d}\rho\log(\rho)\d x & \text{if } \mu=\rho\lebesgueMeasure{d}\in\Ppabs{2}{\reals^d}, \\
+\infty & \text{else}.
\end{cases}
\end{equation*}
That is, $F(z)=z\log(z)$ (with $F(0)=0$). Clearly, $F$ is convex, differentiable, and has superlinear growth.
Moreover, $s^ds^{-d}\log(s^{-d})=-d\log(s)$ is convex and non-increasing, and 
\ifbool{compact}{$(z+w)\log(z+w)\leq 2(1+z\log(z)+w\log(w))$}{
\begin{equation*}
\begin{aligned}
    (z+w)\log(z+w)
    &\leq
    2(1+z\log(z)+w\log(w)).
\end{aligned}
\end{equation*}}
So, if $\mu=\rho\lebesgueMeasure{d}\in\Ppabs{2}{\reals^d}$, the Wasserstein subdifferential of $\cost$ at $\mu$ is
\begin{equation*}
    \subdifferential\cost(\mu)
    =
    \left\{\frac{\gradient{}\left(\rho(\log(\rho)+1)-\rho\log(\rho)\right)}{\rho}\right\}
    =
    \left\{\frac{\gradient{}\rho}{\rho}\right\},
\end{equation*}
provided that all terms are well-defined.
\end{example}

To conclude, we recall from~\cite{Ambrosio2008a,erbar2010heat} the properties of divergences: 

\begin{proposition}[divergences~\cite{Ambrosio2008a,erbar2010heat}]\label{prop:divergence}
Let $F:\nonnegativeReals\to\reals$ be convex, differentiable, with superlinear growth at infinity, and satisfy $F(0)=0$. Let $\bar\mu\in\Pp{2}{\reals^d}$.
Consider the functional
\begin{equation*}
    \cost(\mu)=
    \begin{cases}
    \displaystyle\int_{\reals^d}F\left(\radonNikodyn{\mu}{\bar\mu}(x)\right)\d\bar\mu(x) & \text{if } \mu\ll\bar\mu,\\
    +\infty & \text{else},
    \end{cases}
\end{equation*}
where $\radonNikodyn{\mu}{\bar\mu}$ is the Radon-Nikodyn derivative. 
Then, 
\begin{enumerate}
    \item $\effectiveDomain{\cost}\subset\{\mu\in\Pp{2}{\reals^d}: \mu\ll\bar\mu\}$ and $\cost$ is proper;
    \item $\cost$ is lower semi-continuous w.r.t. weak convergence in $\Pp{2}{\reals^d}$; 
    \item $\cost$ is lower semi-continuous w.r.t. narrow convergence.
\end{enumerate}
Moreover, if $s\mapsto s^d F(s^{-d})$ is convex and non-increasing on $(0,+\infty)$ and $\bar\mu$ is log-concave, then
\begin{enumerate}[resume]
    \item $\cost$ is convex along generalized geodesics.
\end{enumerate}
Additionally, if there exists $C>0$ such that $F(z+w)\leq C(1+F(z)+F(w))$ for all $z,w\in\reals^d$ and $\bar\mu=e^{-V}\lebesgueMeasure{d}$ with $V:\reals^d\to\reals$ convex and continuously differentiable, then
\begin{enumerate}[resume]
    \item $\cost$ is Wasserstein subdifferentiable at $\mu=\rho\lebesgueMeasure{d}\in\Ppabs{2}{\reals^d}$ and the Wasserstein subdifferential of $\cost$ at $\mu$ is given by
    \begin{equation*}
        \subdifferential\cost(\mu)
        =
        \left\{
        \frac{\gradient{}(\rho/e^{-V} F'(\rho/e^{-V})-F(\rho/e^{-V}))}{\rho} e^{-V}
        \right\},
    \end{equation*}
    provided that all terms are well defined  (i.e., $\rho/e^{-V} F'(\rho/e^{-V})+\rho/e^{-V}\in\Wrloc{1}{\reals^d; \mu}$ and  $\frac{\gradient{}(\rho/e^{-V} F'(\rho/e^{-V})-F(\rho/e^{-V}))}{\rho} e^{-V}\in\Lp{2}{\reals^d,\reals^d;\mu}$).
\end{enumerate}
\end{proposition}

\ifbool{showproofs}{\booltrue{showproofdivergence}}{
\ifbool{arxiv}{\boolfalse{showproofdivergence}}{}} % do not show in arxiv
\ifbool{showproofdivergence}{\proofdivergence{showproofs}}{}

A prominent example is the \gls{acr:kl} divergence: 

\begin{example}[\gls{acr:kl} divergence]\label{ex:kl divergence}
Let $\bar\mu=e^{-V}\lebesgueMeasure{d}$ with $V$ convex and continuously differentiable. 
Consider the \gls{acr:kl} divergence, defined as
\begin{equation*}
    \kullbackLeibler{\mu}{\bar\mu}
    \coloneqq 
    \begin{cases}
    \displaystyle\int_{\reals^d}\log\left(\frac{\rho(x)}{e^{-V(x)}}\right)\frac{\rho(x)}{e^{-V(x)}}e^{-V(x)}\d x & \text{if } \mu=\rho\lebesgueMeasure{d}\in\Ppabs{2}{\reals^d},
    \\
    +\infty &\text{else}.
    \end{cases}
\end{equation*}
We can proceed as in~\cref{ex:entropy} and use~\cref{prop:divergence} to obtain that 
\begin{equation*}
\begin{aligned}
    \subdifferential\kullbackLeibler{\cdot}{\bar\mu}(\mu)
    =
    \left\{
    \frac{\gradient{}(\rho/e^{-V}+\rho/e^V\log(\rho/e^{-V})-\rho/e^{-V}\log(\rho/e{-V}))}{\rho /e^{-V}}\right\}
    %\\
    %&=
    %\frac{\gradient{}(\rho/e^{-V})}{\rho}e^{-V}
    %\\
    %&=
    =
    \left\{
    \frac{\gradient{}\rho}{\rho} + \gradient{}V
    \right\},
\end{aligned} 
\end{equation*}
provided that all terms are well-defined.
\end{example}

We summarize these functionals, together with their Wasserstein (sub)gradients, in \cref{tab:gradients}.

\begin{table}[ht]
    \small 
    \newcolumntype{C}{>{$} c <{$}}
    %\newcolumntype{C}{>{$\displaystyle} c <{$}}
    \centering
    \begin{small}
    \def\arraystretch{1.7}
    \setlength{\tabcolsep}{\ifbool{compact}{3pt}{8pt}}
    \begin{tabular}{CCc}\toprule
         \text{Functional}
         & \text{Wasserstein (sub)gradient}
         & Notes
         \\ \midrule
         \expectedValue{\mu}{f}
         & \gradient{}f
         & cond. on $f$
         \\
         \min_{\gamma\in\setPlans{\mu}{\bar\mu}}\int_{\reals^d\times\reals^d}c(x,y)\d\gamma(x,y)
         & \gradient{x}c(\cdot,\optMap{\mu}{\bar\mu}(\cdot)) & cond. on $c$, $\optMap{\mu}{\bar\mu}$ exists
         \\
         \frac{1}{2}\int_{\reals^d\times\reals^d}U(x-y)\d(\mu\productMeasure\mu)(x,y)
         & \gradient{}U\conv\mu
         & cond. on $U$
         \\
         \int_{\reals^d}F(\rho(x))\d x
         & \frac{\gradient{}\left(\rho F'(\rho)-F(\rho)\right)}{\rho}
         & $\mu=\rho\lebesgueMeasure{d}$, cond. on $F$
         \\
         \int_{\reals^d}F(\rho(x)/e^{-V(x)})e^{-V(x)}\d x
         & \frac{\gradient{}(\rho/e^{-V} F'(\rho/e^{-V})-F(\rho/e^{-V}))}{\rho/e^{-V}}
         & $\mu=\rho\lebesgueMeasure{d}$, $\bar\mu=e^{-V}\lebesgueMeasure{d}$, cond. on $F, V$
         \\
         \bottomrule 
    \end{tabular}
    \end{small}
    \caption{Wasserstein (sub)gradients of real-valued functions over the Wasserstein space. The abbreviation ``cond.'' stands for ``conditions'' (refer to the appropriate proposition in \cref{subsec:examples of smooth functionals} for the list of conditions).}
    \label{tab:gradients}
\end{table}
\section{Optimality Conditions}\label{sec:optimality}

Consider a lower semi-continuous (w.r.t. weak convergence in $\Pp{2}{\reals^d}$) functional $\cost:\Pp{2}{\reals^d}\to\realsBar$. For ease of notation, we denote (open) Wasserstein balls of radius $r>0$ centered at $\hat\mu\in\Pp{2}{\reals^d}$ by\ifbool{compact}{
$\wassersteinBall{2}{r}{\hat\mu}
    \coloneqq 
    \{\mu\in\Pp{2}{\reals^d}
    :
    \wassersteinDistance{2}{\mu}{\hat\mu}<r\}.$}{
\begin{equation*}
    \wassersteinBall{2}{r}{\hat\mu}
    \coloneqq 
    \{\mu\in\Pp{2}{\reals^d}
    :
    \wassersteinDistance{2}{\mu}{\hat\mu}<r\}.
\end{equation*}}

\subsection{Unconstrained Optimization}

In the setting of unconstrained optimization, minimizers are defined as usual.
\begin{definition}[local, global, and strict minimizers for unconstrained optimization]\label{def:minima unconstrained}
A probability measure $\mu^\ast\in\Pp{2}{\reals^d}$ is a \emph{local minimizer} of $\cost$ if there exists $r>0$ such that $\cost(\mu)\geq \cost(\mu^\ast)$ for all $\mu\in\wassersteinBall{2}{r}{\mu^\ast}$.
If this holds for any $r>0$, then $\mu^\ast$ is a \emph{global minimizer}. If $\cost(\mu)>\cost(\mu^\ast)$ for all $\mu\in\Pp{2}{\reals^d}, \mu\neq \mu^\ast$, then $\mu^\ast$ is the \emph{strict minimizer} of $\cost$.
\end{definition}

Similarly to Euclidean settings, local minimality and derivatives are intimately related.

\begin{theorem}[first-order necessary conditions]\label{thm:first order conditions}
Let $\mu^\ast$ be a local minimizer of $\cost$ and assume $\cost$ is Wasserstein differentiable at $\mu^\ast$.  Then, the Wasserstein gradient of $\cost$ vanishes at $\mu^\ast$: 
\begin{equation*}
    \gradient{\mu}\cost(\mu^\ast)(x)=0\quad \text{for }\mu^\ast\text{-almost every }x\in\reals^d,
\end{equation*}
i.e., $\gradient{\mu}\cost(\mu^\ast)=0$ in $\Lp{2}{\reals^d,\reals^d;\mu^\ast}$.
\end{theorem}

\ifbool{showproofs}{\booltrue{showproofnecessaryconditions}}{}
\ifbool{arxiv}{\booltrue{showproofnecessaryconditions}}{}

\ifbool{showproofnecessaryconditions}{
Surprisingly, the proof of \cref{thm:first order conditions} is not elementary. In particular, we first need a revisited version of the fundamental lemma of calculus of variations:

\begin{lemma}[fundamental lemma of calculus of variations, revisited]\label{lemma:variations}
Let $\mu\in\Pp{2}{\reals^d}$ and $H\in\tangentSpace{\mu}{\Pp{2}{\reals^d}}$. Assume that for all $\psi\in\Ccinf{\reals^d}$
\begin{equation}\label{eq:lemma variations}
    \int_{\reals^d}\innerProduct{H(x)}{\nabla\psi(x)}\d\mu(x)=0.
\end{equation}
Then, $H\equiv 0$ in $\Lp{2}{\reals^d,\reals^d;\mu}$.
\end{lemma}

\prooflemmavariations{booltrue}
\prooffirstorderconditions{boolfalse}
}{}

Under an additional convexity assumption, we get sufficient conditions:
\begin{theorem}[first-order sufficient conditions]\label{thm:first order sufficient conditions}
Assume that $\cost$ is $\alpha$-geodesically convex with $\alpha\geq0$. Suppose there exists $\mu^\ast\in\Pp{2}{\reals^d}$ such that $\cost$ is Wasserstein subdifferentiable at $\mu^\ast$ and a Wasserstein subgradient of $\cost$ vanishes at $\mu^\ast$, i.e., there is $\xi\in\subdifferential{\cost}(\mu^\ast)$ so that $\xi(x)=0$ for $\mu^\ast$-almost every $x\in\reals^d$. Then, $\mu^\ast$ is a global minimizer of $\cost$. If $\alpha>0$, $\mu^\ast$ is the strict minimizer of $\cost$.
\end{theorem}

\ifbool{showproofs}{\booltrue{showprooffirstordersufficientconditions}}{
\ifbool{arxiv}{\boolfalse{showprooffirstordersufficientconditions}}{}} % do not show in arxiv
\ifbool{showprooffirstordersufficientconditions}{\prooffirstordersufficientconditions{showprooffirstordersufficientconditions}}{}

\begin{example}
Consider $\cost(\mu)=\expectedValue{\mu}{\frac{1}{2}x^2}$. Then, by \cref{prop:expectation}, $\gradient{\mu}\cost(\mu^\ast)=\Id$. Thus, by~\cref{thm:first order conditions}, a necessary condition for optimality is $\gradient{\mu}\cost(\mu^\ast)(x)=0$ for $\mu^\ast$-almost every $x\in\reals^d$, which holds if and only if $\mu^\ast=\diracMeasure{0}$. By \cref{prop:expectation}, $\cost$ is $1$-geodesically convex, and so $\mu^\ast$ is the strict minimizer of $\cost$, by~\cref{thm:first order sufficient conditions}.
\end{example}

\subsection{Equality-constrained Optimization}

Consider now the equality-constrained optimization problem
\begin{equation}\label{eq:constrained optimization}
    \inf_{\mu\in\Pp{2}{\reals^d}}
    \left\{
        \cost(\mu): \constraint(\mu)=0
    \right\},
\end{equation}
for some lower semi-continuous (w.r.t. weak convergence in $\Pp{2}{\reals^d}$) functional $\constraint:\Pp{2}{\reals^d}\to\reals$. To start, we adapt~\cref{def:minima unconstrained} to the constrained setting.

\begin{definition}[local, global, and strict minimizers for constrained optimization]\label{def:minima constrained}
A probability measure $\mu^\ast\in\Pp{2}{\reals^d}$ is a \emph{local minimizer} of \eqref{eq:constrained optimization} if there exists $r>0$ such that $\cost(\mu)\geq \cost(\mu^\ast)$ for all $\mu\in\wassersteinBall{2}{r}{\mu^\ast}\cap\{\mu\in\Pp{2}{\reals^d}: \constraint(\mu)=0\}$.
If this holds for any $r>0$, then $\mu^\ast$ is a \emph{global minimizer}.  If $\cost(\mu)>\cost(\mu^\ast)$ for all $\mu\in\Pp{2}{\reals^d}\cap\{\mu\in\Pp{2}{\reals^d}: \constraint(\mu)=0\}, \mu\neq \mu^\ast$, then $\mu^\ast$ is the \emph{strict minimizer} of \eqref{eq:constrained optimization}.
\end{definition}

The goal of this section is to formulate first-order necessary conditions for~\eqref{eq:constrained optimization}. First, however, we need to define the analog of continuous differentiability for real-valued functionals in the Wasserstein space.

\begin{definition}[continuous differentiability in the Wasserstein space]\label{def:continous differentiability}
    A real-valued functional $\constraint:\Pp{2}{\reals^d}\to\reals$ is \emph{Wasserstein continuously differentiable} at $\mu\in\Pp{2}{\reals^d}$ if for all $T_1=\gradient{}\phi_1$ and $T_2=\gradient{}\phi_2$ with $\phi_1,\phi_2\in\Ccinf{\reals^d}$ there exists $\bar\varepsilon>0$ so that 
    \begin{enumerate}
        \item $\constraint$ is Wasserstein differentiable at $\mu_{\varepsilon_1,\varepsilon_2}\coloneqq \pushforward{(\Id+\varepsilon_1T_1(x)+\varepsilon_2T_2(x))}\mu$ for all $\varepsilon_1,\varepsilon_2\in(-\bar\varepsilon,\bar\varepsilon)$; 
        \item the real-valued functions $(\varepsilon_1,\varepsilon_2)\mapsto\int_{\reals^d}\innerProduct{\gradient{\mu}\constraint(\mu_{\varepsilon_1,\varepsilon_2})(x)}{T_i(x)}\d\mu_{\varepsilon_1,\varepsilon_2}(x)$, with $i\in\{1,2\}$, are continuous at all $\varepsilon_1,\varepsilon_2\in(-\bar\varepsilon,\bar\varepsilon)$.
    \end{enumerate}
\end{definition}

In particular, all differentiable functionals presented in~\cref{subsec:examples of smooth functionals} are Wasserstein continuously differentiable.

\begin{proposition}[continuous differentiability of functionals]\label{prop:continuous differentiability functionals}
\,
\begin{enumerate}
    \item
    Let $c$ satisfy the assumptions of~\cref{prop:optimal transport discrepancy} and the twist condition (see \cref{prop:optimal transport discrepancy}(vii)), and let $\mu\in\Ppabs{2}{\reals^d}$ be absolutely continuous. Then, the optimal transport discrepancy~\eqref{eq:optimal transport discrepancy:definition} is Wasserstein continuously differentiable at $\mu$. 
    
    \item 
    Let $V$ satisfy the assumptions of~\cref{prop:expectation} and $\mu\in\Pp{2}{\reals^d}$. Then, the expected value~\eqref{eq:expected value:definition} is Wasserstein continuously differentiable at $\mu$.

    \item 
    Let $U$ satisfy the assumptions of~\cref{prop:interaction} and $\mu\in\Pp{2}{\reals^d}$. Then, the interaction-type functional~\eqref{eq:interaction:definition} is Wasserstein continuously differentiable at $\mu$.
\end{enumerate}
\end{proposition}

\ifbool{showproofs}{\booltrue{showprooffcontinuousdifferentiability}}{
\ifbool{arxiv}{\boolfalse{showprooffcontinuousdifferentiability}}{}} % not show in arxiv
\ifbool{showprooffcontinuousdifferentiability}{\prooffcontinuousdifferentiability{showprooffcontinuousdifferentiability}}{}

Armed with continuous differentiability, we can now formulate first-order necessary conditions for the equality-constrained optimization problem~\eqref{eq:constrained optimization}. Similarly to Euclidean settings, necessary conditions for optimality take the form of Lagrange multipliers.

\begin{theorem}[first-order necessary conditions for equality-constrained optimization]\label{thm:wasserstein lagrange multipliers}
Let $\mu^\ast$ be a local minimizer of \eqref{eq:constrained optimization}. Assume that
\begin{enumerate}
    \item $\cost$ is Wasserstein differentiable at $\mu^\ast$;
    \item $\constraint$ is Wasserstein continuously differentiable at $\mu^\ast$; and
    \item $\gradient{\mu}\constraint (\mu^\ast)$ is non-vanishing at $\mu^\ast$, that is, there exists $H\in\tangentSpace{\mu^\ast}\Pp{2}{\reals^d}$ such that
\begin{equation*}
    \int_{\reals^d}\innerProduct{\gradient{\mu}K(\mu^\ast)(x)}{H(x)}\d\mu^\ast(x)\neq 0.
\end{equation*}
\end{enumerate}
Then, there exists a unique $\multiplier\in\reals$ (called multiplier) such that 
\begin{equation*}
    \gradient{\mu}\cost(\mu^\ast)(x)
    +
    \lambda
    \gradient{\mu}\constraint(\mu^\ast)(x)=0\quad\text{for }\mu^\ast\text{-almost every } x\in\reals^d.
\end{equation*}
\end{theorem}

\ifbool{showproofnecessaryconditions}{
To prove the theorem we need the following preliminary lemma, which will later be instrumental in ``making arbitrary variations admissible'':
\begin{lemma}\label{lemma:admissibility variation}
Let $\mu^\ast$ and $H$ as in \cref{thm:wasserstein lagrange multipliers}. Let $T=\gradient{}\psi$ for some $\psi\in\Ccinf{\reals^d}$. Then, there exists $\bar\varepsilon>0$ and a continuously differentiable function $\sigma_{\bar\varepsilon}\in\C{1}{}:(-\bar\varepsilon,\bar\varepsilon)\to\reals$ with $\sigma(0)=0$ such that for all $\varepsilon\in(-\bar\varepsilon,\bar\varepsilon)$
\begin{equation*}
    \constraint(\pushforward{(\Id+\varepsilon T+\sigma(\varepsilon)H)}\mu)=0.
\end{equation*}
\end{lemma}

\proofadmissiblevariation{booltrue}
\prooffirstordernecessaryconditionsconstrained{boolfalse}
}

We will defer first-order sufficient conditions to the next section, as they are a simple corollary of first-order sufficient conditions for inequality-constrained optimization.

\subsection{Inequality-constrained Optimization}

Consider now the inequality-constrained optimization problem
\begin{equation}\label{eq:inequality constrained optimization}
    \inf_{\mu\in\Pp{2}{\reals^d}}
    \left\{
        \cost(\mu): \constraint(\mu)\leq 0
    \right\},
\end{equation}
for some lower semi-continuous (w.r.t. weak convergence in $\Pp{2}{\reals^d}$) functional $\constraint:\Pp{2}{\reals^d}\to\reals$. In this setting, the definition of local, global, and strict minimizers is analogous to the ones of~\cref{def:minima constrained} for equality-constrained optimization.
As in Euclidean settings, necessary conditions are effectively analogous to necessary conditions for equality-constrained optimization, with a complementarity condition and sign restriction on the multiplier:

\begin{theorem}[first-order necessary conditions for inequality-constrained optimization]\label{thm:inequality:wasserstein lagrange multipliers}
Let $\mu^\ast$ be a local minimizer of \eqref{eq:inequality constrained optimization}. Assume that
\begin{enumerate}
    \item $\cost$ is Wasserstein differentiable at $\mu^\ast$;
    
    \item $\constraint$ is Wasserstein continuously differentiable at $\mu^\ast$; and 

    \item $\gradient{\mu}\constraint (\mu^\ast)$ is non-vanishing at $\mu^\ast$, that is, there exists $H\in\tangentSpace{\mu^\ast}\Pp{2}{\reals^d}$ such that
\begin{equation*}
    \int_{\reals^d}\innerProduct{\gradient{\mu}K(\mu^\ast)(x)}{H(x)}\d\mu^\ast(x)\neq 0.
\end{equation*}
\end{enumerate}
Then, there exists a unique $\multiplier\geq 0$ (called multiplier) such that
\begin{align*}
    \gradient{\mu}\cost(\mu^\ast)(x)
    +
    \lambda
    \gradient{\mu}\constraint(\mu^\ast)(x)&=0
    \quad\text{for }\mu^\ast\text{-almost every } x\in\reals^d
    \\
    \lambda\constraint(\mu^\ast)&=0.
\end{align*}
\end{theorem}

\ifbool{showproofs}{\booltrue{showprooffirstordernecessaryconditionsconstrainedinequality}}{
\ifbool{arxiv}{\boolfalse{showprooffirstordernecessaryconditionsconstrainedinequality}}{}} % not show in arxiv
\ifbool{showprooffirstordernecessaryconditionsconstrainedinequality}{\prooffirstordernecessaryconditionsconstrainedinequality{showprooffirstordernecessaryconditionsconstrainedinequality}}{}

In particular, \cref{thm:inequality:wasserstein lagrange multipliers} is analogous to~\cref{thm:wasserstein lagrange multipliers}, with the additional conditions of complementary slackness (i.e., $\lambda\constraint(\mu^\ast)=0$) and non-negativity of the multiplier.

\begin{remark}
A few remarks are in order.
First, \cref{thm:first order conditions,thm:wasserstein lagrange multipliers,thm:inequality:wasserstein lagrange multipliers} provide necessary conditions for optimality. In particular, they do \emph{not} provide conditions for existence of minimizers for $\cost$, ~\eqref{eq:constrained optimization}, or~\eqref{eq:inequality constrained optimization}.
% Second, \cref{thm:wasserstein lagrange multipliers} readily extends to multidimensional constraints of the form $\constraint{}:\Pp{2}{\reals^d}\to\reals^l$, with $l\in\naturals_{\geq 1}$, which we deliberately omitted to ease the exposition.
Second, \cref{thm:first order conditions,thm:wasserstein lagrange multipliers,thm:inequality:wasserstein lagrange multipliers} apply to maximization problems with no changes. One simply considers upper semi-continuous (w.r.t. weak convergence in $\Pp{2}{\reals^d}$) functional $\cost:\Pp{2}{\reals^d}\to[-\infty,+\infty)$ with the effective domain $\effectiveDomain{\cost}\coloneqq\{\mu\in\Pp{2}{\reals^d}: \cost(\mu)>-\infty\}\neq\emptyset$. 
We will tacitly use this fact throughout the next sections.
\end{remark}

Under a convexity assumption, we can also study sufficient conditions. The first result is an application of~\cref{thm:first order sufficient conditions}.

\begin{theorem}[first-order sufficient conditions for inequality-constrained optimization]
\label{thm:first order sufficient conditions constrained}
Let $\cost:\Pp{2}{\reals^d}\to\realsBar$ and $\constraint:\Pp{2}{\reals^d}\to\reals$ be convex along geodesics.
Suppose there exist (finite) $\lambda\geq 0$ and $\mu^\ast\in\Pp{2}{\reals^d}$ such that
\begin{enumerate}
    \item $\cost$ is subdifferentiable at $\mu^\ast$;
    \item $\constraint$ is Wasserstein differentiable at $\mu^\ast$;
    \item $\mu^\ast$ is feasible (i.e., $\constraint(\mu^\ast)\leq 0$);
    \item complementary slackness holds so that $\lambda\constraint(\mu^\ast)=0$; and 
    \item there is a subgradient $\xi\in\subdifferential{\cost}(\mu^\ast)$ so that 
\begin{equation*}
\begin{aligned}
    \xi(x)+\lambda\gradient{\mu}\constraint(\mu^\ast)(x)&=0
    \quad\text{for }\mu^\ast\text{-almost every } x\in\reals^d.
\end{aligned}
\end{equation*}
\end{enumerate}
Then, $\mu^\ast$ is a global minimizer of~\eqref{eq:inequality constrained optimization}. Moreover, if $\cost$ is $\alpha_{\cost}$-convex along geodesics, $\constraint$ is $\alpha_{\constraint}$-convex along geodesics, and $\alpha_{\cost}+\lambda\alpha_{\constraint}>0$, then $\mu^\ast$ is the strict minimizer.
\end{theorem}

\ifbool{showproofs}{\booltrue{showprooffirstordersufficientconditionsconstrained}}{
\ifbool{arxiv}{\boolfalse{showprooffirstordersufficientconditionsconstrained}}{}} % not show in arxiv
\ifbool{showprooffirstordersufficientconditionsconstrained}{\prooffirstordersufficientconditionsconstrained{showprooffirstordersufficientconditionsconstrained}}{}

If the constraint is ``linear'' (i.e., both geodesically convex and concave) we obtain first-order sufficient conditions for equality-constrained optimization.

\begin{corollary}[``linear'' constraints]
\label{cor:first order sufficient conditions constrained}
Let $\cost:\Pp{2}{\reals^d}\to\realsBar$ be geodesically convex, and $\constraint:\Pp{2}{\reals^d}\to\reals$ be both geodesically convex and geodesically concave.
Consider 
\begin{equation}\label{eq:first order sufficient conditions constrained corollary}
    \inf_{\mu\in\Pp{2}{\reals^d}}
    \left\{
        \cost(\mu): \constraint(\mu)= 0
    \right\}.
\end{equation}
Suppose there exists (finite) $\lambda\in\reals$ and $\mu^\ast\in\Pp{2}{\reals^d}$ such that
\begin{enumerate}
    \item $\cost$ is Wasserstein subdifferentiable at $\mu^\ast$; 
    \item $\constraint$ is Wasserstein differentiable at $\mu^\ast$;
    \item $\mu^\ast$ is feasible (i.e., $\constraint(\mu^\ast)=0$); and 
    \item there is a subgradient $\xi\in\subdifferential{\cost}(\mu^\ast)$ so that
\begin{equation*}
\begin{aligned}
    \xi(x)+\lambda\gradient{\mu}\constraint(\mu^\ast)(x)=0\quad\text{for }\mu^\ast\text{-almost every } x\in\reals^d.
\end{aligned}
\end{equation*}
\end{enumerate}
Then, $\mu^\ast$ is a global minimizer of \eqref{eq:first order sufficient conditions constrained corollary}.  Moreover, if $\cost$ is $\alpha_{\cost}$-geodesically convex with $\alpha_{\cost}>0$, then $\mu^\ast$ is the strict minimizer.
\end{corollary}

\ifbool{showproofs}{\booltrue{showproofcorfirstordersufficientconditionsconstrained}}{
\ifbool{arxiv}{\boolfalse{showproofcorfirstordersufficientconditionsconstrained}}{}} % not show in arxiv
\ifbool{showproofcorfirstordersufficientconditionsconstrained}{\proofcorfirstordersufficientconditionsconstrained{showproofcorfirstordersufficientconditionsconstrained}}{}

Unfortunately, the Wasserstein distance is \emph{not} convex along geodesics, which prevents us from deploying \cref{thm:first order sufficient conditions constrained} whenever $\constraint$ is defined in terms of the Wasserstein distance (e.g., $\constraint(\mu)=\wassersteinDistance{2}{\mu}{\bar\mu}-\varepsilon$ for some $\bar\mu\in\Pp{2}{\reals^d}$ and $\varepsilon>0$).
We resolve this issue in greater generality by studying optimal transport discrepancies.

\begin{theorem}[first-order sufficient conditions for constrained optimization, revisited]\label{thm:first order sufficient conditions constrained wasserstein}
Let $\cost:\Pp{2}{\reals^d}\to\realsBar$ be convex along any interpolation curve with convexity parameter $\alpha_J\in\reals$ and suppose $c:\reals^d\times\reals^d\to\nonnegativeReals$ satisfies the assumptions of~\cref{prop:optimal transport discrepancy} and the twist condition (see~\cref{prop:optimal transport discrepancy}(vii)). Let $\bar\mu\in\Pp{2}{\reals^d}$, and suppose that $c(\cdot,y)$ is $\alpha_c$-convex (with $\alpha_c\in\reals$) for all $y\in\support{\bar\mu}$. Consider 
\begin{equation}\label{eq:first order sufficient conditions constrained wasserstein}
    \inf_{\mu\in\Pp{2}{\reals^d}}
    \left\{
        \cost(\mu): \optimalTransportDiscrepancy{c}{\mu}{\bar\mu}\leq\varepsilon
    \right\}.
\end{equation}
Suppose there exist (finite) $\lambda\geq 0$ so that $\alpha_J+\lambda\alpha_c\geq 0$
%\frac{\negativePart{\alpha_J}}{\alpha_c}$
and $\mu^\ast\in\Pp{2}{\reals^d}$ such that 
\begin{enumerate}
    \item $\cost$ is subdifferentiable at $\mu^\ast$;
    \item $\mu^\ast$ is absolutely continuous or, more generally, there is a unique optimal transport map (w.r.t. the cost $c$) from $\mu^\ast$ to $\bar\mu$ (i.e., $\setOptimalPlans{\mu^\ast}{\bar\mu}=\{\pushforward{(\Id\times\optMap{\mu^\ast}{\bar\mu})}\mu^\ast\}$);
    \item $\mu^\ast$ is feasible (i.e., $\optimalTransportDiscrepancy{c}{\mu}{\bar\mu}\leq\varepsilon$);
    \item complementary slackness holds so that $\lambda(\optimalTransportDiscrepancy{c}{\mu}{\bar\mu}-\varepsilon)=0$, and 
    \item there exists a subgradient $\xi\in\subdifferential{J}(\mu^\ast)$ so that 
\begin{equation*}
\begin{aligned}
    \xi(x)+\lambda\gradient{x} c(x,\optMap{\mu^\ast}{\bar\mu}(x))&=0\quad\text{for }\mu^\ast\text{-almost every } x\in\reals^d.
\end{aligned}
\end{equation*}
\end{enumerate}
Then, $\mu^\ast$ is a global minimizer of \eqref{eq:first order sufficient conditions constrained wasserstein}. If additionally $\alpha_J+\lambda\alpha_c>0$
%or $\lambda>\frac{\negativePart{\alpha_J}}{\alpha_c}$
, then $\mu^\ast$ is the strict minimizer. 
\end{theorem}

\ifbool{showproofs}{\booltrue{showprooffirstordersufficientconditionsconstrainedwasserstein}}{
\ifbool{arxiv}{\boolfalse{showprooffirstordersufficientconditionsconstrainedwasserstein}}{}} % not show in arxiv
\ifbool{showprooffirstordersufficientconditionsconstrainedwasserstein}{\prooffirstordersufficientconditionsconstrainedwasserstein{showprooffirstordersufficientconditionsconstrainedwasserstein}}{}

The special case of the Wasserstein distance is then obtained as a corollary. In this case, convexity along generalized geodesics (instead of general interpolation curves) suffices: 
\begin{corollary}[first-order sufficient conditions for constrained optimization, revisited]\label{cor:first order sufficient conditions constrained wasserstein}
Let $\cost:\Pp{2}{\reals^d}\to\realsBar$ be convex along generalized geodesics with convexity parameter $\alpha\in\reals$. 
Let $\bar\mu\in\Pp{2}{\reals^d}$, and consider 
\begin{equation}\label{eq:first order sufficient conditions constrained wasserstein corollary}
    \inf_{\mu\in\Pp{2}{\reals^d}}
    \left\{
        \cost(\mu): \wassersteinDistance{2}{\mu}{\bar\mu}\leq\varepsilon
    \right\}
    \equiv
    \inf_{\mu\in\Pp{2}{\reals^d}}
    \left\{
        \cost(\mu): \wassersteinDistance{2}{\mu}{\bar\mu}^2\leq\varepsilon^2
    \right\}.
\end{equation}
Suppose there exist $\lambda\geq\frac{\negativePart{\alpha}}{2}$ and $\mu^\ast\in\Pp{2}{\reals^d}$ such that 
\begin{enumerate}
    \item $\cost$ is subdifferentiable at $\mu^\ast$; 
    \item $\mu^\ast$ is absolutely continuous or, more generally, there is a unique optimal transport map from $\mu^\ast$ to $\bar\mu$ (i.e., $\setOptimalPlans{\mu^\ast}{\bar\mu}=\{\pushforward{(\Id\times\optMap{\mu^\ast}{\bar\mu})}\mu^\ast\}$), 
    \item $\mu^\ast$ is feasible (i.e., $\wassersteinDistance{2}{\mu^\ast}{\bar\mu}\leq\varepsilon$);
    \item complementary slackness holds so that $\lambda(\wassersteinDistance{2}{\mu^\ast}{\bar\mu}-\varepsilon)=0$, and
    \item there exists a subgradient $\xi\in\subdifferential{J}(\mu^\ast)$ so that 
\begin{equation*}
    \xi(x)+2\lambda(x-\optMap{\mu^\ast}{\bar\mu}(x))=0\quad\text{for }\mu^\ast\text{-almost every } x\in\reals^d.
\end{equation*}
\end{enumerate}
Then, $\mu^\ast$ is a global minimizer of \eqref{eq:first order sufficient conditions constrained wasserstein corollary}. If additionally $\alpha>0$ or $\lambda>\frac{\negativePart{\alpha}}{2}$, then $\mu^\ast$ is the strict minimizer. 
\end{corollary}

\ifbool{showproofs}{\booltrue{showprooffirstordersufficientconditionsconstrainedwassersteincorollary}}{
\ifbool{arxiv}{\boolfalse{showprooffirstordersufficientconditionsconstrainedwassersteincorollary}}{}} % not show in arxiv
\ifbool{showprooffirstordersufficientconditionsconstrainedwassersteincorollary}{\prooffirstordersufficientconditionsconstrainedwassersteincorollary{showprooffirstordersufficientconditionsconstrainedwassersteincorollary}}{}

We conclude the section with an example of the application of our first-order necessary and sufficient conditions:  

\begin{example}
Consider the optimization problem
\begin{equation*}
    \inf_{\mu\in\Pp{2}{\reals}}
    \left\{\cost(\mu)\coloneqq\expectedValue{\mu}{\frac{1}{2}x^2}: \constraint(\mu)\coloneqq 1-\expectedValue{\mu}{x}\leq 0\right\}.
\end{equation*}
Let us assume that a minimizer exists; we will later use sufficient conditions to prove its optimality. 
As $\gradient{\mu}\constraint(\mu)=-1$ for all $\mu\in\Pp{2}{\reals^d}$, it is straightforward to construct $\varphi$ so that $\int_{\reals}\innerProduct{-1}{\gradient{}\varphi(x)}\d\mu^\ast(x)=0$. Hence, by~\cref{thm:wasserstein lagrange multipliers}, at optimality, the Wasserstein gradient $\gradient{\mu}\cost(\mu)=\Id$ necessarily coincides with the constant function $\lambda$ $\mu^\ast$-almost everywhere, where $\lambda\geq 0$.
Thus, $\int_{\reals}\abs{x-\lambda}\d\mu^\ast(x)=0$, and so $\mu^\ast=\delta_{\lambda}$.
We now distinguish two cases. If $\lambda=0$, we have $\mu^\ast=\delta_0$, which however violates the constraint $K(\mu^\ast)\leq 0$. If instead $\lambda>0$, complementary slackness gives $K(\mu^\ast)=0$ and so  $\expectedValue{\delta_{\lambda}}{x}=1$.
We therefore conclude $\mu^\ast=\delta_{1}$ and $\lambda=1$. By~\cref{prop:expectation}, $\cost$ and $\constraint$ are convex along geodesics with $\alpha_{\cost}=1$ and $\alpha_{\constraint}=0$. Thus, $\mu^\ast$ the strict minimizer, by \cref{thm:first order sufficient conditions constrained}.
\end{example}

\section{Application to the Evaluation of Worst-case Risk over Wasserstein Ambiguity Sets}\label{sec:dro}

Consider the problem of evaluating the worst-case risk of a real-valued function over an ambiguity set defined in terms of the Wasserstein distance, intimately related to~\gls{acr:dro} (cf.~\cref{ex:dro}). 
It is well known that when the risk is the expected value of an (integrable) real-valued function $f:\reals^d\to\reals$, then the worst-case risk admits the following dual reformulation~\cite{Blanchet2019,gao2022distributionally}:
\begin{equation}\label{eq:general dro dual}
    \sup_{\mu\in\closedWassersteinBall{2}{\radius}{\refmu}} \expectedValue{\mu}{f}
    =
    \inf_{\lambda\geq 0} \lambda\varepsilon^2-\expectedValue{\refmu}{\inf_{y\in\reals^d}\lambda\norm{x-y}^2-f(y)},
\end{equation}
where $\closedWassersteinBall{2}{\radius}{\refmu}$ denotes the \emph{closed}\footnote{Here, the closure is intended with respect to the topology induced by the Wasserstein distance.} Wasserstein ball of radius $\radius$ centered at $\refmu\in\Pp{2}{\reals^d}$:
\begin{equation*}
    \closedWassersteinBall{2}{\radius}{\refmu}
    \coloneqq 
    \{
    \mu\in\Pp{2}{\reals}:
    \wassersteinDistance{2}{\mu}{\refmu}\leq \radius\}.
\end{equation*}
In this section, we use necessary and sufficient conditions for optimality in the Wasserstein space to study the more general setting
\begin{equation}\label{eq:general dro}
    J^\ast\coloneqq \sup_{\mu\in\closedWassersteinBall{2}{\radius}{\refmu}}\cost(\mu)
\end{equation}
for several cost functionals $\cost:\Pp{2}{\reals^d}\to\reals$ which are \emph{not} necessarily representable as expected values.
We start by showing that, up to arbitrarily small perturbations of the radius of the Wasserstein ball, one can always assume that $\refmu$ is absolutely continuous. 

\begin{proposition}[absolute continuity of the center]
\label{prop:absolutely continuity ref}
Let $\varepsilon>0$, $\delta\in(0,\frac{\varepsilon}{2})$, and $\refmu\in\Pp{2}{\reals^d}$.
Then, there exists $\refmu'_\delta\in\closedWassersteinBall{2}{\delta}{\refmu}$ absolutely continuous w.r.t. the Lebesgue measure such that
\begin{equation*}
\begin{aligned}
    \sup_{\mu\in\closedWassersteinBall{2}{\radius-2\delta}{\refmu'_\delta}}\cost(\mu)
    \leq 
    \sup_{\mu\in\closedWassersteinBall{2}{\radius-\delta}{\refmu}}\cost(\mu)
    \leq 
    \sup_{\mu\in\closedWassersteinBall{2}{\radius}{\refmu'_\delta}}&\cost(\mu)
    \leq 
    \sup_{\mu\in\closedWassersteinBall{2}{\radius+\delta}{\refmu}}\cost(\mu)
    \leq 
    \sup_{\mu\in\closedWassersteinBall{2}{\radius+2\delta}{\refmu'_\delta}}\cost(\mu).
\end{aligned}
\end{equation*}
\end{proposition}

\ifbool{showproofs}{\booltrue{showproofabsolutelycontinuityref}}{
\ifbool{arxiv}{\booltrue{showproofabsolutelycontinuityref}}{}} % show in arxiv
\ifbool{showproofabsolutelycontinuityref}{\proofabsolutelycontinuityref{showproofabsolutelycontinuityref}}{}

Second, we recall from~\cite{Ambrosio2008a,Yue2021} well-known compactness results for Wasserstein balls\ifbool{arxiv}{. Its proof, included for completeness, is relegated to~\cref{app:proofs}.}{.}
\begin{proposition}[properties of Wasserstein balls~\cite{Ambrosio2008a,Yue2021}]\label{prop:wasserstein balls}
Let $\refmu\in\Pp{2}{\reals^d}$ and $\radius>0$. Then, 
\begin{enumerate}
    \item $\closedWassersteinBall{2}{\radius}{\refmu}$ is closed, but not compact w.r.t. weak convergence;
    \item $\closedWassersteinBall{2}{\radius}{\refmu}$ is compact w.r.t. narrow convergence.   
\end{enumerate}
\end{proposition}

\ifbool{showproofs}{\booltrue{showproofwassersteinballs}}{
\ifbool{arxiv}{\boolfalse{showproofwassersteinballs}}{}} % not show in arxiv
\ifbool{showproofwassersteinballs}{\proofwassersteinballs{showproofwassersteinballs}}{}

\subsection{Pedagogical Example}
\label{sec:dro warmup}
As a pedagogical example, consider
\begin{equation}\label{eq:dro linear cost}
    J(\mu)
    =
    \expectedValue{\mu}{\innerProduct{w}{x}}
\end{equation}
for some non-zero $w\in\reals^d$. In this case, the solution of~\eqref{eq:general dro} is easily found via~\eqref{eq:general dro dual}.
Yet, as an illustrative example, we study \eqref{eq:general dro} through necessary and sufficient conditions for optimality in the Wasserstein space.
% To do so, we first show that maxima are attained at the boundary of the Wasserstein ball, and so we can replace the inequality $\wassersteinDistance{2}{\mu}{\refmu}\leq\varepsilon$ with equality: 
% \begin{lemma}[Existence of a maximum at the boundary]
% \label{lemma:min boundary linear cost}
% There exists a worst-case probability measure $\mu^\ast$ attaining the supremum~\eqref{eq:general dro} such that $\wassersteinDistance{2}{\refmu}{\mu^\ast}=\radius$.
% \end{lemma}
% \ifbool{showproofs}{\booltrue{showprooflemmaminboundarylinearcost}}{
% \ifbool{arxiv}{\booltrue{showprooflemmaminboundarylinearcost}}{}} % show in arxiv
% \ifbool{showprooflemmaminboundarylinearcost}{\prooflemmaminboundarylinearcost{showprooflemmaminboundarylinearcost}}{}
Without loss of generality, we assume that $\refmu$ is absolutely continuous; see \cref{prop:absolutely continuity ref}. The proof strategy is conceptually simple: We assume enough regularity to use necessary conditions for optimality (\cref{thm:inequality:wasserstein lagrange multipliers}) to construct candidates for optimality and then leverage sufficient conditions for optimality (\cref{thm:first order sufficient conditions constrained wasserstein}) to prove that one candidate solution is indeed optimal. 
\begin{proposition}[optimal solution]
\label{prop:dro linear cost}
Consider problem~\eqref{eq:general dro} with cost functional~\eqref{eq:dro linear cost}.
Assume $\refmu\in\Ppabs{2}{\reals^d}$ is absolutely continuous with respect to the Lebesgue measure.
Then, the unique worst-case probability measure is $\mu^\ast=\pushforward{T}\refmu$, where
\begin{equation*}
    T(x)=x+\radius\frac{w}{\norm{w}},
\end{equation*}
and the corresponding worst-case cost is
\begin{equation}\label{eq:dro linear cost cost}
    \cost^\ast=\expectedValue{\refmu}{\innerProduct{w}{x}}+\radius\norm{w}.
\end{equation}
\end{proposition}

\ifbool{showproofs}{\booltrue{showproofdrolinearcost}}{
\ifbool{arxiv}{\boolfalse{showproofdrolinearcost}}{}} % not show in arxiv
\ifbool{showproofdrolinearcost}{\proofdrolinearcost{showproofdrolinearcost}}{}

In particular,~\eqref{eq:dro linear cost cost} suggests that the~\gls{acr:dro} problem is equivalent to an explicit regularization scheme of the nominal cost $\cost(\refmu)$. We refer to~\cite{gao2022finite,gao2017wasserstein,shafieezadeh2019regularization} for details on the interplay between \gls{acr:dro} and explicit regularization.

\subsection{Mean-Variance}
Consider now the mean-variance functional 
\begin{equation}\label{eq:dro variance}
    J(\mu)
    =
    \expectedValue{\mu}{\innerProduct{w}{x}}+\weight\variance{\mu}{\innerProduct{w}{x}}
\end{equation}
for some non-zero $w\in\reals^d$ and (finite) $\weight>0$. Among others, the risk measure~\eqref{eq:dro variance} is widely used in portfolio selection problems~\cite{Markowitz1952,Merton1969}. Being a non-linear functional of the measure, the duality result~\eqref{eq:general dro dual} does not apply and the mean-variance functional has received little attention in the context of \gls{acr:dro}~\cite{Blanchet2021DistributionallyDistances,Nguyen2021Mean-CovarianceMeasurement}.
We start by studying the properties of $\cost$.

\begin{lemma}[properties of $\cost$]
\label{lemma:properties dro variance}
The mean-variance functional $-\cost$ is continuous with respect to weak convergence and $(-2\weight\norm{w}^2)$-convex along any interpolating curve.
\end{lemma}

\ifbool{showproofs}{\booltrue{showprooflemmapropertiesdrovariance}}{
\ifbool{arxiv}{\boolfalse{showprooflemmapropertiesdrovariance}}{}} % not show in arxiv
\ifbool{showprooflemmapropertiesdrovariance}{\prooflemmapropertiesdrovariance{showprooflemmapropertiesdrovariance}}{}

% \begin{remark}\label{rem:variance lsc narrow}
% In this case, we cannot resort to lower semi-continuity and compactness w.r.t. narrow convergence to prove existence of maximizers. Indeed, $-\cost$ is \emph{not} lower semi-continuous w.r.t. narrow convergence. Consider $d=1$, $w=1$, and $\weight=1$. Let $\mu_n\coloneqq(1-\frac{1}{n})\diracMeasure{0}+\frac{1}{n}\diracMeasure{n}$, narrowly converging to $\diracMeasure{0}$. Then, $\liminf_{n\to\infty}-\cost(\mu_n)=\liminf_{n\to\infty}-1-(n-1^2)=-\infty$, but $-\cost(\diracMeasure{0})=0$. So, $\liminf_{n\to\infty}-\cost(\mu_n)\not\geq-\cost(\diracMeasure{0})$.
% \end{remark}
% Second, we show that, if they exist, worst-case probability measures lie at the boundary: 

% \begin{lemma}[Existence of a maximizer at the boundary]
% \label{lemma:min boundary variance}
% Assume that there exists a worst-case probability measure $\mu^\ast\in\Pp{2}{\reals^d}$ attaining the supremum~\eqref{eq:general dro}. Then, $\mu^\ast$ lies on the boundary of the Wasserstein ball; i.e., $\wassersteinDistance{2}{\refmu}{\mu^\ast}=\radius$.
% \end{lemma}

% \ifbool{showproofs}{\booltrue{showprooflemmaminboundaryvariance}}{
% \ifbool{arxiv}{\booltrue{showprooflemmaminboundaryvariance}}{}} % show in arxiv
% \ifbool{showprooflemmaminboundaryvariance}{\prooflemmaminboundaryvariance{showprooflemmaminboundaryvariance}}{}

We can now deploy~\cref{thm:inequality:wasserstein lagrange multipliers,thm:first order sufficient conditions constrained wasserstein} to characterize optimal solutions. More specifically, we will now show that the evaluation of the worst-case probability measure amounts to computing the roots of a fourth-order polynomial:

\begin{proposition}[optimal solution of mean-variance \gls{acr:dro}]
\label{thm:dro variance}
Consider problem~\eqref{eq:general dro dual} with the mean-variance cost functional~\eqref{eq:dro variance}.
Assume $\refmu\in\Ppabs{2}{\reals^d}$ is absolutely continuous with respect to the Lebesgue measure.
Then, the worst-case probability measure is unique, and it is given by $\mu^\ast=\pushforward{T}\refmu$, where
\begin{equation*}
    T(x)=
    \left(\eye{}+\frac{\frac{\weight}{\lambda^\ast}}{1-\frac{\weight}{\lambda^\ast}\norm{w}^2}w\transpose{w}\right)x + \left(\frac{1}{2\lambda^\ast}-\frac{\frac{\weight}{\lambda^\ast}}{1-\frac{\weight}{\lambda^\ast}\norm{w}^2}\expectedValue{\refmu}{\innerProduct{w}{x}}\right)w,
\end{equation*}
where $\lambda^\ast$ is the unique real root strictly larger than $\weight\norm{w}^2$ of the fourth-order polynomial
\begin{equation*}
    \lambda^4
    -2\weight\norm{w}^2 \lambda^3
    +\left(-\frac{\weight^2\norm{w}^2}{\radius^2}\variance{\refmu}{\innerProduct{w}{x}} - \frac{\norm{w}^2}{4\radius^2}+\weight^2\norm{w}^4\right)\lambda^2
    +\frac{\weight\norm{w}^4}{2\radius^2}\lambda
    -\frac{\weight^2\norm{w}^6}{4\radius^2}.
\end{equation*}
The resulting worst-case cost reads
\begin{equation*}
    \cost^\ast
    =
    \expectedValue{\refmu}{\innerProduct{w}{x}}+\frac{1}{2\lambda^\ast}\norm{w}^2+\weight\left(\frac{1}{1-\frac{\weight}{\lambda^\ast}\norm{w}^2}\right)^2\variance{\refmu}{\innerProduct{w}{x}}.
\end{equation*}
\end{proposition}

\ifbool{showproofs}{\booltrue{showproofdrovariance}}{
\ifbool{arxiv}{\boolfalse{showproofdrovariance}}{}} % not show in arxiv
\ifbool{showproofdrovariance}{\proofdrovariance{showproofdrovariance}}{}

\ifbool{compact}{}{
\begin{remark}
When $\weight\to 0$, $\lambda^\ast=1/(2\varepsilon)$ and we recover the results of~\cref{sec:dro warmup}.
Conversely, when $\weight\to\infty$, $\lambda^\ast/\weight=\norm{w}^2\pm\norm{w}\standardDeviation{\refmu}{\innerProduct{w}{x}}/\radius$, and 
\begin{equation}\label{eq:dro variance infinity}
\begin{aligned}
    \frac{\cost^\ast}{\weight}
    &=
    \left(
    \standardDeviation{\refmu}{\innerProduct{w}{x}}+\radius\norm{w}
    \right)^2,
\end{aligned}
\end{equation}
where $\standardDeviation{\refmu}{\innerProduct{w}{x}}\coloneqq\sqrt{\variance{\refmu}{\innerProduct{w}{x}}}$ is the standard deviation of $\innerProduct{w}{x}$. This is again a form of explicit regularization. 
\end{remark}}

Our results are in agreement with~\cite[Appendix C]{Nguyen2021Mean-CovarianceMeasurement}, and additionally fully characterize the worst-case probability measure and prove its uniqueness.
Differently from~\cite{Nguyen2021Mean-CovarianceMeasurement}, we do not base our analysis on structural ambiguity sets consisting of all semidefinite affine pushforwards of the center $\refmu$~\cite[Definitions 6]{Nguyen2021Mean-CovarianceMeasurement}, but we work in the ``full'' space of probability measures~$\Pp{2}{\reals^d}$. Accordingly, the fact that here and below worst-case probability measures sometimes happen to be positive semidefinite affine pushforward of $\refmu$ is not stipulated ex-ante but emerges naturally from our optimality conditions.  In~\cite{Blanchet2021DistributionallyDistances}, instead, the mean-variance functional is replaced by the variance and a lower bound on the worst-case mean.

\subsection{Mean-Standard Deviation}
Consider now the mean-standard deviation functional 
\begin{equation}\label{eq:dro std}
    J(\mu)=\expectedValue{\mu}{\innerProduct{w}{x}}+\weight\standardDeviation{\mu}{\innerProduct{w}{x}}
    =\expectedValue{\mu}{\innerProduct{w}{x}}+\weight\sqrt{\variance{\mu}{\innerProduct{w}{x}}}
\end{equation}
for some non-zero $w\in\reals^d$ and (finite) $\weight>0$. This risk measure has also found application in portfolio theory~\cite{Rockafellar2002DeviationOptimization}. Due to its nonlinearity, the mean-standard deviation functional has not been studied in the context of~\gls{acr:dro}, with the exception of~\cite{Nguyen2021Mean-CovarianceMeasurement} which provides an upper bound (so-called Gelbrich risk) on the worst-case cost. 
First, we study properties of $\cost$.
\begin{lemma}[properties of $J$]\label{lemma:properties dro std}
The mean-standard deviation cost functional $-\cost$ is continuous with respect to weak convergence. Moreover, for all $\mu\in\closedWassersteinBall{2}{\radius}{\refmu}$
\begin{equation}\label{eq:dro std upper bound}
    \cost(\mu)\leq \cost(\refmu) + \radius\norm{w}\sqrt{1+\weight^2}.
\end{equation}
\end{lemma}

\ifbool{showproofs}{\booltrue{showprooflemmapropertiesdrostd}}{
\ifbool{arxiv}{\boolfalse{showprooflemmapropertiesdrostd}}{}} % not show in arxiv
\ifbool{showprooflemmapropertiesdrostd}{\prooflemmapropertiesdrostd{showprooflemmapropertiesdrostd}}{}

% \begin{remark}
% Again, $-\cost$ is not lower semi-continuous w.r.t. narrow convergence. The proof (by counter-example) is identical the one of the mean-variance functional (\cref{rem:variance lsc narrow}).
% \end{remark}
% Second, we show that, if they exist, worst-case probability measures lie at the boundary: 
% \begin{lemma}[Existence of a maximizer at the boundary]
% \label{lemma:min boundary std}
% Assume that there exists a worst-case probability measure $\mu^\ast$ attaining the supremum~\eqref{eq:general dro}. Then, $\mu^\ast$ lies on the boundary of the Wasserstein ball; i.e., $\wassersteinDistance{2}{\refmu}{\mu^\ast}=\radius$.
% \end{lemma}

% \ifbool{showproofs}{\booltrue{showprooflemmaboundarydrostd}}{
% \ifbool{arxiv}{\booltrue{showprooflemmaboundarydrostd}}{}} % show in arxiv
% \ifbool{showprooflemmaboundarydrostd}{\prooflemmaboundarydrostd{showprooflemmaboundarydrostd}}{}

Armed with this result, we can now study optimal solutions.
\begin{proposition}[optimal solution of mean-standard deviation~\gls{acr:dro}]
\label{thm:dro std}
Consider problem~\eqref{eq:general dro dual} with the mean-standard deviation cost functional~\eqref{eq:dro std}.
Assume that $\refmu\in\Ppabs{2}{\reals^d}$ is absolutely continuous with respect to the Lebesgue measure.
Then, the worst-case probability measure is given by $\mu^\ast=\pushforward{T}\refmu$, where
\begin{equation*}
    T(x)=
    \left(\eye{}+\frac{\weight\radius}{\norm{w}\sqrt{1+\weight^2}\standardDeviation{\refmu}{\innerProduct{w}{x}}}w\transpose{w}\right)x
    +
    \left(1-\frac{\weight\expectedValue{\refmu}{\innerProduct{w}{x}}}{\standardDeviation{\refmu}{\innerProduct{w}{x}}}\right)\frac{\radius}{\sqrt{1+\weight^2}}\frac{w}{\norm{w}},
\end{equation*}
and the resulting worst-case cost reads
\begin{equation}\label{eq:dro std worst case cost}
    \cost^\ast
    =
    \expectedValue{\refmu}{\innerProduct{w}{x}}+\weight\standardDeviation{\refmu}{\innerProduct{w}{x}}+\radius\norm{w}\sqrt{1+\weight^2}
    =
    \cost(\refmu)+\radius\norm{w}\sqrt{1+\weight^2}.
\end{equation}
\end{proposition}

\ifbool{showproofs}{\booltrue{showproofdrostd}}{
\ifbool{arxiv}{\boolfalse{showproofdrostd}}{}} % not show in arxiv
\ifbool{showproofdrostd}{\proofdrostd{showproofdrostd}}{}

As in~\cref{sec:dro warmup},~\eqref{eq:dro std worst case cost} suggests that~\eqref{eq:general dro} coincides with an explicit regularization scheme, as already observed in~\cite{wozabal2014robustifying}.

\section{Minimum Entropy on Wasserstein Balls}\label{sec:kl}

Consider the optimization problem
\begin{equation}\label{eq:min entropy wasserstein ball}
\begin{aligned}
    \inf_{\mu\in\Pp{2}{\reals^d}}
    \left\{\kullbackLeibler{\mu}{\muprior}:
    \wassersteinDistance{2}{\mu}{\muref}\leq\radius\right\},
\end{aligned}
\end{equation}
with prior probability measure $\mu_\mathrm{p}\in\Pp{2}{\reals^d}$ and reference probability measure $\muref\in\Pp{2}{\reals^d}$.
The optimization problem~\eqref{eq:min entropy wasserstein ball} combines the well-known Kullback’s principle of minimum cross-entropy~\cite{kullback1959information,shore1980axiomatic,shore1981properties} in statistical inference with additional side information, encoded in the probability measure $\muref$~\cite{Vargas2021}.
For instance, in~\cite{Vargas2021}, $\muref$ is chosen to be the empirical probability measure $\refmu=\frac{1}{N}\sum_{i=1}^N\diracMeasure{x_i}$ constructed with $N$ data points $\{(x_i)\}_{i=1}^N$, while $\muprior$ is a strictly positive probability measure on a compact subset of $\reals^d$.
Here, we study optimality conditions for~\eqref{eq:min entropy wasserstein ball}.
We assume that $\muprior=e^{-\Vprior}\lebesgueMeasure{d}$ and $\muref=e^{-\Vref}\lebesgueMeasure{d}$ for some $\Vprior:\reals^d\to\reals$ and $\Vref:\reals^d\to\reals$ convex and continuously differentiable.
To start, we show that a minimizer exists. 
\begin{lemma}[existence]\label{lemma:min entropy wasserstein ball existence}
The optimization problem \eqref{eq:min entropy wasserstein ball} has a solution. Moreover, any minimizer $\mu^\ast$ is absolutely continuous. 
\end{lemma}

\ifbool{showproofs}{\booltrue{showprooflemmaexistencekl}}{
\ifbool{arxiv}{\boolfalse{showprooflemmaexistencekl}}{}} % not show in arxiv
\ifbool{showprooflemmaexistencekl}{\prooflemmaexistencekl{showprooflemmaexistencekl}}{}

Without loss of generality, we assume that $\muprior$ does not belong to the Wasserstein ball of radius $\radius$ centered at $\muref$ (i.e., $\wassersteinDistance{2}{\muprior}{\muref}>\radius$); else, we trivially have $\mu^\ast=\muprior$. Then,~\cref{cor:first order sufficient conditions constrained wasserstein} leads to the following sufficient optimality condition, which yields an (almost) closed-form solution when $\muprior$ and $\muref$ are Gaussian. 

\begin{proposition}[conditions for optimality]
\label{prop:conditions for optimality kl}
%Consider problem~\eqref{eq:min entropy wasserstein ball} and let $\mu^\ast=\rho^\ast\lebesgueMeasure{d}\in\Ppabs{2}{\reals^d}$ be an optimal solution of \eqref{eq:min entropy wasserstein ball}.
Suppose there exists $\mu^\ast\in\Ppabs{2}{\reals^d}$ and $\lambda\geq 0$ such that (i) $\wassersteinDistance{2}{\mu^\ast}{\muref}\leq\radius$, (ii) $\lambda(\wassersteinDistance{2}{\mu^\ast}{\muref}-\radius)=0$, and (iii)
\begin{equation}\label{eq:minimum entropy necessary condition}
    \gradient{}\Vprior
    -\gradient{}\optMap{\mu^\ast}{\muref}(\gradient{}\Vref\circ\optMap{\mu^\ast}{\muref})
    +\trace\left((\gradient{}\optMap{\mu^\ast}{\muref})^{-1}\sum_{i=1}^d\diff{}{x_i}\gradient{}\optMap{\mu^\ast}{\muref}\right)
    +2\lambda(\Id - \optMap{\mu^\ast}{\muref})=0
\end{equation}
with all terms well-defined $\mu^\ast$-almost everywhere. Then, $\mu^\ast$ is an optimal solution of~\eqref{eq:min entropy wasserstein ball}.
In particular, if $\Vprior(x)$ and $\Vref(x)$ are convex quadratic of the form $\Vprior(x)=\frac{1}{2}\transpose{x}\Sprior^{-1} x-\innerProduct{\Sprior^{-1}\mprior}{x}+C_\mathrm{p}$ and $\Vref(x)=\frac{1}{2}\transpose{x}\Sref^{-1}x-\innerProduct{\Sref^{-1}\mref}{x}+C_\mathrm{r}$ for $\Sprior,\Sref\in\reals^{d\times d}$ symmetric and positive definite, $\mprior,\mref\in\reals^d$, and $C_\mathrm{p},C_\mathrm{v}$ chosen so that $\muref$ and $\muprior$ are valid probability measures (that is, $\muprior,\muref$ are Gaussian with mean $\muprior$, $\mref$ and variance $\Sprior$, $\Sref$, respectively), then the strict minimizer of~\eqref{eq:min entropy wasserstein ball} is the Gaussian probability measure
\begin{equation*}
    \mu^\ast
    =
    \pushforward{\left(A^{-1}x-(\Sprior^{-1}+2\lambda\eye{})^{-1}(\Sprior^{-1}\mprior-A\Sref^{-1}\mref)\right)}\muref,
\end{equation*}
where $A\in\reals^{n\times n}$ symmetric and positive definite and $\lambda> 0$ are the unique solutions of
\begin{equation*}
\begin{aligned}
    \radius^2&=
    \begin{aligned}[t]
    &\trace\left((\eye{}-A^{-1})^2(\Sref+\mref\transpose{\mref})\right)
    +2\innerProduct{(\eye{}-A^{-1})(\Sprior^{-1}+2\lambda\eye{})^{-1}(\Sprior^{-1}\mprior-A\Sref^{-1}\mref)}{\mref}
    \\ &+\norm{(\Sprior^{-1}+2\lambda\eye{})^{-1}(\Sprior^{-1}\mprior-A\Sref^{-1}\mref)}^2,
    \end{aligned}
    \\
    0&=-2\lambda A - A\Sref^{-1}A +  \Sprior^{-1}+2\lambda\eye{}.
\end{aligned}
\end{equation*}
Moreover, $\mu^\ast$ lies at the boundary of the Wasserstein ball (i.e., $\wassersteinDistance{2}{\mu^\ast}{\muref}=\radius$).
\end{proposition}

\ifbool{showproofs}{\booltrue{showproofoptimalsolutionkl}}{
\ifbool{arxiv}{\boolfalse{showproofoptimalsolutionkl}}{}} % not show in arxiv
\ifbool{showproofoptimalsolutionkl}{\proofoptimalsolutionkl{showproofoptimalsolutionkl}}{}

\section{Conclusions}\label{sec:conclusions}
We studied first-order conditions for optimization in the Wasserstein space. We combined the geometric and differential properties of the Wasserstein space with classical calculus of variations to formulate first-order necessary and sufficient conditions for optimality, showing that simple and interpretable rationales (e.g., ``set the gradient to zero'' and ``gradients are aligned at optimality'') carry over to the Wasserstein space. 
With our tools, we can study and solve, sometimes in closed form, optimization problems in distributionally robust optimization and statistical inference. 
We hope our results will pave the way for future research on optimization in the space of probability measures through the lens of optimal transport. 

\bibliographystyle{siamplain}
\bibliography{references}

\appendix
\section{Technical Preliminaries in Measure Theory}\label{app:preliminaries}

In this section, we provide some technical background in measure theory. 
We start with tightness: 
\begin{definition}[Tight]
A set of probability measures $\mathcal{K}\subset\Pp{}{\reals^d}$ is tight if for all $\varepsilon>0$ there exists a compact set $K_\varepsilon\subset\reals^d$ such that $\mu(\reals^d\setminus K_{\varepsilon})\leq\varepsilon$ for all $\mu\in\mathcal{K}$.
\end{definition}
A simple criterion for tightness is the following: 
\begin{lemma}[Criterion for tightness]\label{app:criterion tightness}
Let $\mathcal{K}\subset\Pp{}{\reals^d}$.
Suppose there exists $\Phi:\reals^d\to\nonnegativeReals$ with compact level sets (i.e., $\{x\in\reals^d:\Phi(x)\leq \lambda\}$ is compact for all $\lambda\in\reals$) so that 
\begin{equation*}
    \sup_{\mu\in\mathcal{K}}\int_{\reals^d}\Phi(x)\d\mu(x)<+\infty.
\end{equation*}
Then, $\mathcal{K}$ is tight. 
\end{lemma}
\begin{proof}
The statement can be found in~\cite[Remark 5.1.5]{Ambrosio2008a}; we prove it for completeness. Let $\varepsilon>0$, and let
\begin{equation*}
    C\coloneqq\sup_{\mu\in\mathcal{K}}\int_{\reals^d}\Phi(x)\d\mu(x)<+\infty.
\end{equation*}
Define $K_\varepsilon\coloneqq \{x\in\reals^d:\Phi(x)\leq \frac{C}{\varepsilon}\}$. Since $\Phi$ has compact level sets, $K_\varepsilon$ is compact. Then, for $\mu\in\mathcal{K}$
\begin{equation*}
\begin{aligned}
    \mu(\reals^d\setminus K_\varepsilon)
    &=
    \int_{\reals^d\setminus K_\varepsilon}\d\mu(x)
    \\
    &\leq 
    \int_{\reals^d\setminus K_\varepsilon}\frac{\Phi(x)}{C/\varepsilon}\d\mu(x)
    \\
    &\leq 
    \frac{\varepsilon}{C}\int_{\reals^d}\Phi(x)\d\mu(x)
    \\
    &\leq 
    \varepsilon.
\end{aligned}
\end{equation*}
Since $\mu$ is arbitrary, we conclude. 
\end{proof}
Tightness of a set is intimately related to its compactness (under narrow convergence): 
\begin{theorem}[Prokhorov~~\protect{\cite[Theorem 5.1.3]{Ambrosio2008a}}]\label{thm:prokhorov}
A set $\mathcal{K}\subset\Pp{}{\reals^d}$ is tight if and only if it is relatively compact for the narrow convergence. 
\end{theorem}
As a consequence of Prokhorov theorem, if $(\mu_n)_{n\in\naturals}\subset\Pp{}{\reals^d}$ is so that $\{\mu_n\}_{n\in\naturals}$ is tight, then $\mu_n$ admits a subsequence converging to some $\mu\in\Pp{}{\reals^d}$. 

Next, we focus on the narrow convergence of transport plans $\gamma_n\in\Pp{}{\reals^d\times\reals^d}$. In particular, we show convergence for the inner product $(x,y)\mapsto\innerProduct{x}{y}$. Since this is unbounded functional, the statement does not follow directly from the definition of narrow convergence, and requires adequate assumptions: 
\begin{proposition}[\protect{\cite[Lemma 5.2.4]{Ambrosio2008a}}]\label{prop:narrow convergence transport plan}
Let $(\gamma_n)_{n\in\naturals}\subset\Pp{}{\reals^d\times\reals^d}$ be narrowly converging to $\gamma\in\Pp{}{\reals^d\times\reals^d}$. Suppose that
\begin{equation*}
    \sup_n\int_{\reals^d\times\reals^d}\norm{x}^2+\norm{y}^2\d\gamma_n(x,y)<+\infty,
\end{equation*}
and that either $\pushforward{(\proj_1)}\gamma_n$ or $\pushforward{(\proj_2)}\gamma_n$ have uniformly integrable second moment. Then, 
\begin{equation*}
    \lim_{n\to\infty}\int_{\reals^d\times\reals^d}\innerProduct{x}{y}\d\gamma_n(x,y)
    =
    \int_{\reals^d\times\reals^d}\innerProduct{x}{y}\d\gamma(x,y).
\end{equation*}
\end{proposition}
Finally, we recall the celebrated Gluing Lemma: 
\begin{proposition}[Gluing lemma~\protect{\cite[Lemma 5.3.2]{Ambrosio2008a}}]\label{prop:gluing}
Let $\gamma_{12},\gamma_{13}\in\Pp{}{\reals^d\times\reals^d}$ so that $\pushforward{(\proj_1)}\gamma_{12}=\pushforward{(\proj_1)}\gamma_{13}=\mu_1\in\Pp{}{\reals^d}$.
Then, there exists $\gamma\in\Pp{}{\reals^d\times\reals^d\times\reals^d}$ so that
\begin{equation*}
\begin{aligned}
    \pushforward{(\proj_{12})}\gamma &=\gamma_{12},
    \\
    \pushforward{(\proj_{13})}\gamma &=\gamma_{13}.
\end{aligned}
\end{equation*}
Moreover, if $\gamma_{12}$ or $\gamma_{13}$ are induced by a transport, then $\gamma$ is unique. 
\end{proposition}

\section{Proofs}\label{app:proofs}

\notbool{showproofnecessaryconditions}{
%% Useful lemmas
\subsection{Useful Lemmas}
In this section, we state two lemmas that will be instrumental to prove~\cref{thm:first order conditions,thm:first order sufficient conditions constrained}:
% Lemma 1
\begin{lemma}[Fundamental lemma of calculus of variations, revisited]\label{lemma:variations}
Let $\mu\in\Pp{2}{\reals^d}$ and $H\in\tangentSpace{\mu}{\Pp{2}{\reals^d}}$. Assume that for all $h=\nabla\psi$ with $\psi\in\Ccinf{\reals^d}$
\begin{equation}\label{eq:lemma variations}
    \int_{\reals^d}\innerProduct{H(x)}{h(x)}\d\mu(x)=0.
\end{equation}
Then, $H\equiv 0$ in $\Lp{2}{\reals^d,\reals^d;\mu}$.
\end{lemma}
\prooflemmavariations{boolfalse}
% Lemma 2
\begin{lemma}\label{lemma:admissibility variation}
Let $\mu^\ast\in\Pp{2}{\reals^d}$ and $K:\Pp{2}{\reals^d}\to\reals$ be Wasserstein differentiable at $\mu^\ast$ so that $K(\mu^\ast)=0$, and let $T=\gradient{}\psi$ for some $\psi\in\Ccinf{\reals^d}$. Suppose that
\begin{enumerate}
\item there exists $H=\gradient{}\varphi$ with $\varphi\in\Ccinf{\reals^d}$ so that $\int_{\reals^d}\innerProduct{\gradient{\mu}K(\mu^\ast)(x)}{H(x)}\d\mu^\ast(x)\neq 0$; 

\item for all $\delta, \varepsilon\in\reals$ sufficiently close to 0, $K$ is Wasserstein differentiable at $\pushforward{(\Id+\varepsilon T+\delta H)}\mu^\ast$ and the real-valued functions
\begin{align*}
    (\varepsilon,\delta)
    &\mapsto \int_{\reals^d}\innerProduct{\gradient{\mu}K(\pushforward{(\Id+\varepsilon T+\delta H)}\mu^\ast)(x)}{T(x)}\d(\pushforward{(\Id+\varepsilon T+\delta H)}\mu^\ast)(x)
    \\
    (\varepsilon,\delta)
    &\mapsto \int_{\reals^d}\innerProduct{\gradient{\mu}K(\pushforward{(\Id+\varepsilon T+\delta H)}\mu^\ast)(x)}{H(x)}\d(\pushforward{(\Id+\varepsilon T+\delta H)}\mu^\ast)(x)
\end{align*}
are continuous.
\end{enumerate}
Then, there exists $\bar\varepsilon>0$ and a continuously differentiable function $\sigma\in\C{1}{}:(-\bar\varepsilon,\bar\varepsilon)\to\reals$ with $\sigma(0)=0$ such that for all $\varepsilon\in(-\bar\varepsilon,\bar\varepsilon)$
\begin{equation*}
    \constraint(\pushforward{(\Id+\varepsilon T+\sigma(\varepsilon)H)}\mu)=0.
\end{equation*}
\end{lemma}
\proofadmissiblevariation{boolfalse}}

%% Proofs
\ifbool{showproofs}{\subsection{Proofs of the Statements}}{}
\ifbool{showproofs}{\subsubsection{Proofs of the Statements of~\cref{sec:preliminaries}}}{}

% Background
\notbool{showproofinverseoptimaltransportmap}{\proofinverseoptimaltransportmap{boolfalse}}{}
\notbool{showproofperturbationtransportmap}{\proofperturbationtransportmap{boolfalse}}{}
\notbool{showproofgradientsactontangentvectors}{\proofgradientsactontangentvectors{boolfalse}}{}
\notbool{showproofgradientsarestrong}{\proofgradientsarestrong{boolfalse}}{}
\notbool{showproofdifferentiablecontinuous}{\proofdifferentiablecontinuous{boolfalse}}{}
\notbool{showproofgradientsgeodesicallyconvex}{\proofgradientsgeodesicallyconvex{boolfalse}}{}
\notbool{showproofsummultiplicationrule}{\proofsummultiplicationrule{boolfalse}}{}
\notbool{showproofchainrule}{\proofchainrule{boolfalse}}{}

% Functionals
\notbool{showproofoptimaltransportdiscrepancy}{\proofoptimaltransportdiscrepancy{boolfalse}}{}
\notbool{showproofwassersteindistance}{\proofwassersteindistance{boolfalse}}{}
\notbool{showproofexpectation}{\proofexpectation{boolfalse}}{}
\notbool{showproofinteraction}{\proofinteraction{boolfalse}}{}
\notbool{showproofvariance}{\proofvariance{boolfalse}}{}
\notbool{showproofinternalenergy}{\proofinternalenergy{boolfalse}}{}
\notbool{showproofdivergence}{\proofdivergence{boolfalse}}{}

\ifbool{showproofs}{\subsubsection{Proofs of the Statements of~\cref{sec:optimality}}}{}
% Optimality
%\prooflemmavariations{boolfalse}
\notbool{showproofnecessaryconditions}{\prooffirstorderconditions{boolfalse}}{}
\notbool{showprooffirstordersufficientconditions}{\prooffirstordersufficientconditions{boolfalse}}{}
%\proofadmissiblevariation{boolfalse}
\notbool{showprooffcontinuousdifferentiability}{\prooffcontinuousdifferentiability{boolfalse}}{}
\notbool{showproofnecessaryconditions}{\prooffirstordernecessaryconditionsconstrained{boolfalse}}{}
\notbool{showprooffirstordernecessaryconditionsconstrainedinequality}{\prooffirstordernecessaryconditionsconstrainedinequality{boolfalse}}{}
\notbool{showprooffirstordersufficientconditionsconstrained}{\prooffirstordersufficientconditionsconstrained{boolfalse}}{}
\notbool{showproofcorfirstordersufficientconditionsconstrained}{\proofcorfirstordersufficientconditionsconstrained{boolfalse}}{}
\notbool{showprooffirstordersufficientconditionsconstrainedwasserstein}{\prooffirstordersufficientconditionsconstrainedwasserstein{boolfalse}}{}
\notbool{showprooffirstordersufficientconditionsconstrainedwassersteincorollary}{\prooffirstordersufficientconditionsconstrainedwassersteincorollary{boolfalse}}{}

\ifbool{showproofs}{\subsection{Proofs of the Statements of~\cref{sec:dro}}}{}
% DRO
\notbool{showproofabsolutelycontinuityref}{\proofabsolutelycontinuityref{boolfalse}}{}
\notbool{showproofwassersteinballs}{\proofwassersteinballs{boolfalse}}{}
% \notbool{showprooflemmaminboundarylinearcost}{\prooflemmaminboundarylinearcost{boolfalse}}{}
\notbool{showproofdrolinearcost}{\proofdrolinearcost{boolfalse}}{}
%\proofcorworstcasecostlinearcost{boolfalse}
\notbool{showprooflemmapropertiesdrovariance}{\prooflemmapropertiesdrovariance{boolfalse}}{}
% \notbool{showprooflemmaminboundaryvariance}{\prooflemmaminboundaryvariance{boolfalse}}{}
\notbool{showproofdrovariance}{\proofdrovariance{boolfalse}}{}
\notbool{showprooflemmapropertiesdrostd}{\prooflemmapropertiesdrostd{boolfalse}}{}
% \notbool{showprooflemmaboundarydrostd}{\prooflemmaboundarydrostd{boolfalse}}{}
\notbool{showproofdrostd}{\proofdrostd{boolfalse}}{}

\ifbool{showproofs}{\subsection{Proofs of the Statements of~\cref{sec:kl}}}{}
% KL
\notbool{showprooflemmaexistencekl}{\prooflemmaexistencekl{boolfalse}}{}
\notbool{showproofoptimalsolutionkl}{\proofoptimalsolutionkl{boolfalse}}{}

\end{document}